\documentclass[11pt,oneside,a4paper]{article}

\usepackage[english]{babel}	
\usepackage[T1]{fontenc}		
\usepackage[utf8]{inputenc}	
\usepackage{lmodern}	
\usepackage{amsmath}
\usepackage{amsthm}	
\usepackage{amssymb}
\usepackage{mathrsfs}
\usepackage{paralist}	
\usepackage{geometry}			
\usepackage[all]{xy}
\usepackage[colorlinks = true,
            linktocpage = true,
            linkcolor = blue,
            urlcolor  = blue,
            citecolor = blue,
            anchorcolor = black]{hyperref}
            
\usepackage{titlesec}

\titleformat{\paragraph}[runin]{\normalsize}{\theparagraph}{3pt}{}{}
\titlespacing{\paragraph}{0pt}{15pt}{16pt}

\makeatletter
\renewcommand{\theparagraph}{(\arabic{section}.\arabic{subsection}.\arabic{paragraph})}
\makeatother

\numberwithin{equation}{section} 

\newtheoremstyle{bfnote}
	{}{}
	{\itshape}{}
	{\bfseries}{.}
	{ }
	{\thmname{#1}\thmnumber{ #2}\thmnote{ (#3)}}
	
\theoremstyle{bfnote}
\newtheorem{thm}{Theorem}[section]
\newtheorem{lem}[thm]{Lemma}
\newtheorem{prop}[thm]{Proposition}
\newtheorem{cor}[thm]{Corollary}

\theoremstyle{definition}

\theoremstyle{remark}
\newtheorem{rem}[thm]{Remark}
\newtheorem{ex}[thm]{Example}

\newcommand{\N}{\mathbb{N}}
\newcommand{\Z}{\mathbb{Z}}

\newcommand{\R}{\mathbb{R}}
\newcommand{\C}{\mathbb{C}}
\newcommand{\mE}{\mathscr E}
\newcommand{\mF}{\mathscr F}
\newcommand{\mH}{\mathscr H}
\newcommand{\mK}{\mathscr K}
\newcommand{\del}{\partial}

\DeclareMathOperator{\im}{Im}
\DeclareMathOperator{\re}{Re}
\DeclareMathOperator{\Vol}{Vol}
\DeclareMathOperator{\Id}{Id}

\DeclareMathOperator{\End}{End}
\DeclareMathOperator{\image}{im}
\DeclareMathOperator{\ind}{ind}
\DeclareMathOperator{\codim}{codim}

\date{}
\author{Thibault Langlais\footnote{Mathematical Institute, University of Oxford, Oxford OX2 6GG, United Kingdom. \newline E-mail address: langlais@maths.ox.ac.uk. ORCID iD: \href{https://orcid.org/0000-0002-6434-2988}{0000-0002-6434-2988}.} }
\title{Analysis and spectral theory of neck-stretching problems}

\begin{document}

\maketitle

\begin{abstract}
    \noindent We study the mapping properties of a large class of elliptic operators $P_T$ in gluing problems where two non-compact manifolds with asymptotically cylindrical geometry are glued along a neck of length $2T$. In the limit where $T \rightarrow \infty$, we reduce the question of constructing approximate solutions of $P_T u = f$ to a finite-dimensional linear system, and provide a geometric interpretation of the obstructions to solving this system. Under some assumptions on the real roots of the model operator $P_0$ on the cylinder, we construct a Fredholm inverse for $P_T$ with good control on the growth of its norm. As applications of our method, we study the decay rate and density of the low eigenvalues of the Laplacian acting on differential forms, and give improved estimates for compact $G_2$-manifolds constructed by twisted connected sum. We relate our results to the swampland distance conjectures in physics.
\end{abstract}

\small
\tableofcontents

\normalsize 

\section{Introduction and motivation}  \label{section:intro}

This paper is concerned with analytical aspects of the study of special geometric structures near degenerate limits. Such structures are defined as solutions of a non-linear PDE, and explicit examples can be quite challenging to find, especially in the absence of symmetries or in the compact setting. A standard way to produce such solutions is to use a gluing-perturbation method. This idea has been applied in various contexts for differential-geometric constructions, such as self-dual metrics \cite{donaldson1989connected,floer1991self}, hyperk\"ahler metrics on K3 surfaces \cite{lebrun1994kummer}, minimal surfaces \cite{kapouleas1990complete,kapouleas1991compact}, or compact manifolds with special holonomy \cite{joyce1996compacti,joyce1996compactii}, to only cite a few of them. In a typical gluing construction, the size of the gluing region is controlled by a parameter $T$, and a key step in the perturbation procedure is to understand the linearised equations for appropriately large values of $T$. Therefore, we shall be concerned with linear operators during most of the paper. We are interested in controlling the mapping properties of elliptic operators in the degenerate limit $T \rightarrow \infty$. For formally self-adjoint operators, this corresponds to understanding the low eigenvalues. 

In order to formulate these questions more precisely, we will restrict ourselves to constructions where the gluing region is modelled on a cylinder of length $2T$ and study the neck-stretching limit $T \rightarrow \infty$. We will especially focus on the case where the model operator on the cylinder has real roots, which substantially complicates the analysis. The aim is to develop a systematic method to deal with such cases. As an application we give a precise description of the lower spectrum of the Laplacian acting on differential forms in the neck-stretching limit. We find that the decay rate and the density of the low eigenvalues is determined by the topology of the cross-section of the gluing region.

The techniques developed in this paper are notably intended for the study of compact manifolds with holonomy $G_2$. The first examples were exhibited by Joyce \cite{joyce1996compacti,joyce1996compactii} using a gluing construction to resolve the singularities of a flat orbifold.  This construction was more recently extended by Joyce--Karigiannis in \cite{joyce2021new}. The only other known construction of compact manifolds with full holonomy $G_2$ is the twisted connected sum, first due to Kovalev \cite{kovalev2003twisted} and later improved by Corti--Haskins--Nordstr\"om--Pacini \cite{corti2013asymptotically,corti2015g} and further extended by Nordstr\"om \cite{nordstrom2018extra}. The twisted connected sum fits into our general setup, and we give improved estimates for this construction. We use these improved estimates to deduce the decay properties of the low eigenvalues of the Laplacian on twisted connected sums and understand their geometric origin.

\paragraph*{} Our results can be related to the swampland conjectures in physics \cite{vafa2005string} (see also \cite{brennan2017string} and \cite{palti2019swampland} for detailed reviews). Indeed, compact manifolds with special holonomy play an important role for compactifications in quantum gravity theories (e.g. Calabi--Yau compactifications in string theory or $G_2$ compactifications in M-theory). The geometry of the internal manifold then governs the low-energy physics in the remaining non-compact dimensions. Typically, the number of massless fields can be deduced from the Betti numbers of the internal manifold, and the masses of the so-called Kaluza--Klein modes are determined by the eigenvalues of the Laplacian acting on differential forms. Therefore, deforming the geometry of the internal manifold amounts to varying the parameters of the low-energy effective field theory. 

Regardless of a particular choice of internal geometry, the swampland distance conjectures \cite{ooguri2007geometry} concern universal features that are expected to hold for any moduli space of effective field theories that admit a consistent quantum gravity completion. It is in particular conjectured that infinite distance limits in such moduli spaces correlate with the appearance of an infinite tower of light states in the effective field theory. Therefore, it is an interesting problem to study the moduli dependence of the spectrum of the Laplacian on manifolds with special holonomy and understand the decay of the low eigenvalues near singular limits. Such a problem has for instance been studied in \cite{ashmore2021moduli} for quintic Calabi--Yau threefolds, using numerical methods. In the present paper, we give an analytical point of view on the study of the towers of light states in the context of $G_2$ geometry. In M-theory compactified on $G_2$-manifolds, the eigenvalues of the Laplacian acting on co-closed $q$-forms for $q \leq 3$ correspond to the squared masses of Kaluza--Klein states. Our analysis provides precise estimates for the decay rate and the distribution of the low eigenvalues of the Laplacian on twisted connected sums and explains how they arise. We hope that this could be of interest for the study of the distance conjectures and shed some light on the geometric origin of the towers of light states, even if in a very specific example.

\paragraph*{\textbf{Organisation of the paper.}} In Section \ref{section:setup} we describe the general gluing problem that we are interested in, set our notations, and define the notions of adapted operators and of substitute kernel and cokernel. Our main results are stated in \S\ref{subsection:results}. Section \ref{section:pdecyl} is concerned with the analysis of translation-invariant PDEs on cylinders and contains most of the technical ingredients underlying our results. Our exposition is self-contained, and for this purpose we include a brief review of standard results in \S\ref{subsection:sepvar}. Section \ref{section:construction} is dedicated to the analysis of the mapping properties of adapted operators in the neck-stretching limit. Under some assumptions we prove a theorem on the invertibility of adapted operators (Theorem \ref{thm:main}), but our method is more general and we also comment on how to adapt it in different contexts. In the subsequent sections we apply our techniques to the study of the low eigenvalues of the Laplacian (Section \ref{section:spectral}) and to the twisted connected sum construction of compact $G_2$-manifolds (Section \ref{section:compact}).

\paragraph*{\textbf{Acknowledgements.}} I would like to thank Tristan Ozuch for helpful discussions which clarified some points in the analysis presented in this paper. I am also grateful to my supervisor Jason Lotay for his support and advice. Research supported by scholarships from the Clarendon Fund and the Saven European Programme.

\section{Setup and discussion of results} 	\label{section:setup}

In this section we explain our setup and formulate the main results of this paper. The gluing problem under consideration is described in \S\ref{subsection:gluingpb}, in which we introduce the building blocks and the class of adapted operators that we are interested in. In \S\ref{subsection:substk} we motivate and introduce the notions of substitute kernel and cokernel for adapted operators, following ideas present for instance in \cite{kapouleas1990complete,kapouleas2004constructions} or \cite{kovalev2000gluing}. Our main results are discussed in \S\ref{subsection:results} and related to the swampland distance conjectures in \ref{par:relsdc}. Last, we give an overview of our strategy in \ref{par:strategy}.

    \subsection{Model gluing problem}		\label{subsection:gluingpb}

\paragraph{} Before describing the type of gluing problems we are interested in, we need to recall a few standard definitions, starting with the notion of \emph{Exponentially Asymptotically Cylindrical (EAC)} manifold. Let $Z$ be an oriented non-compact manifold of dimension $n$ and $X$ an oriented compact manifold of dimension $n-1$. We say that $Z$ is asymptotic to the cylinder $Y = \R \times X$ at infinity if there exist a compact $K \subset Z$ and an orientation-preserving diffeomorphism $\phi : (0,\infty) \times X \rightarrow Z \backslash K$. The compact manifold $X$ is called the \emph{cross-section} of $Z$. It will often be useful to pick a positive function $\rho : Z \rightarrow \R$ such that $\rho(\phi(t,x)) = t$ for $(t,x) \in [1, \infty) \times X$ and $\rho < 1$ outside of $\phi([1,\infty) \times X)$. Following the terminology of \cite{haskins2015asymptotically}, we will call such function a \emph{cylindrical coordinate function}.

We say that a Riemannian metric $g$ on $Z$ is EAC of rate $\mu > 0$  if we have, for all integers $l \geq 0$:
\begin{equation} 
    | \nabla^l_Y (\phi^* g - g_Y) |_{g_Y} = O \left( e^{-\mu t} \right)
\end{equation}
as $t \rightarrow \infty$, where $g_Y = dt^2 + g_X$ is a cylindrical metric on $Y = \R \times X$, $\nabla_Y$ the associated Levi-Civita connection and $| \cdot |_{g_Y}$ the associated norm on tensor bundles.

Given the above data, we may define a notion of \emph{adapted bundle} as follows. Any vector bundle $E_0 \rightarrow X$ equipped with a metric $h_0$ and a connection $\nabla_0$ can be extended to a vector bundle $\underline E_0 \rightarrow Y$ with translation-invariant metric and connection $(\underline h_0, \underline \nabla_0)$ (see Section \ref{section:pdecyl} for more details). We call such bundles on $Y = \R \times X$ translation-invariant bundles. Let $E \rightarrow Z$ be a vector bundle on $Z$, endowed with a metric $h$ and a connection $\nabla$. We say that $E$ is an adapted bundle on $(Z,g)$ if there exist a translation-invariant vector bundle $(\underline E_0,\underline h_0, \underline \nabla_0)$ and a bundle isomorphism $\Phi_E : {\underline E_0}_{|(0,\times X)} \rightarrow E_{| Z \backslash K}$ covering $\phi$, such that for all integers $l \geq 0$:
\begin{equation}
    | \underline \nabla_0^l (\Phi_E^* h - \underline h_0)|_0 = O(e^{-\mu t}), ~~ \text{and} ~~ | \underline \nabla_0^l(\Phi_E^* \nabla - \underline \nabla_0) |_0 = O \left( e^{- \mu t} \right)
\end{equation}
as $t \rightarrow \infty$.

\begin{rem}
    This definition is valid for both real and complex vector bundles. As we will make heavy use of the Fourier transform in Section \ref{section:pdecyl}, we usually assume that we are dealing with complex vector bundles and that the bundle metric is hermitian. Our result will also apply to real vector bundles by considering their complexification.
\end{rem}

Last, we may also define the notion of \emph{adapted differential operator} between adapted bundles. Let $E$, $F$ be adapted bundles on $Z$ and $P : C^\infty(E) \rightarrow C^\infty(F)$ be a differential operator of order $k \geq 1$. If $u$ is a smooth section of $\underline E_0$ defined over the half-cylinder $(0,\infty)$, let:
\begin{equation}
    \widetilde P u = \Phi_F^{-1} P \Phi_E u.
\end{equation}
This defines a differential operator $\widetilde P : C^\infty(\underline E_0) \rightarrow C^\infty (\underline F_0)$ over the cylinder $(0,\infty) \times X$, modelling the action of $P$ on sections supported in $Z \backslash K$. The operator $\widetilde P$ can be written in the form:
\begin{equation}
    \widetilde P = \sum_{j=0}^k A_{k-j}(t)\del_t^j
\end{equation}
where $\del_t$ is the covariant derivative in the direction $\frac{\del}{\del t}$ for the connection $\underline \nabla_0$ on $E$ (which coincides with the Lie derivative of $\frac{\del}{\del t}$ because of translation-invariance), and $A_{k-j}(t) : C^\infty(E_0) \rightarrow C^\infty(F_0)$ are differential operators depending smoothly on $t.$ We say that $P$ is adapted (with exponential rate $\mu > 0$) if there exists a translation-invariant differential operator $P_0 : C^\infty (E_0) \rightarrow C^\infty(F_0)$ of the form:
\begin{equation}
    P_0 = \sum_{j=0}^k A_{k-j}\del_t^j
\end{equation}
such that for any smooth section $u$ of $\underline E_0$ defined on $(0,\infty) \times X$ and for any $l\geq 0$ and $0 \leq j \leq k$ we have:
\begin{equation}
    | \underline \nabla_0^l (A_{k-j}(t)u - A_{k-j}u ) |_0 = O \left( e^{-\mu t} \sum_{i \leq l} |\underline \nabla_0^i u|_0 \right)
\end{equation}
as $t \rightarrow \infty$. That is, we essentially want the coefficients of $\widetilde P - P_0$ and all their derivatives to have exponential decay when $t \rightarrow \infty$. The operator $P_0$ is called the \emph{indicial operator} of $P$. Note that the formal adjoint of an adapted $P$ is also adapted, and its indicial operator is naturally $P_0^*$.

\begin{ex}
    On an EAC manifold $(Z,g)$ with cross-section $X$, the bundles $TZ$, $TZ^{\otimes l}$ and $\Lambda^q T^* Z$ are adapted, endowed with the Levi-Civita connection and the metric induced by $g$. The operators $d+d^*$ and $\Delta = dd^* + d^*d$ are adapted.
\end{ex}

\paragraph{} We now describe the general gluing problem we are interested in. Let $Z_1$ and $Z_2$ be two EAC manifolds and assume that the cross-section of $Z_2$ is the same as the cross-section $X$ of $Z_1$, but with opposite orientation. By definition, there exist compacts $K_i \subset Z_i$ and diffeomorphisms $\phi_i : (0,\infty) \times X_i \rightarrow Z_i \backslash K_i$ where $X_1 = X = \overline X_2$, and we can pick cylindrical coordinate functions $ \rho_i : Z_i \rightarrow \R_{>0}$. For $T \geq 0$ we can construct an oriented compact manifold $M_T$ by gluing the domains $\{ \rho_1 \leq T+2 \} \subset Z_1$ and $\{ \rho_2 \leq T+2 \} \subset Z_2$ along the annuli $\{ T \leq \rho_i \leq T+2 \} \simeq [-1,1] \times X$ with the identification:
\begin{equation}        \label{eq:gluingid}
    \phi_1(T+1+t,x) \simeq \phi_2(T+1-t,x), ~~ \forall (t,x) \in [-1,1] \times X.
\end{equation}
Define a smooth function $\rho_T$ on $M_T$ by:
\begin{equation*}
    \rho_T \equiv
    \begin{cases}
        \rho_1 - T - 1  & \text{in} ~~ \{ \rho_1 \leq T+2 \}\\
        T+ 1 - \rho_2   & \text{in} ~~ \{ \rho_2 \leq T+2 \}
    \end{cases}.
\end{equation*}
This is well-defined as $\rho_1 - T - 1$ coincides with $T+ 1 - \rho_2$ under the identification \eqref{eq:gluingid}. Intuitively, the function $\rho_T$ parametrises the neck of $M_T$. In particular, the domain $\{ | \rho | \leq T \}$ is diffeomorphic to the finite cylinder $[-T, T] \times X$. Our goal is to study the mapping properties of elliptic operators defined on $M_T$ as $T$ becomes very large, and relate it to the corresponding properties of operators on $Z_i$.

To that end, suppose that the manifolds $Z_i$ are endowed with EAC metrics $g_i$ asymptotic to the same translation-invariant metric $g_Y = dt^2 + g_0$ on $Y = \R \times X$. It will also be useful to fix a cutoff function $\chi : \R \rightarrow [0,1]$ such that $\chi \equiv 0$ on $(-\infty, - \frac 1 2]$ and $\chi \equiv 1$ on $[\frac 1 2, + \infty)$. If $T \in \R$ we let $\chi_T(t) = \chi(t-T)$. Then, for $T$ large enough
\begin{equation*}
    g_{i,T} = (1-\chi_T(\rho_i)) g_i + \chi_T(\rho_i) g_Y
\end{equation*}
is a Riemannian metric on $Z_i$ which coincides with $g_i$ on $\{ \rho_i \leq T- \frac 1 2 \}$ and with $g_Y$ on $\{ \rho_i \geq T + \frac 1 2 \}$. Moreover, the difference $g_i - g_{i,T}$ and all their derivatives are uniformly bounded by $O(e^{- \mu T})$. Note that here we implicitly identify $Z_i \backslash K_i$ with the half cylinder $(0,\infty) \times X_i$ to make notations lighter. We can patch $g_{1,T}$ and $g_{2,T}$ to form a Riemannian metric $g_T$ on $M_T$, defining:
\begin{equation*}
    g_T \equiv
    \begin{cases}
        g_{1,T}  & \text{if} ~~ \rho_T \leq 0 \\
        g_{2,T}  & \text{if} ~~ \rho_T \geq 0
    \end{cases} .
\end{equation*}
Similarly, if we have adapted bundles $(E_i, h_i, \nabla_i)$ on $Z_i$ such that their asymptotic models are both isomorphic to the same translation-invariant vector bundle $(\underline E_0,\underline h_0,\underline \nabla_0)$ on $\R \times X$, we can use the same cutoff procedure to patch them up on $M_T$ and form a vector bundle $E_T$ equipped with a metric $h_T$ and a connection $\nabla_T$.

\paragraph{} Consider matching adapted bundles $E_i, F_i$ on $Z_i$ ($i = 1,2$) asymptotic to the same translation-invariant bundles $\underline E_0, \underline F_0$, and adapted elliptic operators $P_i : C^{\infty}(E_i) \rightarrow C^\infty(F_i)$ of order $k$. Denote by $P_{i,0}(x,\del_x,\del_t) : C^\infty(\underline E_0) \rightarrow C^\infty(\underline F_0)$ the indicial operator of $P_i$, where we use $\del_x$ as a loose notation for the derivatives along the cross-section $X$. In order to patch up these operators we need to assume the following compatibility condition \cite{kovalev2000gluing}:
\begin{equation}		\label{eq:compp}
    P_{2,0}(x,\del_x,\del_t) = P_{1,0}(x,\del_x,-\del_t).
\end{equation}
Assuming that it is satisfied, define:
\begin{equation*}
    P_{i,T} = (1-\chi_T(\rho_i)) P_i + \chi_T(\rho_i) P_{i,0}
\end{equation*}
which coincides with $P_i$ for $\rho_i \leq T-\frac 1 2$ and with $P_{i,0}$ for $\rho_i \geq T + \frac 1 2$. For large enough $T$, the operators $P_{i,T}$ are elliptic, and moreover the coefficients of $P_i-P_{i,T}$ and all their derivatives are uniformly bounded by $O(e^{-\mu T})$. Patching $P_{1,T}$ and $P_{2,T}$ as above, we obtain a family of operators $P_T : C^\infty(E_T) \rightarrow C^\infty(F_T)$ which are elliptic for large enough $T$. Elliptic regularity on compact manifolds implies that the action of $P_T$ on Sobolev spaces of sections induce Fredholm maps. Our goal is to construct Fredholm inverses for these maps, with a good control on their norm as $T \rightarrow \infty$. 

\begin{rem}     \label{rem:twist}
    In our applications, we will also be interested in a variation of the above gluing construction where the EAC manifolds $Z_1$ and $Z_2$ are glued along a non-trivial isometry $\gamma : X \rightarrow X$. This corresponds to replacing the identification \eqref{eq:gluingid} by:
    \begin{equation}
        \phi_1(T+1+t,x) \simeq \phi_2(T+1-t,\gamma(x)), ~~ \forall (t,x) \in [-1,1] \times X.
    \end{equation}
    All the matching conditions have to be changed accordingly, but otherwise everything we will do applies without modification, up to a mere change of notations. This will be important in particular for the construction of compact $G_2$-manifolds by twisted connected sums that we study in \S\ref{subsection:tcs}.
\end{rem}

\paragraph{} Before explaining our results in more detail in the next part, let us make our conventions for Sobolev and $C^l$ norms explicit. For $p \geq 1$ and $l \in \N$, the $W^{l,p}$-norm of a section $u \in C^\infty(E_T)$ can be defined as:
\begin{equation}
    \| u \|_{W^{l,p}} = \sum_{ j \leq l } \| \nabla^j_T u \|_{L^p}
\end{equation}
where the fibrewise norm of $\nabla^j_T u$ is computed with the metrics $h_T$, $g_T$ and we integrate over the volume form of $g_T$. The Sobolev space $W^{l,p}(E_T)$ is defined as the completion of $C^\infty(E_T)$ for the $W^{l,p}$-norm. 

Moreover, the $C^l$ norm of a smooth section $u \in C^\infty(E_T)$ will be defined as 
\begin{equation}
    \| u \|_{C^l} = \sum_{j \leq l} \| \nabla^j_T u \|_{C^0}
\end{equation}
where the fibrewise norm of $\nabla^j_T u$ is as before relative to the metrics $h_T$ and $g_T$.

With these conventions, one can deduce from standard results about analysis on EAC manifolds (see \S\ref{subsection:lmt}) the following uniform a priori estimates:

\begin{prop}  \label{prop:unifapriori}
    With the above setup, let $p > 1$ and $l \in \N$. Then the map
      $$ P_T : W^{k+l,p}(E_T) \rightarrow W^{l,p}(F_T) $$
    is uniformly bounded as $T \rightarrow \infty$. Moreover there exist constants $C, C' > 0$ such that for $T$ large enough and for any $u \in W^{k+l,p}(E_T)$ we have:
      $$ \| u \|_{W^{k+l,p}} \leq C \left( \| P_T u \|_{W^{l,p}} + \| u \|_{L^p} \right) .$$
\end{prop}

\begin{rem}     \label{rem:sobemb}
    In the same way, it follows from the embedding theorems of \cite{lockhart1987fredholm} for weighted spaces on EAC manifolds that the constants in the Sobolev embeddings $W^{r,p} \hookrightarrow W^{s,q}$ and $W^{r,p} \hookrightarrow C^l$, in the range where such embeddings exist and are continuous, are in fact uniformly bounded as $T \rightarrow \infty$. This will be useful in Sections \ref{section:spectral} and \ref{section:compact}.
\end{rem}

By elliptic regularity, the above maps can be inverted in the $L^2$-orthogonal complements of the kernels of $P_T$ and of its adjoint. However, since the dimensions of these spaces are not continuously invariant (only the index is), they do depend on the precise way we take cutoffs to define our gluing, and so will the norm of the inverse. In order to make general statements, we would like to define notions of substitute kernel and cokernel in the fashion of \cite{kapouleas1990complete} (see also \cite[\S18]{kapouleas2004constructions}), determined by the gluing data and in the complement of which we have a good control on the norm of the inverse of $P_T$. Under the restricting assumption that the indicial operator $P_0 = P_{1,0}$ is invertible, these have been defined and studied in \cite{kovalev2000gluing}. However in many cases of interest this is not satisfied, as the indicial operator may have real roots. In the next section we will define notions of substitute kernel and cokernel adapted to the case where there are real roots.

  \subsection{Substitute kernel and cokernel}   \label{subsection:substk}

\paragraph{} In order to define the substitute kernel and cokernel, a good understanding of the mapping properties of translation-invariant operators on cylinders and of adapted operators on EAC manifolds is crucial. For completeness, the results that we need are gathered in \S\ref{subsection:sepvar} and \S\ref{subsection:lmt}. Original references are \cite{agmon1961properties} and \cite{lockhart1985elliptic}, as well as \cite{melrose1993atiyah} for a detailed discussion of analysis in the broader context of $b$-calculus.

In the situation described in the previous part, let $P_0 = P_{1,0} : C^\infty(\underline E_0) \rightarrow C^\infty(\underline F_0)$ be the indicial operator of $P_1$, acting on the cylinder $Y = \R \times X$. Points in $Y$ will be denoted by $y = (t,x)$. A particularly important role in our analysis is played by solutions of the homogeneous equation $P_0 u = 0$ of the form:
\begin{equation*}
    u(t,x) = \sum_{j=1}^m e^{i \lambda_j t}u_j(t,x)
\end{equation*}
where $\lambda_1,...,\lambda_m$ are real numbers and the sections $u_j$ are polynomial in the variable $t$. Such solutions are called \emph{polyhomogeneous solutions} of rate $0$, and we denote by $\mE$ the vector space they span. As a matter of general theory, this is a finite-dimensional space, and in particular there are only finitely many values $\lambda \in \R$ such that the homogeneous equation $P_0 u = 0$ admits a non-trivial solution of the form $u(t,x) = e^{i \lambda t} u_\lambda(t,x)$, where $u_\lambda$ is polynomial in $t$. These values are called the \emph{real roots} of $P_0$ (see Section \ref{section:pdecyl} for a detailed discussion). In Section \ref{section:pdecyl} we will see that each root $\lambda_j$ has a certain order $d_j \in \N^*$ such that the sections $u_j$ in a polyhomogeneous solution of rate $0$ are polynomials of order at most $d_j -1$ in the variable $t$. We will usually denote by $d$ the maximum of the orders $d_j$. Let us point out here that the real roots of the formal adjoint $P^*_0$ are the same as the real roots of $P_0$, and denote by $\mE^*$ the space of polyhomogeneous solutions of rate $0$ of $P^*_0 u = 0$. It follows from the compatibility condition \eqref{eq:compp} that the space of polyhomogeneous solutions of $P_{2,0} u = 0$ of rate $0$ is $\{u(-t,x), u \in \mE \}$, and similarly for the adjoint operators.

Let us denote by $\mK_i$ the space of solutions of $P_i u = 0$ with sub-exponential growth and $\mK_{i,0}$ the subspace of decaying solutions, for $i=1,2$. By Lockhart--McOwen theory (\cite{lockhart1985elliptic}, see also \S\ref{subsection:lmt} for more details), $\mK_i$ has finite dimension and each of its elements is asymptotic to a polyhomogeneous solution of $P_{i,0} u = 0$ with rate $0$, up to exponentially decaying terms. More precisely, for any $u \in \mK_1$, there exists a polyhomogeneous solution $u_0 \in \mE$ such that for any $l \in \N$:
\begin{equation*}
    |\underline \nabla_0^l (u(t,x) - u_0(t,x))|_0 = O \left( e^{- \delta t} \right)
\end{equation*} 
when $t \rightarrow \infty$, for any sufficiently small $\delta > 0$. Here we implicitly identify $u$ over $Z_1 \backslash K_1$ with a section of $\underline E_0$ over $(0,\infty) \times X$. Therefore, we can define a linear map $\kappa_1 : \mK_1 \rightarrow \mE$, such that for any $u \in \mK_1$, the difference $u-\kappa_1[u]$ and all its derivatives have exponential decay at infinity. Taking care of the fact that we need to change the sign of the variable $t$, we can similarly define a map $\kappa_2 : \mK_2 \rightarrow \mE$ such that $| u(x, t) - \kappa_2[u](x,-t)|_0 = O(e^{- \delta t })$ as $t \rightarrow \infty$ for all $u \in \mK_2$, with the usual identifications. For $i=1,2$, the kernel of the map $\kappa_i$ in $\mK_i$ is $\mK_{i,0}$. Considering adjoint operators, we may also define $\mK^*_i$, $\mK^*_{i,0}$ and linear maps $\kappa_i^* : \mK_i^* \rightarrow \mE^*$.

\paragraph{} With these notations in hand, let $u_1 \in \mK_1$, $u_2 \in \mK_2$ and fix $T >0$. We say that $u_1$ and $u_2$ are matching at $T$ if the following condition is satisfied:
\begin{equation}		\label{eq:matchingu}
    \kappa_1[u_1](t+T+1,x) = \kappa_2[u_2](t-T-1,x), ~~ \forall (t,x) \in \R \times X.
\end{equation}
Given a matching pair of solutions $(u_1,u_2)$, we can define a section of the bundle $E_T \rightarrow M_T$ as follows:
\begin{equation*}
    u_T = (1-\chi_{T+1}(\rho_1)) u_1 + (1-\chi_{T+1}(\rho_2)) u_2
\end{equation*}
where we consider $\chi(\rho_i)u_i$ as a section of $E_T$ supported in the domain $\{ \rho_i \leq T+2 \} \subset M_T$. In particular, $u_T \equiv u_1$ in the domain $\{ \rho_1 \leq T + \frac 1 2 \}$, $u_T \equiv u_2$ in $\{ \rho_2 \leq T + \frac 1 2 \}$ and it smoothly interpolates between the two in $\{ |\rho_T | \leq \frac 1 2 \}$. It is easy to see that $P_T u_T \equiv 0$ outside of the annulus $\{ |\rho_T | \leq \frac 3 2 \}$. Moreover, the matching condition \eqref{eq:matchingu} ensures that for any $l \in \N$, small enough $\delta > 0$ and arbitrary norms on $\mK_1$ and $\mK_2$, there exists a constant $C>0$ independent of $T \geq 1$ such that:
\begin{equation}        \label{eq:badcst}
    \|P_T u_T \|_{C^l} \leq C e^{- \delta T}(\|u_1\| + \| u_2\|)
\end{equation}
for any matching pair of solutions $(u_1,u_2)$. In that sense, $u_T$ is an approximate solution of $P_T u = 0$. The \emph{substitute kernel} $\mK_T$ of $P_T$ is defined as the finite-dimensional subspace of $C^\infty(E_T)$ of approximate solutions constructed in this way, from a matching pair $(u_1,u_2)$ of solutions of $P_i u = 0$. Similarly we define the \emph{substitute cokernel} $\mK^*_T$ as the substitute kernel of $P^*_T$. These definitions depend on the arbitrary choice of cutoff function $\chi$, but since the difference for two different choices of cutoff function would decay exponentially with $T$ this will not be an issue.

\paragraph{} \label{par:ranksystem} For these notions of substitute kernel and cokernel to be convenient to handle in practice, it is simpler to assume that for $T$ large enough the dimensions of $\mK_T$ and $\mK^*_T$ are independent of $T$. This is automatically satisfied if the indicial operator $P_0$ has only one root. Indeed, in this case we can express the matching condition at $T$ as a finite-dimensional linear system depending polynomially on $T$ by choosing convenient bases for $\image \kappa_1$, $\image \kappa_2$ and $\mE$. The minors of this system are polynomial in $T$, and therefore are either identically $0$ or do not vanish for $T$ large enough. Hence the rank of the system do not depend on $T$ for $T$ large enough, and neither does the dimension of its kernel. We can argue similarly for the substitute cokernel $\mK^*_T$.
  
For more general operators, the matching condition will be expressed as a finite-dimensional linear system with coefficients depending analytically on $T$, and although the non-trivial minors of the system only have isolated zeroes we cannot ensure that there are only finitely many of them. This is the situation that we want to avoid. Therefore we will assume that $P$ has only one real root to state our main result, about the existence of a Fredholm inverse for $P$ in the complement of the substitute kernel and cokernel. This is sufficient for our applications in Sections \ref{section:spectral} and \ref{section:compact}. However the method we develop is more general, and for most of the paper we do not need to take any restricting assumptions on the roots of the indicial operator.
  
Assuming that the spaces $\mK_T$ and $\mK^*_T$ have constant dimension for $T$ large enough, it follows from \eqref{eq:badcst} that for any Sobolev norm $W^{k,p}$ and any small enough $\delta > 0$, there exists a constant $C>0$ such that for $T$ large enough and for any $u \in \mK_T$:
\begin{equation*}
    \| P_T u \|_{W^{k,p}} \leq C e^{- \delta T} \| u \|_{L^p}.
\end{equation*}
Similar bounds hold for $P^*_T$. Hence there is no hope to have a control on the norm of the inverse of $P_T$ better than $O(e^{\delta T})$ if we do not work on the complement of $\mK_T$ and $\mK^*_T$.

    \subsection{Results and strategy}       \label{subsection:results}
  
\paragraph{} We keep the setup and notations of the previous part. Our main result is the following theorem, which says that under the limiting assumption described above we can find a Fredholm inverse for $P_T$ in the complement of the substitute kernel and cokernel, with norm bounded by a power of $T$:

\begin{thm}  \label{thm:main}
    Let $p > 1$ and $l \in \N$, and assume that $P_0$ has only one real root. Then there exist constants $C, C^\prime > 0$ and an exponent $\beta \geq  0$ such that for $T$ large enough the following holds. 
    
    For any $f \in W^{l,p}(F_T)$, there exist a unique $u \in W^{k+l,p}(E_T)$ orthogonal to $\mK_T$ and a unique $w \in \mK^*_T$ such that $f = P_T u + w$. Moreover, $u$ satisfies the bound:
    \begin{equation*}
        \| u \|_{W^{k+l,p}} \leq C \| f \|_{W^{l,p}}+ C^\prime T^\beta \|f\|_{L^p}.
    \end{equation*}
\end{thm}
  
\paragraph{} \label{par:spectral} In some cases, we are also able to determine the optimal exponent $\beta$. This is for instance the case of the Laplacian operator $\Delta_T$ of the metric $g_T$. The Laplacian acting on $q$-forms on $\R \times X$ has no real roots when $b^{q-1}(X) + b^q(X) = 0$ and admits $0$ as unique real root when $b^{q-1}(X) + b^q(X) > 0$. If $b^{q-1}(X) + b^q(X) = 0$, it follows from the results of \cite{kovalev2000gluing} that the norm of the inverse of $\Delta_T$ orthogonally to the space of harmonic $q$-forms is bounded independently of $T$. When $b^{q-1}(X) + b^q(X) > 0$, we will see that $\beta = 2$ is optimal in Theorem \ref{thm:density} and the substitute kernel gives a good approximation of the space of harmonic forms (see Corollary \ref{cor:laplacian}). 

If we consider the $L^2$-range, this means that the behaviour of the low eigenvalues of $\Delta_T$ depends on the topology of the cross-section $X$. When $b^{q-1}(X) + b^q(X) = 0$ the lowest nonzero eigenvalue of $\Delta_T$ acting on $q$-forms is uniformly bounded below as $T \rightarrow \infty$. On the other hand, if $b^{q-1}(X) + b^q(X) > 0$ then the first eigenvalue satisfies a bound of the type $\lambda_1(T) \geq \frac{C}{T^2}$ for some constant $C > 0$. It is an interesting problem to determine the distribution the eigenvalues that have the fastest decay rate. Let us define the densities of low eigenvalues as:
\begin{align*}      
    \Lambda_{q, \inf}(s) & = \liminf_{T \rightarrow \infty} \# \left \{ \text{eigenvalues of} ~ \Delta_T ~\text{acting on} ~ q \text{-forms in} ~ \left (0, \pi^2 sT^{-2} \right] \right\} \\
    \Lambda_{q, \sup}(s) & = \limsup_{T \rightarrow \infty} \# \left \{ \text{eigenvalues of} ~ \Delta_T ~\text{acting on} ~ q \text{-forms in} ~ \left (0, \pi^2 sT^{-2} \right] \right\}
\end{align*}
where we count eigenvalues with multiplicity. The normalisation by $T^{-2}$ comes from the fact that we expect the lowest eigenvalues to be decaying at precisely this rate, whilst the factor $\pi^2$ is just a matter of convenience. We can similarly define the densities $\Lambda^*_{q,\inf}(s)$ and $\Lambda^*_{q,\sup}(s)$ of low eigenvalues of the Laplacian acting on co-exact $q$-forms.  We are interested in understanding the asymptotic behavior of these densities as $s \rightarrow \infty$. In \S\ref{subsection:spectrum} we prove the following:

\begin{thm}     \label{thm:density}
    If $b^{q-1}(X) + b^q(X) > 0$, then the densities of low eigenvalues satisfy:
    \begin{equation*}
       \Lambda_{q, \sup}(s) =  \Lambda_{q, \inf}(s) + O(1) =  2 (b^{q-1}(X) + b^q(X)) \sqrt{s} + O(1) .
    \end{equation*}
    as $s \rightarrow \infty$. If moreover $b^q(X) > 0$ then:
    \begin{equation*}
        \Lambda^*_{q, \sup}(s) = \Lambda^*_{q, \inf}(s) + O(1) =  2 b^q(X) \sqrt{s} + O(1).
    \end{equation*}
\end{thm}

This theorem essentially says that the lowest eigenvalues of $\Delta_T$ are asymptotically distributed as the low eigenvalues of the Laplacian acting on the product $S^1_{2T} \times X$, where the first factor is a circle of length $2T$. We shall moreover see in \ref{par:stronger} that stronger statements relating the lower spectrum of $\Delta_T$ to the spectrum of the Laplacian on $S^1_{2T} \times X$ would not hold. The reason for this is that the interaction between the building blocks of the construction creates a shift in the spectrum of $\Delta_T$, compared with the spectrum of the Laplacian on the product. We shall also briefly discuss possible generalisations of the above statement to other self-adjoint operators.

\paragraph{} As a second application of Theorem \ref{thm:main} and Corollary \ref{cor:laplacian}, we shall give improved estimates for compact $G_2$-manifolds constructed by twisted connected sum in Section \ref{section:compact}, thereby addressing some analytical issues in the literature. The building blocks of the twisted connected sum construction are a pair of EAC $G_2$ manifolds $(Z_1,\varphi_1)$ and $(Z_2,\varphi_2)$ with rate $\mu > 0$, satisfying a certain matching condition (see \cite{kovalev2003twisted, corti2015g}). These building blocks can be glued along a non-trivial isometry $\gamma$ of the cross-section $X$ (see Remark \ref{rem:twist}) to form a $1$-parameter family of compact manifolds $M_T$, equipped with closed $G_2$-structures $\varphi_T$ with exponentially decaying torsion (see \S\ref{subsection:tcs} for more details). For large enough $T$, $\varphi_T$ can be deformed to a nearby torsion-free $G_2$-structure $\widetilde \varphi_T$ with $\| \widetilde \varphi_T - \varphi_T \|_{C^0} = O(e^{-\delta T})$ \cite[Theorem 5.34]{kovalev2003twisted}. Using Theorem \ref{thm:main} we prove that stronger estimates hold:

\begin{cor}     \label{cor:tcsestimates}
    Let $k \in \N$ and let $\delta > 0$ be smaller than $\mu$ and the square roots of the smallest non-zero eigenvalues of the Laplacian acting on $2$-and $3$-forms on $X$. Then there exists a constant $C > 0$ such that for $T$ large enough the following estimates hold:
    \begin{equation*}
        \| \widetilde \varphi_T - \varphi_T \|_{C^k} \leq C e^{-\delta T} .
    \end{equation*}
\end{cor}

This result gives a much stronger control on $\| \widetilde \varphi_T - \varphi_T \|$ than the $C^{0,\alpha}$ bound that was the best previously known. It has been used in various places in the literature but to our knowledge there was still no clear proof of this fact. These stronger estimates are especially important to do analysis on twisted connected sums, as they also imply a control on $g_{\widetilde \varphi_T} - g_{\varphi_T}$ and all its derivatives in $O(e^{-\delta T})$, and similarly for the operators that are naturally associated with it. Hence the study of the mapping properties of elliptic operators on twisted connected sums is amenable to the techniques described in Sections \ref{section:construction} and \ref{section:spectral}, and in particular Theorem \ref{thm:main} and Theorem \ref{thm:density} also apply in this context. 

\paragraph{}    \label{par:relsdc} Let us explain how our results can be related to the swampland distance conjectures in physics. Consider a family of compact $G_2$-manifolds $(M_T, \widetilde \varphi_T)$ constructed by twisted connected sum. The cross section $X$ is the product of a K3 surface with a flat $2$-torus. Therefore, its Betti numbers are $b^0(X) = 1$, $b^1(X)= 2$, $b^2(X) = 23$ and $b^3(X) = 44$. By Theorem \ref{thm:density} and Corollary \ref{cor:tcsestimates}, we understand the asymptotic behaviour of the masses of the light Kaluza--Klein states in the limit $T \rightarrow \infty$. Indeed, the decay rate and the distribution of the low eigenvalues of the Laplacian on $(M_T, \widetilde \varphi_T)$ are equivalent to those of the Laplacian on $S^1_{2T} \times T^2 \times K3$. This may have an interpretation in terms of dualities in physics.

Quantitatively, we may deduce that the lower spectrum of the Laplacian of $\widetilde \varphi_T$ acting on co-closed $3$-forms is made of an infinite number of eigenvalues which decay at rate $T^{-2}$. Moreover, for large enough $T$ the number of eigenvalues below $\frac{\pi^2 s}{T^2}$ is of order $88 \sqrt s$. From the physical perspective, when we consider M-theory compactified on $(M_T,\widetilde \varphi_T)$ this translates into an infinite tower of scalar fields becoming light as $T \rightarrow \infty$, in line with the swampland distance conjectures. By the above analysis, we obtain not only the rate of decay but also the density of this tower of states. Considering $0$, $1$ and $2$-forms we would similarly obtain other towers of Kaluza--Klein states becoming light. The geometric origin of these towers of light states is found to be the topology of the cross-section, through the roots of the Laplacian acting in the cylindrical neck region. 

\paragraph{}    \label{par:strategy} Let us finish this section with an overview of our strategy. We prove Theorem \ref{thm:main} by an explicit construction method, similar to the constructions by Kapouleas of minimal surfaces in Euclidean space \cite{kapouleas1990complete,kapouleas1991compact}, by which we were inspired. The idea is to use cutoffs separate the analysis in three different domains: the neck region, which is close to a finite cylinder $[-T, T] \times X$, and two regions isometric to the domains $\{ \rho_i \leq T+1 \} \subset Z_i$. One challenge is that when the indicial operator $P_0$ acting on the cylinder has real roots, it is not invertible nor even Fredholm in the Sobolev range that we would like to consider. However, this failure is due to the asymptotic behaviour of solutions, and we only need to work on a compact region of the cylinder. To deal with this issue, let us denote by $W^{l,p}_c$ the subspace of $W^{l,p}$ constituted by sections with compact essential support. In Section \ref{section:pdecyl} we prove the following theorem, which is the main analytical ingredient of our construction:

\begin{thm}		\label{thm:localsol}
  Let $P_0 : C^\infty(\underline E_0) \rightarrow C^\infty(\underline F_0)$ be an elliptic translation-invariant operator of order $k$ acting on the cylinder $\R \times X$. Let $d$ be the maximal order of a real root of $P_0$.
  
  For $p > 1$, there exists a map $Q_0 : L^p_c(\underline F_0) \rightarrow W^{k,p}_{loc}(\underline E_0)$ such that for any $f \in L^p(\underline F_0)$ with compact essential support, $P_0 Q_0 f = f$. Moreover there exists a constant $C > 0$ such that for any $T \geq 1$ and any $f \in L^p_c(\underline F_0)$ with essential support contained in $(-T,T) \times X$:
  \begin{equation*}
      \| Q_0 f \|_{W^{k,p}((-T,T) \times X)} \leq C T^d \|f\|_{L^p}.
  \end{equation*}
\end{thm}
  
\begin{rem}
  The existence of the map $Q_0$ can be deduced from standard results as \cite{lockhart1985elliptic} or \cite{melrose1993atiyah} for instance. However, the explicit expression that we give for $Q_0$ will be important for our purpose, since a precise understanding of the asymptotic behaviour of $Q_0f$ will play a key role in the construction of Section \ref{section:construction}. 
\end{rem}

With this theorem in hand, we can try to build approximate solutions of the equation $P_T u = f$ by first taking a cutoff $f_0$ of $f$ in the neck region, and considering an equation of the type $P_0 u_0 = f_0$. We can consider $f_0$ as a section of $\underline F_0$ supported in $[-T,T] \times X$. This equation can be solved using the above theorem. Taking cutoff $u_0$ of the solution, the equation $P_T u = f$ can be replaced with an equation of the form $P_T (u-u_0) = f'$, where $f'$ is appropriately small in the neck region. Thus $f'$ can be written as a sum $f_1 + f_2$, where each $f_i$ is a section of $F_i$ defined in the domain $\{ \rho_i \leq T + 1 \} \subset Z_i$ and satisfies good decay properties. This allows us to use weighted analysis to study the equations $P_i u_i = f_i$ in a range where the operators $P_i$ satisfy the Fredholm property. Similar ideas can be found for instance in \cite{ozuch2022noncollapsed}. 

Unfortunately, there are generally obstructions to solving $P_i u_i = f_i$ in weighted spaces, and the main difficulty is to understand how these obstructions interact. Using a pairing defined in \S\ref{subsection:polysol}, we can keep track of the obstructions and express their vanishing (up to an exponentially decaying error term) as a finite-dimensional linear system, which we call the \emph{characteristic system} of our gluing problem. The unknown of this system is an element $v \in \mE$ which represents our degrees of freedom in solving the equation $P_0 u = f_0$. The coefficients of the system are linearly determined by $f$. In \S\ref{subsection:construction}, we prove that in full generality the characteristic system admits a solution if and only if $f$ is orthogonal to the substitute cokernel. With the extra assumption that $P$ has only one root, this allows us to build an approximate solution of the equation $P_Tu=f$, and when $T$ is large enough we can prove Theorem \ref{thm:main} using an iteration process. However, our method could apply more generally, as long as one can ensure that the characteristic system admits a solution with reasonable bounds.

\section{Translation-invariant elliptic PDEs on cylinders}  	\label{section:pdecyl}

Throughout this section, we fix a compact oriented manifold $X$ and let $Y = \R \times X$. If $E \rightarrow X$ is a vector bundle, we denote by $\underline E \rightarrow Y$ the pull-back of $E$ by the projection $Y \rightarrow X$ on the second factor. Given any connection $\nabla$ on $E$, we can endow $\underline E$ with the pull-back connection $\underline \nabla$. Parallel transport along the vector field $\frac{\del}{\del_t}$ naturally defines a translation operator on $\underline E$. A section of $\underline E$ is translation-invariant if and only if it is the pull-back of a section of $E$. We assume that $Y$ is equipped with a cylindrical metric $g_Y = dt^2 + g_X$ and endow $\underline E$ with a translation-invariant metric. Sobolev and $C^k$ norms on $Y$ are defined with respect to this data. 

In \S\ref{subsection:sepvar}, we introduce some background about analysis on cylinders. In \S\ref{subsection:polysol} we study the action of a general elliptic translation-invariant operator $P$ on polyhomogeneous sections and define a pairing between the spaces of polyhomogeneous solutions of $Pu = 0$ and of $P^* v = 0$. Last, we prove Theorem \ref{thm:localsol} in \S\ref{subsection:localex} by constructing explicit solutions of the equation $Pu = f$.  Using the above pairing, we can precisely analyse the asymptotic behaviour of these solutions, which will play a key role in Section \ref{section:construction}.

    \subsection{Analysis on cylinders by separation of variables}       \label{subsection:sepvar}

\paragraph{} On the cylinder $Y = \R \times X$, we have natural isomorphisms identifying $L^p(\underline E)$ with $L^p(\R,L^p(E))$ for any $p \geq 1$, which follow from Fubini's theorem. Therefore we can think of sections of translation-invariant vector bundles over $Y$ as maps from $\R$ to an appropriate Banach space of sections over $X$. Moreover, for $p \geq 2$ there is a continuous embedding $L^p(E) \hookrightarrow L^2(E)$, which therefore induces a continuous embedding $L^p(\underline E) \rightarrow L^p(\R,L^2(E))$. Hence there exists a constant $C > 0$ such that for any $u \in L^p(\underline E)$:
\begin{equation*}
    \|u\|_{L^p(\R,L^2(E))} = \left( \int_\R \|u_t \|^p_{L^2} dt \right)^{\frac{1}{p}} \leq C \|u\|_{L^p}
\end{equation*}
where $u_t = \left. u \right|_{\{t\} \times X}$. The main tools that we will need for the analysis of PDEs on a cylinder $\R \times X$ are the Fourier transform and the convolution along the variable $t \in \R$. Below we recall some definitions.

Let $H$ be a complex Hilbert space, and consider functions $f : \R \rightarrow H$. Later, we will take $H$ to be the space $L^2(E)$. We denote by $\mathscr S(\R,H)$ the space of $H$-valued Schwartz functions, that is, the space of smooth functions taking values in $H$ that have all derivatives rapidly decaying at infinity. If $f \in \mathscr S(\R,H)$, we can define its Fourier transform $\hat f : \R \rightarrow H$ by:
\begin{equation}		\label{eq:defft}
    \hat f(\lambda) = \int_\R e^{-i \lambda t} f(t) dt
\end{equation}
for any $\lambda \in \R$. As in the case of scalar-valued functions, $\hat f$ also belongs to the space $\mathscr S(\R,H)$. This defines an invertible map $\mathscr S(\R,H) \rightarrow \mathscr S(\R,H)$, and the inverse Fourier transform takes the usual expression. As $H$ is a Hilbert space, the Plancherel theorem holds and the Fourier transform extends to a bounded linear map $L^2(\R,H) \rightarrow L^2(\R,H)$, which is, up to constant, an isometry \cite[Th. 2.47]{abels2011pseudodifferential}.

Just as for scalar-valued maps, the Plancherel theorem also implies that the Fourier transform can be extended to the dual space $\mathscr{S}^\prime(\R,H)$ of $\mathscr{S}(\R,H)$. In particular, for any $p \geq 1$ we can define the Fourier transform of an $L^p$-function through the embedding $L^p(\R,H) \rightarrow \mathscr{S}^\prime(\R,H)$. On $\mathscr{S}^\prime(\R,H)$ we can also define weak derivatives by duality, and moreover the relation $\widehat{f^\prime}(\lambda) = i\lambda \hat f(\lambda)$ can also be proved by evaluating against test functions.

Let us consider two complex Hilbert spaces $H_1,H_2$ and let $B(H_1,H_2)$ be the Banach space of bounded linear operators from $H_1$ to $H_2$. Let $R : \R \rightarrow B(H_1,H_2)$ be a smooth map, such that $R$ an all its derivatives have at most polynomial growth. Then $R$ induces a linear operator $A_R : \mathscr S(\R,H_1) \rightarrow \mathscr S(\R,H_2)$ acting on a Schwartz function $f$ by:
\begin{equation}
    A_R[f](t) = \frac{1}{2\pi} \int e^{i \lambda t} T(\lambda) \hat f(\lambda) d\lambda, ~~~ \forall t \in \R.
\end{equation}
The Plancherel theorem implies that if $T$ is bounded, then $A_R$ extends to a bounded linear operator $L^2(\R,H_1) \rightarrow L^2(\R,H_2)$. The Hilbert space valued Mikhlin multiplier theorem gives a sufficient condition for $A_R$ to extend as a bounded linear map for other $L^p$-spaces (see \cite[Th. 5.8]{abels2011pseudodifferential} or \cite[Th. 6.1.6]{bergh2012interpolation}):

\begin{thm}     \label{thm:multiplier}
    Assume that there exists a constant $C$ such that $\| R(\lambda)\| + \|\lambda R^\prime(\lambda) \| \leq C$ for all $\lambda \in \R$, where $\| \cdot \|$ denotes the norm of $B(H_1,H_2)$. Then $A_R$ extends as a bounded linear map $L^p(\R,H_1) \rightarrow L^p(\R,H_2)$ for all $1 < p < \infty$.
\end{thm}

In \S\ref{subsection:localex}, we will also need to consider functions $F : \R \rightarrow H$ defined by integrals of the form
\begin{equation*}
    F(t) = \int_{-\infty}^t \frac{(t-\tau)^{l-1}}{(l-1)!}  f(\tau) d\tau
\end{equation*}
where $l \geq 1$ and $f \in L^1_c(\R,H)$ is a compactly supported integrable function. The function $F$ is continuous, and since the support of $f$ is compact $F$ has at most polynomial growth at infinity, and therefore it defines an element of $\mathscr{S}^\prime(\R,H)$. If we denote the Heaviside step function by $H$ and define $H_l(t) = \frac{t^{l-1}}{(l-1)!}H(t)$, then $F$ can be written more compactly as the convolution $H_l * f$. Note that the $n$-th order weak derivative of $H_l$ is $H_{l-n}$ if $n < l$ and the Dirac mass $\delta$ if $l = n$. The weak derivatives of $F \in \mathscr{S}^\prime(\R,H)$ are naturally given by:
\begin{equation}        \label{eq:weakderf}
    F^{(n)} = 
    \begin{cases}
        H_{l-n} * f & \text{if} ~ l < n, \\
        f  &\text{if} ~ l = n, \\
        f^{(n-l)} & \text{if} ~ l > n .
    \end{cases}
\end{equation}
This can be proved by integrating against a test function $g \in \mathscr{S}(\R,H)$, as in the case of scalar-valued functions. 
   
\paragraph{} Let us now turn to the study of PDEs on cylinders. Let $\underline E$ and $\underline F$ be translation-invariant vector bundles over $Y = \R \times X$, equipped with translation-invariant metrics and connections. We will denote by $y = (t,x)$ the points in $Y$. Moreover let $\del_t$ be the covariant derivative along $\frac{\del}{\del_t}$, and define $D_t = -i\del_t$. 
 
A differential operator $P : C^\infty(\underline E) \rightarrow C^\infty(\underline F)$ of order $k$ is translation-invariant if it takes the form:
\begin{equation*}
    P(x,\del_x,D_t) = \sum_{l=0}^k A_{k-l}(x,\del_x) D_t^l
\end{equation*}
where $A_{k-l}(x,\del_x)$ are differential operators $C^\infty(E) \rightarrow C^\infty(F)$. If $P$ has order $k$, it is a standard fact that it induces continuous maps $P : W^{k+l,p}(\underline E) \rightarrow W^{l,p}(\underline F)$ on Sobolev spaces of sections.

From now on we assume that $P$ is elliptic, and for any $T > 0$ we denote by $\underline E_T$ the restriction of $\underline E$ to the finite cylinder $(-T,T) \times X$. If $u \in L^p(\underline E_2)$ and $Pu \in W^{l,p}(\underline E_2)$, then by elliptic regularity the restriction of $u$ to $(-1,1) \times X$ is in $W^{k+l,p}(\underline E_1)$, and moreover we have interior estimates:
\begin{equation*}
    \| u \|_{W^{k+l,p}(\underline E_1) } \leq C \left ( \| Pu \|_{W^{l,p}(\underline E_2)} + \| u \|_{L^p(\underline E_2)} \right ) .
\end{equation*}
Combined with the translation-invariance of $P$, we get interior estimates for sections of $\underline E_T$ that are independent of $T$:

\begin{prop}	\label{prop:lpaprioriest}
    Let $P : C^\infty(E) \rightarrow C^\infty(F)$ be a translation-invariant elliptic operator of order $k$, and let $p > 1$ and $l \in \N$. Then there exists $C > 0$ such that for any $T \geq 1$ the following holds. If $f \in W^{l,p}(\underline F_{T+1})$ and $u \in L^p(\underline E_{T+1})$ is a solution of $Pu = f$, then $\left. u \right|_{(-T,T) \times X} \in W^{k+l,p}(\underline E_T)$ with the bound:
    \begin{equation*}
        \| u \|_{W^{k+l,p}(\underline E_T) } \leq C \left ( \| f \|_{W^{l,p}(\underline E_{T+1})} + \| u \|_{L^p(\underline E_{T+1})} \right ) .
    \end{equation*}
\end{prop}
  
When $p \geq 2$, we can make a stronger statement. If $u \in L^2(\underline E_2)$ and $Pu \in W^{l,p}(\underline E_2)$, then the restriction of $u$ to $(-1,1) \times X$ is in $W^{k+l,p}(\underline E_1)$ and moreover we have an estimate:
\begin{equation*}
    \| u \|_{W^{k+l,p}(\underline E_1) } \leq C \left ( \| Pu \|_{W^{l,p}(\underline E_2)} + \| u \|_{L^2(\underline E_2)} \right ) .
\end{equation*}
Since $L^p((-2,2), L^2(E))$ continuously embeds into $L^2((-2,2),L^2(E)) \simeq L^2(\underline E_2)$, there exists a constant $C^\prime > 0$ such that if $u \in L^p((-2,2),L^2(E))$ then we have the following interior estimate:
\begin{equation*}
    \| u \|_{W^{k+l,p}(\underline E_1) } \leq C \left ( \| Pu \|_{W^{l,p}(\underline E_2)} + \| u \|_{L^p((-2,2),L^2(E))} \right ) .
\end{equation*}
Using translation-invariance this implies:

\begin{prop}  \label{prop:lptwoaprioriest}
    Let $P : C^\infty(E) \rightarrow C^\infty(F)$ be a translation-invariant elliptic operator of order $k$, and let $p \geq 2$ and $l \in \N$. Then there exists $C > 0$ such that for any $T \geq 1$ the following holds. If $f \in W^{l,p}(\underline F_{T+1})$ and $u \in L^p((-T-1,T+1),L^2(E))$ is a solution of $Pu = f$, then $\left. u \right|_{(-T,T) \times X} \in W^{k+l,p}(\underline E_T)$ with the bound::
    \begin{equation*}
        \| u \|_{W^{k+l,p}(\underline E_T) } \leq C \left ( \| f \|_{W^{l,p}(\underline E_{T+1})} + \| u \|_{L^p((-T-1,T+1),L^2(E))} \right ) .
    \end{equation*}
\end{prop}

\paragraph{} \label{par:roots} In the remainder of this part, we will be concerned with equations of the type
\begin{equation}		\label{eq:pufcore}
    P(x,\del_x,D_t) u(t,x) = f(t,x)
\end{equation}
where $P$ is a translation-invariant elliptic operator. It is usually studied by taking its Fourier transform in the variable $t$, which takes the form
\begin{equation}	\label{eq:ftpuf}
    P(x,\del_x,\lambda) \hat u(x,\lambda) = \hat f(x,\lambda).
\end{equation}
For any fixed $\lambda \in \C$, the operator $P(x,\del_x,\lambda) : C^\infty(E) \rightarrow C^\infty(F)$ is an elliptic operator of order $k$, and hence defines Fredholm maps $W^{k+l,p}(E) \rightarrow W^{l,p}(F)$ for $p > 1$ and $l \geq 0$. By the results of \cite{agmon1961properties} these maps are analytic in the variable $\lambda$, and there exists a discrete set $\mathscr C_P \subset \C$ such that the homogeneous equation 
\begin{equation} 	\label{eq:ftzero}
    P(x,\del_x,\lambda) \hat u(x,\lambda) = 0
\end{equation}
has a non-trivial solution if and only if $\lambda \in \mathscr C_P$. Moreover, $\mathscr C_P$ is finite on any strip $\{ \delta_1 < \im \lambda < \delta_2 \}$ of $\C$. The elements of $\mathscr C_P$ are called the \emph{roots} of $P$. The discrete set $\mathscr D_P = \{ \im \lambda, \lambda \in \mathscr C_P \}$ is called the set of \emph{indicial roots} of $P$.

\begin{ex}     \label{ex:roots}
    Consider the translation-invariant bundle $\Lambda_\C T^* Y$ of complex-valued differential forms. It splits as a direct sum:
    \begin{equation*}
        \Lambda_\C T^*Y = \underline{\Lambda_\C T^* X} \oplus dt \wedge \underline{\Lambda_\C T^* X}
    \end{equation*}
    where $\underline{\Lambda_\C T^* X}$ is the pull-back of the bundle of differential forms on $X$. The operators $d_Y$ and $d^*_Y$ take the form:
    \begin{equation*}
        \begin{cases}
            d_Y (\alpha + dt \wedge \beta) = d_X\alpha + dt \wedge (\del_t \alpha - d_X \beta) \\
            d_Y^* (\alpha + dt \wedge \beta) = d_X^* \alpha - \del_t \beta - dt \wedge d_X^*\beta
        \end{cases}
    \end{equation*}
    Thus if we define $J \in \End(\Lambda_\C T^*Y)$ by $J \eta = dt \wedge \eta - \iota_{\frac{\del}{\del t}} \eta$, where $\iota$ denotes the interior product, we can write the Fourier transform of the operator $d_Y + d^*_Y$ as 
    \begin{equation*}
        (d_Y+d^*_Y)(\lambda) \eta = (d_X + d^*_X)\alpha - dt \wedge (d_X+d^*_X) \beta + i \lambda J\eta.
    \end{equation*}
    with $\eta = \alpha + dt \wedge \beta$. The Laplacian  $\Delta_Y = d_Yd^*_Y + d^*_Yd_Y$ can be written as $\Delta_Y = -\del_t^2 + \Delta_X$, so that its Fourier transform is $\Delta_X + \lambda^2$. For both operators, the roots are exactly the values $\pm i \sqrt{\lambda_n}$, where $\lambda_n \geq 0$ are the eigenvalues of the Laplacian $\Delta_X$. In particular the only real root is $\lambda_0 = 0$, and the corresponding translation-invariant solutions are of the form $\alpha + dt \wedge \beta$, where $\alpha$ and $\beta$ are harmonic forms on $X$.
\end{ex}

\paragraph{} For $\lambda \in \C$, we will write $P(\lambda)$ as a short-hand for $P(x,\del_x,\lambda)$. It can be seen as a Fredholm map $W^{k+l,p}(E) \rightarrow W^{l,p}(F)$, analytic in the variable $\lambda$. This implies that $P(\lambda)$ is invertible for $\lambda \notin \mathscr C_P$ \cite{agmon1961properties,agranovich1963elliptic}. Its inverse $R(\lambda)$ is called the \emph{resolvent} of $P(\lambda)$; for any $m \leq k+l$ it can be considered as a bounded operator form $W^{l,p}(F)$ to $W^{m,p}(E)$ (which is compact when $m<k+l$). We will denote by $\| R(\lambda) \|_{l,m}$ the operator norm of the resolvent seen as a map $W^{l,p}(F) \rightarrow W^{m,p}(E)$. By the results of \cite{agmon1961properties} the resolvent is meromorphic in $\lambda \in \C$, with poles exactly at the roots of $P$. That is, around any $\lambda_0 \in \mathscr C_P$ we can write:
\begin{equation*}
    R(\lambda) = \frac{R_{-d}(\lambda_0)}{(\lambda-\lambda_0)^d} + \cdots + \frac{R_{-1}(\lambda_0)}{\lambda-\lambda_0} + \sum_{n=0}^{\infty} R_n(\lambda_0) (\lambda-\lambda_0)^n
\end{equation*}
where $R_{l}(\lambda_0)$ are bounded operators $W^{l,p}(F) \rightarrow W^{m,p}(E)$ and the series has positive radius of convergence. The largest positive integer $d$ such that $R_{-d}(\lambda_0) \neq 0$ is called the order of $\lambda_0$. The notions of root, pole and order do not depend on the Sobolev spaces we choose to work with.
 
The following bounds on the resolvent $R(\lambda)$ and its derivative $R^\prime(\lambda) = \frac{dR}{d\lambda}(\lambda)$ are crucial for our purpose, and follow from the more general \cite[Theorem 5.4]{agmon1961properties}:

\begin{thm}	\label{thm:boundres}
    Let $p > 1$, $l \in \N$ and $P$ be a translation-invariant elliptic operator. Then the following holds:
    \begin{enumerate}[(i)]
	    \item The resolvent $R(\lambda)$ has no poles in a double sector $\{ \arg( \pm \lambda) \leq \delta, ~~ | \lambda | \geq N \}$ and in this domain there exists a constant $C>0$ such that:
        \begin{equation*}
            \sum_{j=0}^k \left \| \lambda^{k-j} R(\lambda) \right \|_{l,l+j} \leq C.
        \end{equation*}
	    \item Furthermore, as $| \lambda | \rightarrow \infty$ along the real axis:
	    \begin{equation*}
	       \sum_{j=0}^k \left \| \lambda^{k-j} R^\prime(\lambda) \right \|_{l,l+j} = O \left( \frac{1}{\lambda} \right ).
	    \end{equation*}
    \end{enumerate}
\end{thm}

\paragraph{} The last result that we want to mention here is the following well-known proposition (see \cite{kondrat1967boundary} for an original reference), which can be seen as a particular case of Theorem \ref{thm:localsol}. When $P$ has not roots along the real axis the following holds.

\begin{prop}		\label{prop:dzero}
    Let $p >1$, $l \in \N$, and assume that $P$ has no real roots. Then the map $W^{k+l,p}(\underline E) \rightarrow W^{l,p}(\underline F)$ induced by $P$ admits a bounded inverse.
\end{prop}

A sketch proof of this Proposition is as follows. If $f$ is a smooth, compactly supported section of $\underline F$, then equation \eqref{eq:ftpuf} admits a solution $\hat u(\lambda) = R(\lambda) \hat f(\lambda)$, where we can consider $\hat f$ as a Schwartz function valued in $L^2(F)$ and the resolvent $R(\lambda)$ as a bounded map $L^2(F) \rightarrow L^2(E)$.  Hence we have a solution $u = Q[f] \in \mathscr S(\R,L^2(E))$ of $Pu = f$ defined as:
\begin{equation}
    Q[f](t) = \frac{1}{2\pi} \int e^{i\lambda t} R(\lambda) \hat f(\lambda) d\lambda, ~~~ \forall t \in \R .
\end{equation}
It follows from Theorem \ref{thm:multiplier} and the above bounds on the resolvent that $Q$ extends to a bounded linear map $L^p(\R,L^2(F)) \rightarrow L^p(\R,L^2(E))$ for any $1 < p < \infty$. If $p \geq 2$, the fact that $P$ admits a bounded inverse in the $L^p$-Sobolev range can therefore be deduced from Proposition \ref{prop:lptwoaprioriest} and the continuous embedding $L^p(\underline E) \hookrightarrow L^p(\R,L^2(E))$, and the case $1 < p < 2$ can be treated by duality.

When $P$ has real roots the statement of Proposition \ref{prop:dzero} does not hold anymore and the map induced by $P$ on Sobolev spaces is not even Fredholm. It still has finite-dimensional kernel but the cokernel has infinite dimension. In order to understand the mapping properties of $P$ in more detail, we want to make sense of the inverse Fourier transform of the singular part of the resolvent.

    \subsection{Polyhomogeneous sections}  \label{subsection:polysol}
  
\paragraph{} In this part, we prove that the action of $P$ on polyhomogeneous sections admits a right inverse and introduce a pairing which will play an important role in Section \ref{section:construction}. A section of $\underline E \rightarrow Y$ is called exponential if it is of the form $u(x,t) = e^{i\lambda t} p(x,t)$, where $\lambda \in \C$ is called the rate of $u$ and $p$ is polynomial in the variable $t$. A polyhomogeneous section is a finite sum of exponential sections. 
  
To understand the action of $P$ on polyhomogeneous sections, we fix $\lambda_0 \in \C$ and define:
\begin{equation}
    P_{\lambda_0}(x,\del_x, D_t) = e^{-i \lambda_0 t} P(x,\del_x, D_t) e^{i\lambda_0 t}
\end{equation}
which is a translation-invariant operator on $Y$. Explicitly $P_{\lambda_0}$ has for expression:
\begin{equation*}
    P_{\lambda_0}(D_t) = \sum_{n \geq 0} \frac{1}{n!} \frac{\del^n P}{\del \lambda^n} (\lambda_0) D_t^n.
\end{equation*}
We consider $P_{\lambda_0}$ as an operator mapping the space $W^{k,p}(E)[t]$ into $L^p(F)[t]$, that is we consider the action on sections of $\underline E \rightarrow Y$ that are polynomial in the variable $t$ and have $W^{k,p}$ coefficients. Our goal is to show that $P_{\lambda_0}$ admits a right inverse $Q_{\lambda_0}$.
  
Consider the resolvent $R(\lambda)$ as an operator $L^p(F) \rightarrow W^{k,p}(E)$. If $\lambda_0$ is a root of $P$, it is a pole of $R$ and we denote by $d(\lambda_0)$ its degree. By convention we set $d(\lambda_0) = 0$ if $\lambda_0$ is not a root of $P$. In general we may expand $R(\lambda)$ near $\lambda_0$ as:
\begin{equation*}
    R(\lambda) = \frac{R_{-d(\lambda_0)}(\lambda_0)}{(\lambda-\lambda_0)^{d(\lambda_0)}} + \cdots + \frac{R_{-1}(\lambda_0)}{\lambda-\lambda_0} + R_0(\lambda_0) + \sum_{m \geq 1} R_m(\lambda_0) (\lambda-\lambda_0)^m.
\end{equation*}
where for $m \geq - d(\lambda_0)$, $R_m(\lambda_0) : L^p(X,F) \rightarrow W^{k,p}(X,E)$ are bounded operators. The relations $R(\lambda) P(\lambda) = \Id_{W^{k,p}(E)}$ and $P(\lambda)R(\lambda) = \Id_{L^p(F)}$ that hold away form the roots of $P$ imply:
\begin{equation}		\label{eq:przero}
    \sum_{m+n = 0} \frac{1}{n!} R_m(\lambda_0) \frac{\del^n P}{\del \lambda^n}(\lambda_0) = \Id_{W^{k,p}(E)}, ~~ \sum_{m+n = 0} \frac{1}{n!}\frac{\del^n P}{\del \lambda^n}(\lambda_0) R_m(\lambda_0) = \Id_{L^p(F)} 
\end{equation}
and for any non-zero $l \in \Z$:
\begin{equation}		\label{eq:prl}
    \sum_{m+n = l} \frac{1}{n!} R_m(\lambda_0) \frac{\del^n P}{\del \lambda^n}(\lambda_0) = 0 = \sum_{m+n = l} \frac{1}{n!}\frac{\del^n P}{\del \lambda^n}(\lambda_0) R_m(\lambda_0) .
\end{equation}

\begin{ex}        \label{ex:resolvent}
    One can easily see from Example \ref{ex:roots} that $\lambda_0 = 0$ is a root of order $1$ of the operator $d_Y+d_Y^*$, and of order $2$ for the operator $\Delta_Y$. The singular parts of their resolvent can be computed with the above relations. For the operator $d_Y + d^*_Y$, relations \eqref{eq:prl} for $l=-1$ imply that $R_{-1}^{d+d^*}(0)$ vanishes on the orthogonal space to harmonic forms and maps into the space of harmonic forms.  Relations \eqref{eq:przero} imply that: 
    \begin{equation*}
        iR_{-1}^{d+d^*}(0) J \eta = \eta = iJ R_{-1}^{d+d^*}(0) \eta 
    \end{equation*}
    for any translation-invariant harmonic form on $Y$. As $J^2 = - 1$ we obtain:
    \begin{equation*}
        R_{-1}^{d+d^*}(0) = iJ \circ p_h = i p_h \circ J
    \end{equation*}
    where $p_h$ is the $L^2$-orthogonal projection onto the space of harmonic forms. As the Laplacian $\Delta_Y$ is the square of the operator $d_Y +d^*$ this implies that:
    \begin{equation*}
        R_{-2}^\Delta(0) =  (R_{-1}^{d+d^*}(0))^2 = (iJ)^2 p_h^2 = p_h.
    \end{equation*}
    On the other hand, as the Fourier transform of $\Delta_Y$ is an analytic function of the variable $\lambda^2$ it is easy to see that $R_{-1}^\Delta(0) = 0$.
\end{ex}

\paragraph{} With these notations in hand, let $D_t^{-1}$ be the endomorphism of $L^p(F)[t]$ mapping $\frac{(it)^j}{j!} v$ to $\frac{(it)^{j+1}}{(j+1)!} v$ for any $v \in L^p(F)$. This is a right inverse of $D_t$. Let us define the operator $Q_{\lambda_0} : L^p(F)[t] \rightarrow W^{k,p}(E)[t]$ by:
\begin{equation*}
    Q_{\lambda_0}(D_t, D_t^{-1}) = \sum_{m \geq - d(\lambda_0)} R_m(\lambda_0) D_t^m.
\end{equation*}
It maps polynomials of order $m$ to polynomials of order at most $m + d(\lambda_0)$. Moreover relations \eqref{eq:przero} and \eqref{eq:prl} imply the following: 

\begin{lem}
  The map $Q_{\lambda_0} : L^p(F)[t] \rightarrow W^{k,p}(E)[t]$ is a right inverse of $P_{\lambda_0}$. 
\end{lem}

\paragraph{} Let us now turn our attention to the kernel of $P_{\lambda_0}$. It is non-trivial if and only if $\lambda_0$ is a root of $P$, which amounts to saying that the homogeneous equation $P_{\lambda_0} u = 0$ admits a non-trivial translation-invariant solution. Moreover, the kernel of $P_{\lambda_0}$ acting on polynomial sections in the variable $t$ is always finite-dimensional, and the degree of its elements is bounded above by the order of the root $\lambda_0$ minus one \cite{agmon1961properties}. In particular, if $\lambda_0$ has order $1$ the only polynomial solutions of $P_{\lambda_0} u = 0$ are translation-invariant.
  
For any root $\lambda_0$ of $P$, let us denote by $\mE_{\lambda_0}$ the (finite-dimensional) space of exponential solutions of $Pu = 0$ of rate $\lambda_0$, and $\mE^*_{\overline \lambda_0}$ the space of exponential solutions of $P^* v= 0$ of rate $\overline \lambda_0$ (note that $P(\lambda)^* = P^*(\overline \lambda)$ for any $\lambda \in \C$). As we are mainly interested in the real roots of $P$, we denote by $\lambda_1,...,\lambda_m$ the real roots and:
\begin{equation*}
    \mE = \bigoplus_{j=1}^m \mE_{\lambda_j}, ~~~ \mE^* = \bigoplus_{j=1}^m \mE^*_{\lambda_j}.
\end{equation*}
  
We now define a pairing $\mE \times \mE^* \rightarrow \C$ and derive its basic properties. Let $\chi : \R \rightarrow \R$ be a smooth function such that $\chi \equiv 0$ in a neighbourhood of $- \infty$ and $\chi \equiv 1$ in a neighbourhood of $+ \infty$. We define a sesquilinear pairing $( \cdot, \cdot ) : \mE \times \mE^* \rightarrow \C$ by the integral:
\begin{equation}
    (u, v) = \int_\R \left \langle P(D_t) \left [ \chi(t) u(t) \right ], v(t) \right \rangle dt.
\end{equation}
Here, we denote by $\langle \cdot, \cdot \rangle$ the $L^2$-product on the compact manifold $X$. This is well-defined as $P(D_t) \left [ \chi(t) u(t)) \right ]$ is compactly supported for any $u \in \mE$. Further, it does not depend on the choice of function $\chi$. Indeed, if $\tilde \chi$ is another smooth function that satisfies the same assumptions, define $\chi_\tau = (1-\tau) \chi + \tau \tilde \chi$ for $\tau \in [0,1]$. As $\frac{\del\chi_\tau}{\del \tau}(t) u(t)$ is a compactly supported section of $\underline E$, we can integrate by parts to obtain:
\begin{equation*}
    \frac{d}{d\tau} \int_\R \left \langle P(D_t) \left [ \chi_\tau(t) u(t) \right ], v(t) \right \rangle dt = \int_\R \left \langle P(D_t) \left [ \frac{\del\chi_\tau}{\del \tau}(t) u(t) \right ], v(t) \right \rangle dt = 0
\end{equation*}
as $P^*(D_T) v(t) = 0$. Therefore the pairing does not depend on the choice of $\chi$.
  
An important consequence of this observation is that $\mE_{\lambda_i}$ is orthogonal to $\mE^*_{\lambda_j}$ for the pairing $( \cdot, \cdot)$ unless $i = j$. Indeed, let $u \in \mE_{\lambda_i}$ and $v \in \mE^*_{\lambda_j}$, and replace $\chi(t)$ by $\chi(t-\tau)$ in the definition of the pairing, for $\tau \in \R$. Then we can compute by a change of variables:
\begin{equation*}
    \int_\R \left \langle P(D_t) \left [ \chi(t-\tau) u(t) \right ], v(t) \right \rangle dt = e^{i (\lambda_i - \lambda_j) \tau} \left [ (u,v) + \sum_{l \geq 1} a_l(u,v) \tau^l \right ]
\end{equation*}
where the coefficients $a_l(u,v)$ are independent of $\tau$, and only finitely many of them are non-zero. As this has to be equal to $(u,v)$ for all $\tau \in \R$, this implies $a_l = 0$ for $l \geq 1$ and $(u,v) = 0$ when $i \neq j$.
  
\paragraph{} The key property of the pairing $(\cdot, \cdot)$ is the following:
  
\begin{lem}
    The pairing $(\cdot, \cdot)$ is non-degenerate.
\end{lem}
 
\begin{proof}
    By the above remarks it suffices to show that the restriction of $(\cdot, \cdot)$ to $\mE_{\lambda_j} \times \mE^*_{\lambda_j}$ is non-degenerate. Consider first $v \in \ker P^*(\lambda_j)$, so that $\tilde v(t,x) = e^{i \lambda_j t} v(x)$ is an element of $\mE^*_{\lambda_j}$. Considering $v$ as an element of $L^2(F)[t]$, we define $u(t,x) = Q_{\lambda_j} v$. This is a polynomial of order at most $d(\lambda_j)$ in the variable $t$, and it satisfies:
    \begin{equation*}
        P_{\lambda_j}(D_t) u(t) = v.
    \end{equation*}
    Differentiating this expression in the $t$ variable it follows that:
    \begin{equation*}
        P_{\lambda_j}(D_t) [ D_t u(t) ] = 0
    \end{equation*}
    so that $\tilde u(t,x) = e^{i \lambda_j t} D_t u(t,x)$ is in $\mE_{\lambda_j}$. Let us now pick a function $\chi$ as above and compute:
    \begin{align*}
        \int_\R & \left \langle P(D_t) \left [ \chi(t) \tilde u(t) \right ] , e^{i \lambda_j t} v \right \rangle dt = \int_\R \left \langle P_{\lambda_j} (D_t) \left [ \chi(t) D_t u(t) \right ], v \right \rangle dt \\
        & = \frac 1 i \int \frac{d}{dt} \left \langle P_{\lambda_j}(D_t) \left[ \chi(t) u(t) \right], v \right \rangle dt - \frac 1 i \int \left \langle P_{\lambda_j}(D_t) \left[ \chi'(t) u(t) \right], v \right \rangle dt \\
        & = \frac 1 i \langle v, v \rangle - 0
    \end{align*}
    which holds because $P(D_t) [ \chi(t) u(t) ] \equiv v$ as $t$ goes to  $+ \infty$ and $P(D_t) [ \chi(t) u(t) ] = 0$ as $t$ goes to $ -\infty$. Thus we have $(\tilde u, \tilde v) = -i \|v\|_{L^2}^2$ which is non-zero when $v \neq 0$.
  
    In general, let $v(t,x)$ be an element of $\mE^*_{\lambda_j}$ of degree $m$. Then $e^{i\lambda_j t} D_t^m e^{-i\lambda_j t}v(t,x)$ is a non-zero element of $\mE^*_{\lambda_j}$ of degree zero. By the above argument there exists $u(t,x)$ in $\mE_{\lambda_j}$ such that $(u,e^{i\lambda_j t}D_t^m  e^{-i\lambda_j t} v) \neq 0$.  Moreover one can easily check that:
    \begin{equation*}
        (u,e^{i\lambda_j t}D_t^m  e^{-i\lambda_j t} v) = (e^{i\lambda_j t}D_t^m  e^{-i\lambda_j t} u, v)
    \end{equation*}
    and $e^{i\lambda_j t}D_t^m  e^{-i\lambda_j t} u \in \mE_{\lambda_j}$. Hence the pairing $(\cdot, \cdot)$ is non-degenerate.
\end{proof}
 
\begin{ex}     \label{ex:pairing}
    The space of translation-invariant solutions of the operator $d_Y + d^*_Y$ acting on $\Lambda_\C T^* Y$ is:
    \begin{equation*}
        \mE_{d+d^*} = \mE^*_{d+d^*} = \left \{ \alpha + dt \wedge \beta, ~ \alpha, \beta \in C^\infty(\Lambda_\C T^*X), ~ \Delta_X \alpha = \Delta_X \beta = 0 \right \}
    \end{equation*}
    If $\alpha + dt \wedge \beta, ~ \alpha'+ dt \wedge \beta' \in \mE$ we can compute their pairing:
    \begin{align*}
        (\alpha + dt \wedge \beta, \alpha' + dt \wedge \beta') & = \int \langle (d_Y + d^*_Y) (\chi(\tau) \alpha + dt \wedge \beta), \alpha' + dt \wedge \beta' \rangle d\tau \\
        & = \int \chi'(\tau) \langle dt \wedge \alpha - \beta, \alpha' + dt \wedge \beta' \rangle d\tau \\
        & = \langle \alpha, \beta' \rangle - \langle \beta, \alpha' \rangle
    \end{align*}
    which is clearly non-degenerate. 
 
    For the Laplacian $\Delta_Y$ acting on $q$-forms, the spaces $\mE_q$ and $\mE_q^*$ are both isomorphic to the space $q$-forms that can be written as $\eta_0 + t \eta_1$, where $\eta_i = \alpha_i + dt \wedge \beta_i$ with $\alpha_i \in \Omega^q_\C(X)$ and $\beta_i \in \Omega_\C^{q-1}(X)$ harmonic. In the same way one can easily derive:
    \begin{equation*}
        (\eta_0 + t \eta_1, \eta'_0 + t \eta'_1 ) = \langle \alpha_0, \alpha'_1 \rangle + \langle \beta_0, \beta'_1 \rangle - \langle \alpha_1, \alpha'_0 \rangle - \langle \beta_1, \beta'_0 \rangle.
    \end{equation*}
\end{ex}

    \subsection{Existence of solutions}	\label{subsection:localex}

\paragraph{} In this part we prove Theorem \ref{thm:localsol}, beginning by the case $p \geq 2$. Let us consider a translation-invariant elliptic differential operator $P : C^\infty(\underline E) \rightarrow C^\infty(\underline F)$ of order $k$ with real roots $\lambda_1, \ldots ,\lambda_m$. For $1 \leq j \leq m$, let $d(\lambda_j)$ be the order of the root $\lambda_j$. Considering the resolvent as a family of (compact) operators from $L^2(F)$ to $L^2(E)$, we have a decomposition of the form:
\begin{equation}        \label{eq:decompresolv}
    R(\lambda) = R_r(\lambda) + \sum_{j=1}^m \sum_{l = 1}^{d(\lambda_j)} \frac{R_{-l}(\lambda_j)}{(\lambda - \lambda_j)^l}
\end{equation}
where the regular part of the resolvent $R_r(\lambda)$ is an analytic function from a neighborhood of the real line in $\C$ to $B(L^2(F),L^2(E))$. We will denote the second term of the left-hand-side of equation \eqref{eq:decompresolv} by $R_s(\lambda)$; this is the singular part of the resolvent.
    
Since $p \geq 2$, a section $f \in L^p_c(\underline F)$ can be considered as an element of $L^p(\R,L^2(F))$, which has compact essential support. We want to find a solution of equation \eqref{eq:pufcore} through the study of the Fourier transformed equation \eqref{eq:ftpuf}. As the resolvent has poles we need to make sense of the expression $\hat u (\lambda) = R(\lambda) \hat f(\lambda)$, or rather of its inverse Fourier transform. 
    
Differentiating the identity $P(\lambda)R(\lambda) = \Id_{L^2(F)}$ and using the bounds of Theorem \ref{thm:boundres}, we see that the resolvent and all its derivatives have at most polynomial growth at infinity. Since this also true of the singular part of the resolvent, which is bounded at infinity as well as all of its derivatives, then the same holds for the regular part of the resolvent.  On the other hand, from Theorem  \ref{thm:boundres} we have a bound:
\begin{equation*}
    \left \| R(\lambda) \right \| + \left \| \lambda R^\prime(\lambda) \right \| = O\left( \frac{1}{\lambda}\right)
\end{equation*}
as $|\lambda| \rightarrow \infty$. Further this bound clearly also holds for the singular part of the resolvent. Therefore there exists a constant $C > 0$ such that for all $\lambda \in \R$ we have:
\begin{equation*}
    \left \| R_r(\lambda) \right \| + \left \| \lambda R^\prime_r(\lambda) \right \| \leq C
\end{equation*}
Thus Theorem \ref{thm:multiplier} implies that $R_r(\lambda)$ induces a bounded map $Q_r : L^p(\R,L^2(F)) \rightarrow L^p(\R,L^2(E))$ defined as:
\begin{equation}
    Q_r[v](t) = \frac{1}{2\pi} \int e^{i\lambda t} R_r(\lambda) \hat v(\lambda) d\lambda, ~~~ \forall v \in L^p(\R,L^2(E)). 
\end{equation}
Define $u_r = Q_r[f] \in L^p(\R,L^2(E))$. By definition we have $\hat u_r(\lambda) = R_r(\lambda) \hat f(\lambda)$ for $\lambda \in \R$. As $f$ has compact support, its Fourier transform $\hat f(\lambda)$ can be continued as an analytic $L^2(F)$-valued function of the variable $\lambda \in \C$. Moreover, $R_r(\lambda)$ has no poles in a complex strip of the form $\{  | \im \lambda | < \delta \}$ for some $\delta > 0$, and therefore $\hat u_r(\lambda)$ can be extended as an analytic $L^2(E)$-valued function for $\lambda$ varying in a neighbourhood of the real line in $\C$.
    
We now deal with the singular part of the resolvent. Our main problem is that we cannot directly make sense of the inverse Fourier transform of $R_s(\lambda) \hat f(\lambda)$. Nevertheless, it is natural to define the following:
\begin{equation}        \label{eq:defus}
    u_s(t) = Q_s[f](t) = \sum_{j=1}^m \sum_{l=1}^{d(\lambda_j)} i^l e^{i \lambda_j t} \int_{-\infty}^t \frac{(t-\tau)^{l-1}}{(l-1)!} e^{- i \lambda_j \tau} R_{-l}(\lambda_j) f(\tau) d\tau.
\end{equation}
Note that the integrals are well-defined because $f$ has compact essential support, and therefore $u_s$ is a map $\R \rightarrow L^2(E)$. If we define $H_{l,\lambda}(t) = e^{i\lambda_j t} \frac{i^l t^{l-1}}{(l-1)!} H(t)$ where $H$ is the Heaviside step function, then we can write $u_s$ more compactly as a convolution:
\begin{equation}
    u_s = \sum_{j=1}^m \sum_{l=1}^{d(\lambda_j)} H_{l,\lambda_j} * (R_{-l}(\lambda_j) f).
\end{equation}
In general $u_s$ is not in $L^p(\R,L^2(E))$, but its restriction to any finite interval $(-T,T)$ is $L^p$. We will shortly provide more precise estimates, but we first want to prove that $u = u_r + u_s$ satisfies $Pu =f$.

In order to do this, let us first compute $P(D_t) u_r(t)$, considering $D_t$ as a weak derivative wherever appropriate. Taking the Fourier transform, we may compute $P(\lambda)u_r(\lambda)$ for $\lambda \in \R \backslash \{\lambda_1,...,\lambda_m \}$ as follows. As $P(\lambda)R(\lambda) = \Id_{L^2(F)}$, we have:
\begin{equation*}
    P(\lambda) \hat u_r(\lambda) = \hat f(\lambda) - P(\lambda)R_s(\lambda) \hat f(\lambda).
\end{equation*}
For each root $\lambda_j$, we can expand $P(\lambda)$ in Taylor series around $\lambda_j$ to compute:
\begin{equation*}
    P(\lambda) \sum_{l = 1}^{d(\lambda_j)} \frac{R_{-l}(\lambda_j)}{(\lambda - \lambda_j)^l} = \sum_n \sum_{l = 1}^{d(\lambda_j)} \frac{(\lambda-\lambda_j)^{n-l}}{n!} \frac{\del^n P}{\del \lambda^n}(\lambda_j)R_{-l}(\lambda_j).
\end{equation*}
By relations \eqref{eq:prl}, the expansion of the sum in powers of $\lambda-\lambda_j$ is polynomial, that is the sum of the terms containing negative powers of $\lambda-\lambda_j$ vanishes. This yields:
\begin{equation*}
    P(\lambda) \hat u_r (\lambda) = \hat f(\lambda) - \sum_{j=1}^m \sum_{l \geq 1, ~ n-l \geq 0}  \frac{(\lambda-\lambda_j)^{n-l}}{n!} \frac{\del^n P}{\del \lambda^n}(\lambda_j) R_{-l}(\lambda_j) \hat f(\lambda)
\end{equation*}
which holds for $\lambda \neq \lambda_j$. As both sides of the equality are analytic in the variable $\lambda$, this is in fact true for all $\lambda$ contained in a neighbourhood of the real line in $\C$. We can therefore take the inverse Fourier transform to obtain:
\begin{equation}    \label{eq:purcomp}
    P(D_t) u_r(t) = f(t) - \sum_{j=1}^m \sum_{l \geq 1, ~ n-l \geq 0}  e^{i \lambda_j t} D_t^{n-l} \left [ \frac{1}{n!} \frac{\del^n P}{\del \lambda^n}(\lambda_j) R_{-l}(\lambda_j) e^{-i \lambda_j t} f(t) \right ].
\end{equation}
    
Next, we compute $P(D_t) u_s$. Let us remark that for $n < l$ we have the following identity:
\begin{equation}        \label{eq:ninfl}
    e^{i\lambda t} D_t^n e^{-i \lambda t} H_{l,\lambda} = (D_t-\lambda)^n H_{l,\lambda} = H_{l-n, \lambda}
\end{equation}
and for $n = l$, we have:
\begin{equation}        \label{eq:neql}
    e^{i\lambda t} D_t^l e^{-i \lambda t} H_{l,\lambda} = \delta
\end{equation}
where $\delta$ here is a Dirac mass centered at $t = 0$. Writing $P(D_t) = e^{i\lambda_j t} P_{\lambda_j}(D_t) e^{-i \lambda_j t}$ where $P_{\lambda_j}(D_t)$ is the operator defined in \S\ref{subsection:polysol}, we have
\begin{equation*}
    P \sum_{l=1}^{d(\lambda_j)} H_{l,\lambda_j} * (R_{-l}(\lambda_j) f) = \sum_{n \geq 0} \sum_{l=1}^{d(\lambda_j)} e^{i \lambda_j t} D_t^n e^{-i \lambda_j t} H_{l, \lambda_j} * \left [ \frac{1}{n!} \frac{\del^n P}{\del \lambda^n}(\lambda_j) R_{-l}(\lambda_j) f \right ].
\end{equation*}
If we split the sum into two parts, we see using \eqref{eq:weakderf} and \eqref{eq:neql} that the sum of the terms for which $n \geq l$ is equal to:
\begin{equation}        \label{eq:splitsum}
    \sum_{n \geq l} \sum_{l=1}^{d(\lambda_j)} e^{i\lambda_j t} D_t^{n-l} \left [ \frac{1}{n!} \frac{\del^n P}{\del \lambda^n}(\lambda_j) R_{-l}(\lambda_j) e^{-i \lambda_j t} f(t) \right ] .
\end{equation}
On the other hand, the sum of the terms for which $n < l$ can be computed using \eqref{eq:weakderf} and \eqref{eq:ninfl}, and in fact this sum vanishes by \eqref{eq:prl}:
\begin{equation}
    \sum_{n < l} \sum_{l=1}^{d(\lambda_j)} H_{l-n,\lambda_j} * \left [ \frac{1}{n!} \frac{\del^n P}{\del \lambda^n}(\lambda_j) R_l(\lambda_j) f\right ] = 0 .
\end{equation}
Comparing with \eqref{eq:purcomp}, this proves that $Pu =f$. Once we prove that $u_s$ is in $L^p(I,L^2(E))$ for any finite interval $I \subset \R$, it will follow from Proposition \ref{prop:lptwoaprioriest} that $u$ is $W^{k,p}_{loc}$. Thus we have a well-defined map $Q = Q_r + Q_s : L^p_{c}(\underline F) \rightarrow W^{k,p}_{loc}(\underline E)$ which is a right inverse for $P$.
    
\paragraph{} It remains to prove the estimates of Theorem \ref{thm:localsol}. Let $T \geq 1$ and $f \in L^p_c(\underline F)$ with essential support contained in $(-T,T) \times X$. For $-T-1 \leq t \leq T+1$, we can write:
\begin{align*}
    u_s(t) & = \sum_{j=1}^m \sum_{l=1}^{d(\lambda_j)} \int_0^\infty \frac{i^l \tau^{l-1}}{(l-1)!} e^{ i \lambda_j \tau} R_{-l}(\lambda_j) f(t-\tau) d\tau \\
    & = \sum_{j=1}^m \sum_{l=1}^{d(\lambda_j)} \int_0^{2T+1} \frac{i^l \tau^{l-1}}{(l-1)!} e^{ i \lambda_j \tau} R_{-l}(\lambda_j) f(t-\tau) d\tau
\end{align*}
In particular the $L^2$-norm of $u_s(t)$ is bounded by:
\begin{align*}
    \|u_s(t)\|_{L^2(E)} & \leq \sum_{j=1}^m \sum_{l=1}^{d(\lambda_j)} \int_0^{2T+1} \frac{\tau^{l-1}}{(l-1)!} \| R_{-l}(\lambda_j) f(t-\tau)\|_{L^2(F)} d\tau  \\
    & \leq C \int_{-\infty}^{+\infty} \chi_{(0,2T+1)}(\tau)\left( \sum_{j=1}^m \sum_{l=1}^{d(\lambda_j)} \frac{\tau^{l-1}}{(l-1)!} \right) \| f(t-\tau)\|_{L^2(F)} d\tau
\end{align*}
where $C$ is a constant depending only on the maps $R_{-l}(\lambda_j)$ and $\chi_{(0,2T+1)}$ is the characteristic function of the interval $(0,2T+1)$. The function $\chi_{(0,2T+1)}(\tau)\left( \sum_{j=1}^m \sum_{l=1}^{d(\lambda_j)} \frac{\tau^{l-1}}{(l-1)!} \right)$ is $L^1$, and since $d$ is the maximum of the $d(\lambda_j)$ it has $L^1$-norm bounded by $CT^d$ for some constant $C > 0$ that does not depend on $T \geq 1$. Thus, as a function of the variable $t \in (-T-1,T+1)$, the function $\|u_s(t)\|_{L^2(E)}$ is bounded above by the convolution of the $L^1$-function $\chi_{(0,2T+1)}(\tau)\left( \sum_{j=1}^m \sum_{l=1}^{d(\lambda_j)} \frac{\tau^{l-1}}{(l-1)!} \right)$ and the $L^p$-function $\| f(\tau) \|_{L^2}$, and therefore Young's inequality yields:
\begin{equation*}
    \|u_s\|_{L^p((-T-1,T+1),L^2(E))} \leq C T^d \|f\|_{L^p(\R,L^2(E))} .
\end{equation*}
Consequently, the restriction of $u=u_s+u_r$ to $(-T-1,T+1) \times X$ is in $L^p((-T-1,T+1),L^2(E))$ for any $T \geq 1$. By Proposition \ref{prop:lptwoaprioriest}, $\left. u \right|_{(-T,T) \times X} \in W^{k,p}(\underline E_T)$ and:
\begin{align*}
    \|u\|_{W^{k,p}(\underline E_T)} & \leq C (\|f\|_{L^p(\underline F_{T+1})} + \|u\|_{L^p((-T-1,T+1),L^2(E))} )  \\
    & \leq C (\|f\|_{L^p(\underline F)} + \|u_r\|_{L^p(\R,L^2(E))} + \|u\|_{L^p((-T-1,T+1),L^2(E))}) \\
    & \leq C (\|f\|_{L^p(\underline F)} + \|f\|_{L^p(\R,L^2(F))}) + CT^d\|f\|_{L^p(\R,L^2(F))})  \\
    & \leq CT^d\|f\|_{L^p(\underline F)}
\end{align*}
which holds since the $L^p(\underline E)$-norm of $f$ controls its $L^p(\R,L^2(F))$-norm. Moreover, we can apply the same argument for any arbitrarily large $T^\prime \geq T$ to deduce that $u \in W^{k,p}_{loc}(\underline E)$. This finishes the proof of Theorem \ref{thm:localsol} in the case where $p \geq 2$. The case $1 < p < 2$ can be treated by duality, since the formal adjoint $P^*$ is also a translation-invariant elliptic operator, and the maximal order of the real roots of $P^*$ is also $d$.

\paragraph{} In the remainder of this part shall comment on the asymptotic behaviour of the solutions we constructed above. Let $f \in L^p_c(\underline F)$ and let $u$, $u_s$ and $u_r$ be defined as above. Assume that the essential support of $f$ is contained in $(-T,T) \times X$. Then, outside of this compact set we have:
\begin{equation*}
    Pu_s = 0 = Pu_r.
\end{equation*}
As $Pu = 0$ in this domain it suffices to show that $Pu_s = 0$. This is a consequence of \eqref{eq:splitsum} which yields:
\begin{equation*}
    Pu_s = \sum_{n \geq l} \sum_{l=1}^{d(\lambda_j)} e^{i\lambda_j t} D_t^{n-l} \left [ \frac{1}{n!} \frac{\del^n P}{\del \lambda^n}(\lambda_j) R_{-l}(\lambda_j) e^{-i \lambda_j t} f(t) \right ] .
\end{equation*}
For $|t| > T$ the expression under brackets vanishes identically, which proves our claim.
    
An important consequence of this fact is that $u_r$ has exponential decay as $|t| \rightarrow \infty$, in the sense that the $W^{k,p}$-norm of $e^{\delta \rho} u_r$ is finite for some $\delta > 0$, where $\rho$ denotes an arbitrary smooth function on $Y$ equal to $|t|$ when $|t| \geq 1$.  This can be seen as a particular case of Lockhart--McOwen theory (see \S\ref{subsection:lmt}). On the other hand, it is easy to see from its definition that $u_s$ vanishes identically in the domain $\{ t < - T \}$, and more interestingly, $u_s$ is equal to the restriction of a polyhomogeneous solution of $Pu = 0$ in the domain $\{t > T \}$. Indeed for $t > T$ \eqref{eq:defus} reads:
\begin{equation*}
    u_s(t) = \sum_{j=1}^m \sum_{l=1}^{d(\lambda_j)} i^l e^{i \lambda_j t} \int_{-T}^T \frac{(t-\tau)^{l-1}}{(l-1)!} e^{- i \lambda_j \tau} R_{-l}(\lambda_j) f(\tau) d\tau
\end{equation*}
which is manifestly polyhomogeneous. Let us denote the right-hand-side $u_f \in \mE$. We may use the pairing $( \cdot , \cdot )$ introduced in \S\ref{subsection:polysol} characterise $u_f$ by duality:
   
\begin{lem}      \label{lem:pairing}
    With the above notations, $(u_f,v) = \langle f, v \rangle$ for any $v \in \mE^*$.
\end{lem}
   
\begin{proof}
    Let $\chi$ be a smooth function such that $\chi \equiv 1$ in $(-\infty, 0]$ and $\chi \equiv 0$ in $[1,\infty)$, and let $\chi_\tau(t) = \chi(t-\tau)$. For any $\tau > T$ we have the equality $\langle \chi_\tau P u, v \rangle = \langle f , v \rangle$. On the other hand, let us prove that $\langle P \chi_\tau u, v \rangle = 0$ for any $\tau \in \R$. If $\tau, \tau' \in \R$, we have:
   \begin{align*}
       \langle P \chi_\tau u, v \rangle - \langle P \chi_{\tau'} u, v \rangle & = \langle P (\chi_\tau- \chi_{\tau'}) u, v \rangle \\
       & = \langle (\chi_\tau- \chi_{\tau'}) u, P^*v \rangle = 0
   \end{align*}
   where the integration by parts is justified because $\chi_{\tau}-\chi_{\tau'}$ has compact support. Therefore the value of $\langle P \chi_\tau u, v \rangle = 0$ does not depend on $\tau$. Hence we may send $\tau$ to $-\infty$, and as $u(t)$ has exponential decay as $t \rightarrow - \infty$ we obtain $\langle P \chi_\tau u, v \rangle = 0$.
   
   It follows that $\langle f , v \rangle = - \lim_{\tau \rightarrow \infty} \langle [P,\chi_\tau]u, v \rangle$. Given the exponential decay of $u_r(t)$ and its $k$ first derivatives (in the sense explained above) as $t \rightarrow \infty$, this yields:
   \begin{equation*}
       \langle f , v \rangle = - \lim_{\tau \rightarrow \infty} \langle [P,\chi_\tau]u_s, v \rangle = \lim_{\tau \rightarrow \infty} \langle [P,1- \chi_\tau]u_s, v \rangle = (u_f, v)
   \end{equation*}
   as claimed.
\end{proof}
   
\begin{ex}       \label{ex:fvproduct}
    Consider the case of the Laplacian $\Delta_Y$ acting on $q$-forms. The singular part of the resolvent is $\lambda^{-2}p_h$, where $p_h$ is the projection on the space of harmonic forms. Thus if $\eta$ is a $q$-form on $Y$ supported in $[-T,T] \times X$ and we denote by $\xi(\tau)$ the $L^2$-projection of $\eta_\tau$ onto the space of harmonic forms, and write $\xi(\tau) = \alpha(\tau) + dt \wedge  \beta(\tau)$, we have by definition:
    \begin{align*}
        u_\eta(t) & = i^2 \int_{-T}^T (t-\tau) \xi(\tau) d\tau \\
        & = \int_{-T}^T \tau \alpha(\tau) + dt \wedge \tau \beta(\tau) d\tau - t \int_{-T}^T \alpha(\tau) + dt \wedge \beta(\tau) d\tau.
    \end{align*}
    For any $v(t) = \alpha_0 + dt \wedge \beta_0 + t(\alpha_1 + dt \wedge \beta_1) \in \mE_q$ we can use the formula of Example \ref{ex:pairing} to check:
    \begin{equation*}
        (u_\eta, v) = \int_{-T}^T \tau (\langle \alpha(\tau), \alpha_1 \rangle  + \langle \beta(\tau), \beta_1 \rangle ) + \langle \alpha(\tau), \alpha_0 \rangle  + \langle \beta(\tau), \beta_0 \rangle  d\tau = \langle \eta, v \rangle .
    \end{equation*}
\end{ex}

\section{Construction of solutions}		\label{section:construction}

In this section we explain the main construction of this paper. In \S\ref{subsection:lmt} we review the mapping properties of adapted operators on EAC manifolds. In \S\ref{subsection:characteristic} we explain our method for constructing approximate solutions of the equation $P_Tu=f$ and show that it can be reduced to a finite-dimensional linear system. In \S\ref{subsection:construction} we prove that this system admits a solution if and only if $f$ is orthogonal to the substitute cokernel defined in \S\ref{subsection:substk}. Under the restricting assumption that the indicial operator of the gluing problem has only one real root, this enables us to prove Theorem \ref{thm:main}. We also discuss other possible conditions which would yield the same result.

    \subsection{Analysis on EAC manifolds}      \label{subsection:lmt}

\paragraph{} The mapping properties of adapted operators on EAC manifolds have been studied by Lockhart--McOwen in \cite{lockhart1985elliptic}, and we will give a brief review of their theory. The right functional spaces to consider in this situation are weighted Sobolev spaces. Let $(Z,g)$ be an EAC manifold asymptotic to a cylinder $Y = \R \times X$ at infinity, $(E,h,\nabla)$ an adapted bundle, and pick a cylindrical coordinate function $\rho : Z \rightarrow \R_{>0}$. If $u$ is a smooth compactly supported section of $E$, we can define its $W^{l,p}_{\nu}$-norm ($p \geq 1$, $l \in \N$, $\nu \in \R$) as follows:
\begin{equation*}
    \| u \|_{W^{l,p}_\nu} = \sum_{ j = 0}^l \| e^{\nu \rho} \nabla^j u \|_{L^p}.
\end{equation*}
Note that for $\nu = 0$ this is just the usual $W^{l,p}$ norm. The weighted Sobolev space $W^{l,p}_\nu(E)$ can be defined as the completion of $C^{\infty}_c(E)$ with respect to the $W^{l,p}_\nu$-norm. We also denote by $C^\infty_\nu(E)$ the space of smooth sections of $E$ that have all derivatives bounded by $O(e^{-\nu \rho})$.

\paragraph{} Let $P$ be an adapted elliptic differential operator $P : C^\infty(E) \rightarrow C^\infty(F)$ of order $k$, and let $P_0 : C^\infty(\underline E_0) \rightarrow C^\infty(\underline F_0)$ be its indicial operator. The maps $W^{k+l,p}_\nu(E) \rightarrow W^{l,p}_\nu(F)$ induced by $P$ are bounded linear operators. Moreover, combining the estimates of Proposition \ref{prop:lpaprioriest} with standard interior elliptic estimates, we obtain a priori estimates of the form:
\begin{equation*}
    \| u \|_{W^{k+l,p}_\nu} \leq C \left( \| Pu\|_{W^{l,p}_\nu} + \|u\|_{L^p_\nu} \right).
\end{equation*}
Contrary to the compact case, these estimates are not enough to ensure that the maps induced by $P$ on weighted spaces are Fredholm, essentially because of the failure of the Sobolev embedding theorem between spaces with the same weight. Nevertheless, Lockhart--McOwen showed that the Fredholm property holds if and only if $\nu$ is not an indicial root of $P_0$ (see \ref{par:roots}). When $\nu \notin \mathscr D_{P_0}$ we denote by $\ind_\nu(P)$ the index of the maps induced by $P$ on spaces of weight $\nu$. The following theorem summarises the mapping properties of $P$ in weighted spaces and the index change formula as proved in \cite{lockhart1985elliptic}.

\begin{thm}	\label{thm:indexchange}
    Let $p > 1$ and $l \in \N$. Then the following holds.
    \begin{enumerate}[(i)]
	   \item The maps $W^{k+l,p} _\nu(E) \rightarrow W^{l,p}_\nu(F)$ induced by $P$ are Fredholm if and only if $\nu \notin \mathscr D_{P_0}$. In that case, the image of $P$ is the $L^2$-orthogonal complement of $\ker P^* \cap C^\infty_{-\nu}(F)$.
	   \item If $\nu < \nu'$ are not indicial roots of $P$, then the index change is given by
	\begin{equation*}
	    \ind_{\nu}(P) - \ind_{\nu'}(P) = \sum_{\nu < \im \lambda < \nu'} \dim \mE_{\lambda} .
	\end{equation*}
    \end{enumerate}
\end{thm}

By elliptic regularity, the solutions of the homogeneous equation $Pu = 0$ are smooth. An important property of solutions with sub-exponential growth is that they have a polyhomogeneous expansion at infinity. More precisely, if $0 < \nu' - \nu < \mu$ and $\nu, \nu \notin \mathscr D_P$, then the following holds. For any $u \in C^\infty_\nu$ such that $Pu=0$, there exists $u' \in C^\infty_{\nu'}$ such that when $\rho \rightarrow \infty$, the difference $u-u'$ is an element of $\bigoplus_{\nu < \im \lambda < \nu'} \mE_{\lambda}$ under the usual identification of the domain $\{\rho > 1 \}$ with the cylinder $(1, \infty) \times X$.

From now on, let us assume that $0$ is an indicial root of $P_0$, and let:
\begin{equation}
    \sigma = \min \{ \mu, \min_{\nu \in \mathscr D_{P_0} \backslash \{0\}} | \nu | \}
\end{equation}
Take any $\delta \in (0, \sigma)$. Recall that we defined $\mK$ as the kernel of $P$ acting on sections with sub-exponential growth, and $\mK_0$ the kernel of $P$ acting on decaying sections. In particular, $\mK$ is the kernel of $P$ acting on $W^{k,p}_{-\delta}(E)$ and $\mK_0$ the kernel of the action of $P$ on $W^{k,p}_{\delta}(E)$. In \S\ref{subsection:substk} we defined a map $\kappa : \mK \rightarrow \mF$ such that any element $v \in \mK$ is asymptotic to $\kappa(v)$. Hence $\mK_0$ is the kernel of $\kappa$. Similarly we defined $\mK^*$, $\mK_0^*$ and $\kappa^* : \mK^* \rightarrow \mF^*$. Let us point out that the index change formula in Theorem \ref{thm:indexchange} implies:
\begin{equation}		\label{eq:dimkappa}
    \dim \image \kappa + \dim \image \kappa^* = \dim \mF.
\end{equation}

\paragraph{} We want to study equations of the type $Pu = f$ when $f$ has exponential decay, say $f \in L^p_\delta$. By Theorem \ref{thm:indexchange}, the obstructions to solve this equation for $u \in W^{k,p}_{\delta}$ lie in $\mK^*$, whereas the obstructions to solve it in $W^{k,p}_{-\delta}$ lie in $\mK^*_0$. Here, we want to use the pairing defined in \S\ref{subsection:polysol} to give a precise description of these obstructions and of the asymptotic behaviour of solutions in $W^{k,p}_{-\delta}$.

Let $v \in \mK^*$ be asymptotic to $\kappa^*(v) = v_0 \in \mE^*$ and consider $u \in C^\infty(E)$ asymptotic to $u_0 \in \mE$, such that $u - u_0$ and all their derivatives are exponentially decaying as $\rho \rightarrow \infty$. The $L^2$ product $\langle Pu, v \rangle$ is well-defined as $Pu$ decays exponentially. It turns out that its value only depends on the asymptotic data. More precisely we claim that:

\begin{lem}		\label{lem:valuepuv}
    With the above notations, $\langle Pu, v \rangle = (u_0, v_0)$.
\end{lem}
  
\begin{proof}
    Let $\chi : \R \rightarrow \R$ be a smooth function such that $\chi \equiv 0$ in $(-\infty, 0]$ and $\chi \equiv 1$ in $[1,\infty)$, and let $\chi_\tau(t) = \chi(t-\tau)$ for $\tau \in \R$. Then for any $\tau \geq 1$ we have:
    \begin{align*}
        \langle Pu, v \rangle &= \langle P \chi_\tau(\rho) u, v \rangle + \langle P(1-\chi_\tau(\rho)) u, v \rangle \\
        &= \langle P \chi_\tau(\rho) u, v \rangle + \langle (1-\chi_\tau(\rho)) u, P^* v \rangle \\
        &= \langle P\chi_\tau(\rho) u, v \rangle 
    \end{align*}
    since $P^*v = 0$. Thus $\langle Pu, v \rangle = \lim_{\tau \rightarrow \infty} \langle P\chi_\tau(\rho) u, v \rangle$. As $u-u_0$, $v-v_0$ and the coefficients of $P-P_0$ decay exponentially as $\rho \rightarrow \infty$, as well as all derivatives, this implies:
    \begin{equation*}
        \langle Pu, v \rangle = \lim_{\tau \rightarrow \infty} \langle P_0\chi_\tau u_0, v_0 \rangle = (u_0,v_0)
    \end{equation*}
    since $\langle P_0\chi_\tau u_0, v_0 \rangle = (u_0,v_0)$ for any $\tau \in \R$.
\end{proof}
  
As a consequence of this lemma, $\image \kappa$ and $\image \kappa^*$ are orthogonal for the pairing $( \cdot, \cdot)$. Together with equality \eqref{eq:dimkappa}, this implies that $\image \kappa$ is exactly the orthogonal space of $\image \kappa^*$ for the pairing $( \cdot, \cdot)$. 

Let us denote by $\mK^*_+$ the subspace of $\mK$ orthogonal to $\mK^*_0$ for the $L^2$-product, so that $\kappa^*$ induces an isomorphism between $\mK^*_+$ and $\image \kappa^*$. We also choose an arbitrary complement $\mF_0$ of $\image \kappa$ in $\mE$. Let $m = \dim \image \kappa^*$. Pick smooth sections $h_1, \ldots ,h_m$ which are asymptotic to a basis of $\mF_0$ at infinity, with the difference and all their derivatives exponentially decaying, and denote by $\mF \subset C^\infty(E)$ the vector space they span. By Lemma \ref{lem:valuepuv} we may choose a basis $g_1, \ldots ,g_m$ of $\mK^*_+$ such that $\langle Ph_i, g_j \rangle = \delta_{ij}$  for all $1 \leq i,j \leq m$.

Let $f \in L^p_\delta$ be a section of $F$, and $w$ be the $L^2$-projection of $f$ onto $\mK^*_0$. Let us write:
\begin{equation*}
    f = f^\prime + \sum_{j = 1}^m \langle f, g_j \rangle P h_j + w.
\end{equation*}
where $f^\prime \in L^p_{\delta}$ is  by construction orthogonal to the obstruction space $\mK^*$. As $| \langle f, g \rangle | \leq C \| f \|_{L^p_\delta} \| g \|_{L^q_{-\delta}}$ for any $g \in \mK^*$, where $q$ is the conjugate exponent of $p$ and $C >0$ is some constant, we have $\| f^\prime \|_{L^p_{\delta}} \leq C^\prime \| f \|_{L^p_{\delta}}$ for some universal constant $C^\prime > 0$. By Theorem \ref{thm:indexchange} there exists $u^\prime$ such that $Pu^\prime = f^\prime$ and $\| u^\prime \|_{W^{k,p}_{\delta}} \leq C^{\prime\prime} \|f^\prime\|_{L^p_\delta}$. This proves the following:

\begin{prop}		\label{prop:lmeps}
    For any $p > 1$ and $0 < \delta < \sigma$, there exists a constant $C > 0$, depending only on $p$ and  $\delta$, such that the following holds. Let $f \in L^p_\delta(F)$, and let $w$ by its $L^2$-projection onto $\mK_0$. Then there exists a section  $u^\prime \in W^{k,p}_{\delta}(E)$ with $\| u^\prime \|_{W^{k,p}_{\delta}} \leq C \| f\|_{L^p_\delta}$ and such that
    \begin{equation*}
        P \left( u^\prime +  \sum_{j=1}^m \langle f, g_j \rangle h_j \right) = f - w.
    \end{equation*}
\end{prop}

    \subsection{Characteristic system}      \label{subsection:characteristic}

\paragraph{} In the same setup as Section \ref{section:setup}, we now consider the gluing problem of two adapted operators $P_1$, $P_2$ of order $k$ on EAC manifolds $Z_1$, $Z_2$. For the present discussion there are no restrictions on the real roots of the indicial operator $P_0$. By definition, there is a compact $K_1 \subset Z_1$ and an orientation-preserving diffeomorphism $\phi_1 : (0, \infty) \times X \rightarrow Z_1 \backslash K_1$, and we picked a positive cylindrical coordinate function $\rho_1$ on $Z_1$ such that $\rho_1(\phi_1(t,x)) = t$ when $t \geq 1$ and $\rho_1 < 1$ everywhere else in $Z_1$. As in Section \ref{section:setup} we fix a cutoff function $\chi : \R \rightarrow [0,1]$ such that $\chi \equiv 0$ in $(-\infty, - \frac 1 2 ]$ and $\chi \equiv 1$ in $[\frac 1 2, \infty)$. For $\tau \in \R$ we keep our usual notation $\chi_\tau(t)= \chi(t-\tau)$. It will be convenient to introduce a family $\zeta^1_\tau : Z_1 \rightarrow [0,1]$ of cutoff functions for the construction. For $\tau \geq 0$ define:
\begin{equation*}
    \zeta^1_\tau(z) =
    \begin{cases}
        0 & \text{if} ~~ z \in K_1 \\
        \chi(t-\tau - \frac 12) & \text{if} ~~ z = \phi_1(t,x), ~ (t,x) \in (0,\infty) \times X
    \end{cases}.
\end{equation*}
We similarly define a family of cutoff functions on $Z_2$, denoted by $\zeta^2_\tau$ for $\tau \geq 0$. Consider now the compact manifold $M_T$ obtained by gluing the compact domains $\{\rho_1 \leq T+2 \} \subset Z_1$ and $\{\rho_2 \leq T+2 \} \subset Z_2$ along the annulus $\{T \leq \rho_i \leq T+2 \} $. We can define a family of cutoff functions $\zeta_\tau : M_T \rightarrow [0,1]$ for $0 \leq \tau \leq T$ by patching together $\zeta^1_\tau$ with $\zeta^2_\tau$ in the following way:
\begin{equation*}
    \zeta_\tau \equiv
    \begin{cases}
        \zeta^1_\tau & \text{if} ~ \rho_T \leq 0 \\
        \zeta^2_\tau & \text{if} ~ \rho_T \geq 0
    \end{cases}.
\end{equation*}
Note that the support of $\zeta_\tau$ is diffeomorphic to the finite cylinder $[-T-1+\tau, T+1-\tau] \times X$.

We now turn to the gluing problem of two adapted operators $P_i : C^\infty(E_i) \rightarrow C^\infty(F_i)$ as described in \S\ref{subsection:gluingpb}. Our goal is to prove that we can construct solutions of the equation $P_T u = f$ for $f$ taking values in a complement of the substitute cokernel introduced in \S\ref{subsection:substk}. We shall do this by considering three regions in $M_T$: the neck region $\{ |\rho_T| \leq T \}$ for which our main tool is Theorem \ref{thm:localsol}, and the two compact regions $\{\rho_T \leq 0 \}$ and $\{ \rho_T \geq 0 \}$, for which we will use weighted analysis on $Z_1$ and $Z_2$ in the form of Proposition \ref{prop:lmeps}. The crucial point of the construction is to understand the interactions between these three regions, especially in terms of the obstructions to solving the equation $P_i u = f$ on each $Z_i$. Using the pairing $(\cdot,\cdot)$ defined in \S\ref{subsection:polysol} in order to implicitly keep track of these obstructions, we will be able to essentially reduce this problem to a finite-dimensional linear system.
 
\paragraph{} From now on we fix $p > 1$ and work with Sobolev spaces $W^{l,p}$. Let $f \in L^p(F_T)$ be an arbitrary section. We may identify the section $\zeta_1 f$ with a section of the translation-invariant vector bundle $\underline F_0$ over the cylinder $Y = \R \times X$, which we denote by $f_0$. Moreover, the essential support of $f_0$ is contained in the finite cylinder $[-T,T] \times X$. Note that the Sobolev norm of sections supported in the neck region of $M_T$ and the Sobolev norm of sections supported in the finite cylinder $[-T,T] \times X$ are equivalent. Hence, we have a bound:
\begin{equation*}
    \|f_0\|_{L^p} \leq C \|f\|_{L^p} .
\end{equation*}
By Theorem \ref{thm:localsol}, the operator $P_0$ admits a right inverse $Q_0 : L^p_{c}(\underline F_0) \rightarrow W^{k,p}_{loc}(\underline E_0)$. Thus we can define $u_0 = Q_0 f_0$, which satisfies $P_0 u_0 = f_0$. Using the cutoff function $\zeta_0$ to identify $\zeta_0 u_0$ with a section of $E_T \rightarrow M_T$, one has:
\begin{align*}
    f-P_T\zeta_0 u_0 & = f- [P_T,\zeta_0]u_0 - \zeta_0 P_T u_0 \\
    & = (1-\zeta_1) f - [P_T,\zeta_0]u_0 - \zeta_0 (P_T-P_0) u_0.
\end{align*}
Note that the section $(1- \zeta_1) f - [P_T,\zeta_0]u_0$ is supported in the compact region $\{|\rho_T| \geq T \}$. Moreover the operator $\zeta_0 (P_T-P_0)$ vanishes in the region $\{ |\rho_T | \leq \frac 1 2 \}$ so that we may write:
\begin{equation}
    f-P_T\zeta_0 u_0 = f_1 + f_2
\end{equation}
where $f_1 = \chi(\rho_T) (f-P_T\zeta_0 u_0)$ can be identified with a section of $F_1$ supported in $\{ \rho_1 \leq T+1 \} \subset Z_1$, and $f_2 = (1-\chi(\rho_T)) (f-P_T\zeta_0 u_0)$ can be identified with a section of $F_2$ over $\{ \rho_2 \leq T+1 \} \subset Z_2$. Both sections are $L^p$-bounded. 

From now on we fix some $\delta \in (0, \sigma)$. As the coefficients of $P_i - P_0$ and all their derivatives have exponential decay as $\rho_i \rightarrow \infty$, the $L^p$-norm of $f$ controls the $L^p_\delta$-norms of $f_1$ and $f_2$.  More precisely, the following estimates hold:

\begin{lem}     \label{lem:boundfi}
    Let $d$ be the maximal order of the real roots of $P_0$. Then there exists a constant $C > 0$ such that:
    \begin{equation*}
        \|f_i\|_{L^p_\delta} \leq C T^{d} \| f \|_{L^p},  ~~~~ i = 1,2.
    \end{equation*}
\end{lem} 

\begin{proof}
    Let us prove the estimate for $f_1$, which can be written as:
    \begin{equation*}
        f_1 = (1-\chi(\rho_T))((1-\zeta_1) f - [P_T, \zeta_0]u_0 - \zeta_0(P_T-P_0)u_0).
    \end{equation*}
    The term $(1-\chi(\rho_T))(1-\zeta_1) f$ is supported in the compact region $\{ \rho_1 \leq 2 \} \subset Z_1$ and therefore satisfies:
    \begin{equation*}
        \|(1-\chi(\rho_T))(1-\zeta_1) f \|_{L^p_\delta} \leq e^{2\delta} \|f\|_{L^p}
    \end{equation*}
    since the function $(1-\chi(\rho_T))(1-\zeta_1)$ is bounded by $1$. On the other hand, the second term $(1-\chi(\rho_T)) [P_T, \zeta_0]u_0$ is supported in $\{ \rho_1 \leq 1 \}$, and the $W^{k,p}$-norm of $u_0$ in the cylinder $[-T-1,T+1] \times X$ is bounded by $CT^d \|f\|$ for some constant $C$. As $\zeta_0$ and all its derivatives are uniformly bounded independently from $T$, this yields an estimate:
    \begin{equation*}
        \|(1-\chi(\rho_T)) [P_T, \zeta_0]u_0 \|_{L^p} \leq C^\prime T^d \|f\|_{L^p}.
    \end{equation*}
    For the last term $(1-\chi(\rho_T)) \zeta_0(P_T-P_0)u_0$, we can use the bound on the $W^{k,p}$-norm of $u_0$ and the exponential decay of the coefficients of $P_1-P_0$ and all their derivatives to obtain a similar bound:
    \begin{equation*}
        \|(1-\chi(\rho_T)) \zeta_0(P_T-P_0)u_0\|_{L^p} \leq C^{\prime\prime} T^d \|f\|_{L^p}.
    \end{equation*}
    These three bounds prove the lemma.
\end{proof}

\paragraph{} Next we want to understand the obstructions to solving $P_i u_i = f_i$ with $u_i \in W^{k,p}_\delta(E_i)$. Let us denote by $\langle \cdot, \cdot \rangle_0$ the $L^2$-product on the cylinder $\R \times X$ equipped with its translation-invariant metric. The key result is the following:

\begin{lem}     \label{lem:obst}
    Choose arbitrary norms on $\mK_1^*$ and $\mK_2^*$. For $T \rightarrow \infty$ the following holds. If $g_1 \in \mK_1^*$ and $g_{1,T}(t) = \kappa^*_1[g_1](t+T+1)$, then:
    \begin{equation*}
        \langle f_1, g_1 \rangle =  \langle f,(1-\chi_{T+1}(\rho_1)) g_1 \rangle - \langle (1-\chi) f_0, g_{1,T}\rangle_0  + O \left( e^{-\delta T} \|f\|_{L^p} \|g_1\| \right) .
    \end{equation*}
    If $g_2 \in \mK_2^*$ and $g_{2,T} = \kappa^*_2[g_2](t-T-1)$ then:
    \begin{equation*}
        \langle f_2, g_2 \rangle =  \langle f, (1-\chi_{T+1}(\rho_2))g_2 \rangle + \langle (1-\chi) f_0, g_{2,T}\rangle_0  + O \left( e^{-\delta T} \|f\|_{L^p} \|g_2\| \right) .
    \end{equation*}
\end{lem}

\begin{rem}
    In the statement of the lemma and in the following proof, the notation $O(e^{-\delta T} \| f \|_{L^p} \|g_i\|)$ means that there is a constant $C > 0$, depending on $p > 1$, $\delta \in (0,\sigma)$ and possibly on the choice of norms on $\mK_i^*$ but independent of $T$, $f$ and $g_i$, such that
    \begin{equation*}
        | O(e^{-\delta T} \| f \|_{L^p} \|g_i\|) | \leq C e^{-\delta T} \| f \|_{L^p} \| g_i \| .
    \end{equation*}
\end{rem}

\begin{proof}
    Notice first that for any $\tau \leq T-2$ we have:
    \begin{align*}
        \langle f_1, g_1 \rangle & =  \langle (1-\chi(\rho_T))f, g_1 \rangle - \langle (1-\chi(\rho_T)) P_T \zeta_0 u_0, g_1 \rangle \\
        & = \langle (1-\chi(\rho_T))f, g_1 \rangle - \langle (1-\chi(\rho_T)) P_T \zeta_\tau u_0, g_1 \rangle
    \end{align*}
    since $(\zeta_\tau-\zeta_0)u_0$ has support in $\{ \rho_1 \leq T-1\}$ and $P_1^* g_1 = 0$. Given the decay of the coefficients of $P_1-P_0$ we have:
    \begin{equation}
        \langle (1-\chi(\rho_T)) P_T \zeta_{T-2} u_0, g_1 \rangle = \langle (1-\chi) P_0 \chi_{-2} u_0, g_{1,T} \rangle_0 +  O \left( e^{-\delta T} \|f\|_{L^p} \|g_1\| \right).
    \end{equation}
    Moreover we have $1-\chi(\rho_T) = (1-\chi_{T+1}(\rho_1))$ with the usual identifications. Thus the equality $\langle (1-\chi(\rho_T))f, g_1 \rangle = \langle f, (1-\chi_{T+1}(\rho_1)) g_1 \rangle$ clearly holds.
  
    It remains to compute the value of $\langle (1-\chi) P_0 \chi_{-2} u_0, g_{1,T} \rangle_0$. By integration by parts, for any $\tau \leq -2$ we have:
    \begin{equation}	\label{eq:u0pchi}
        \langle (1-\chi) P_0 \chi_{-2} u_0,g_{1,T} \rangle_0 = \langle (1-\chi) P_0 \chi_{\tau} u_0, g_{1,T} \rangle_0.
    \end{equation}
    But now we can write $P\chi_{\tau} u_0 = \chi_{\tau} f_0 + [P_0,\chi_\tau] u_0$. As $u_0$ and its derivatives of order less than $k$ have exponential decay at infinity and the differential operator $[P_0,\chi_\tau]$ has uniformly bounded coefficients and is supported in $[\tau-\frac{1}{2},\tau + \frac{1}{2}] \times X$, it follows that
	\begin{equation*}
		\lim_{\tau \rightarrow - \infty} \langle (1-\chi) [P_0,\chi_\tau] u_0 , g_{1,T} \rangle_0 = 0 
	\end{equation*}
    and therefore we can send $\tau \rightarrow - \infty$ in \eqref{eq:u0pchi} and obtain:
    \begin{align*}
        \langle (1-\chi) P_0 \chi_{-2} u_0, g_{1,T} \rangle_0  & = \lim_{\tau \rightarrow - \infty} \langle (1-\chi) \chi_\tau f_0, g_{1,T} \rangle_0 \\
        & = \langle (1-\chi) f_0, g_{1,T} \rangle_0 
    \end{align*}
    This proves the first equality of Lemma \ref{lem:obst}.
  
    For the second equality we can prove as above that:
    \begin{equation}
        \langle f_2, g_2 \rangle = \langle f, \chi_{T+1}(\rho_2) g_2 \rangle - \lim_{\tau \rightarrow \infty} \langle \chi P_0 (1 - \chi_\tau) u_0, g_{2,T} \rangle_0 + O \left( e^{-\delta T} \|f\|_{L^p} \|g_2\| \right).
    \end{equation}
    Then for $\tau$ large enough we have:
    \begin{align*}
        \langle \chi P_0 (1 - \chi_\tau) u_0, g_{2,T} & \rangle_0  = \langle \chi f_0, g_{2,T} \rangle_0 + \langle \chi[P_0,1-\chi_\tau] u_0 , g_{2,T} \rangle_0 \\
        & \longrightarrow \langle \chi f_0, g_{2,T} \rangle_0 - (u_{f_0}, g_{2,T})
    \end{align*}
    as $\tau \rightarrow \infty$, where $u_{f_0} \in \mE$ is the polyhomogeneous solution defined in \S\ref{subsection:localex}. By Lemma \ref{lem:pairing}, the last term is equal to:
    \begin{equation*}
        (u_{f_0}, g_{2,T}) = \langle f_0, g_{2,T} \rangle_0.
    \end{equation*}
    The second equality follows.
\end{proof}

In the next section, it will be useful to use a variation of the above lemma for arbitrary solutions of the equation $P_0 u = f_0$. Thus let $ v \in \mE$ and define $u^\prime_0 = Q_0 f_0 + v$, and as above write:
\begin{equation}
    f- P_T \zeta_0 u^\prime_0 = f^\prime_1 + f^\prime_2
\end{equation}
where $f^\prime_i \in L^p_\delta(F_i)$. As a corollary of Lemma \ref{lem:obst}, we can describe the obstructions to solving $P_i u = f^\prime_i$ as follows.

\begin{cor}     \label{cor:obst}
    Choose arbitrary norms on $\mE$, $\mK^*_1$ and $\mK^*_2$. Then, if $g_1 \in \mK_1^*$ and $g_{1,T}(t) = \kappa^*_1[g_1](t+T+1)$ it holds:
    \begin{equation*}
        \langle f^\prime_1, g_1 \rangle =  \langle f,(1-\chi_{T+1}(\rho_1)) g_1 \rangle - \langle (1-\chi) f_0, g_{1,T}\rangle_0  - (v, g_{1,T}) + O \left( e^{-\delta T} (\|f\|_{L^p}+\|v\|) \|g_1\| \right) .
    \end{equation*}
    If $g_2 \in \mK_2^*$ and $g_{2,T} = \kappa^*_2[g_2](t-T-1)$ then:
    \begin{equation*}
        \langle f^\prime_2, g_2 \rangle =  \langle f, (1-\chi_{T+1}(\rho_2))g_2 \rangle + \langle (1-\chi) f_0, g_{2,T}\rangle_0 + (v,g_{2,T})  + O \left( e^{-\delta T} (\|f\|_{L^p}+\|v\|) \|g_2\| \right) .
    \end{equation*}
\end{cor}

\begin{rem}
    As in the previous lemma, the notation $O ( e^{-\delta T} (\|f\|_{L^p}+\|v\|) \|g_i\| )$ means that there is a constant $C > 0$, depending on $p > 1$, $\delta \in (0,\sigma)$, and possibly on the choice of norms on $\mE$, $\mK^*_1$ and $\mK^*_2$ but independent of $T$, $f$, $g_i$ and $v$ such that 
    \begin{equation*}
        | O ( e^{-\delta T} (\|f\|_{L^p}+\|v\|) \|g_i\| ) | \leq C e^{-\delta T} (\|f\|_{L^p}+\|v\|) \|g_i\| .
    \end{equation*}
\end{rem}

\paragraph{} So far, everything we did works without need to impose any conditions on the real roots of the indicial operator $P_0$. We now outline the construction which we will perform in the next part and emphasise where the restricting assumption of Theorem \ref{thm:main} come from. The general idea of our construction is to identify a subspace of $L^p(F_T)$ on which we can find approximate solutions of the equation $P_Tu = f$  with good estimates and a control on the error of the form $\| f-P_Tu\|_{L^p} \leq C e^{-\delta T} \|f\|_{L^p}$ for $T$ large enough. Once we can achieve this, we will simply use an iterative process to build exact solutions, by taking successive projections onto this good subspace. 

By taking cutoffs as above, we can solve the equation $P_0u = f_0$ on the cylinder, with a general solution of the form $u = Q_0 f_0 + v$ for some arbitrary $v \in \mE$. With the above notations, it remains to consider the equations $P_i u_i = f^\prime_i$ on the EAC manifolds $Z_1$ and $Z_2$. The idea is to choose $v$ appropriately so that all the obstructions to finding decaying solutions $u_i \in W^{k,p}_{\delta}(E_i)$ vanish, up to exponentially decaying terms. If this can be done, we just need to take cutoffs of these solutions to build an approximate solution, up to an exponentially decaying term. Using Corollary \ref{cor:obst} we have essentially reduced our linear PDE problem to the following finite-dimensional system, where the unknown is $v \in \mE$:
\begin{equation}    \label{eq:linsys}
    \begin{cases}
        (v, g_{1,T}) = \langle f,(1-\chi_{T+1}(\rho_1)) g_1 \rangle - \langle (1-\chi) f_0, g_{1,T}\rangle_0, & \forall g_1 \in \mK^*_1 \\
        (v,g_{2,T}) = - \langle f, (1-\chi_{T+1}(\rho_2)) g_2 \rangle + \langle (1-\chi) f_0, g_{2,T}\rangle_0, & \forall g_2 \in \mK^*_2
    \end{cases}
\end{equation}
where we use the notations:
\begin{equation*}
    g_{1,T}(t) = \kappa^*_1[g_1](t+T+1), ~~ \text{and} ~~ g_{2,T} = \kappa^*_2[g_2](t-T-1).
\end{equation*}
We call this system the \emph{characteristic system} of our gluing problem. There are obvious obstructions to finding a solution to this system. We need at least to impose $f$ to be orthogonal to all the sections of the form $(1-\chi_{T+1}(\rho_i)) g_i$ with $g_i \in \mK_{i,0}$, as in this case $g_{i,T} = 0$. Actually, a more careful examination of the characteristic system shows that we need $f$ to be orthogonal to the full substitute cokernel. Indeed, a pair $(g_1,g_2) \in \mK_1^* \times \mK_2^*$ is matching at $T$ if and only if $g_{1,T} = g_{2,T}$ with the above notations. Thus, if there exists $v \in \mE$ solving the system we must have
\begin{equation*}
    \langle f,(1-\chi_{T+1}(\rho_1)) g_1 + (1-\chi_{T+1}(\rho_2)) g_2 \rangle = 0
\end{equation*}
for any pair $(g_1,g_2)$ matching at $T$.

As a consequence, the substitute cokernel $\mK_T^*$ naturally arises as a space of obstructions to constructing approximate solutions of $P_Tu = f$ by our method, as it is necessary to require $f$ to be orthogonal to $\mK^*_T$ for the linear system \eqref{eq:linsys} to admit a solution. In fact, we will see that this is also a sufficient condition (Lemma \ref{lem:solchar}). Unfortunately,  the coefficients of this system vary analytically with $T$, and therefore the rank of the system might drop at some points. Furthermore, the system is generally underdetermined, with an obvious kernel formed by the subspace of $\mE$ orthogonal to all $g_{1,T}$ and $g_{2,T}$, for $g_i$ varying in $\mK_i^*$. Hence, even if the characteristic system admits a solution $v$ whenever $f$ is orthogonal to the substitute cokernel we might not be able to obtain reasonable estimates on the norm of $v$, especially near the values of $T$ at which the rank of the system drops. We shall prove that these difficulties can be avoided in the case where the indicial operator $P_0$ has only one root, which will be sufficient for our applications.

    \subsection{Main construction}      \label{subsection:construction}

\paragraph{} Let us first consider the case where $P_0$ has a single root $\lambda_0$ of order $1$, before generalising to any order. In that case, the elements of $\mE$ are of the form $e^{i \lambda_0 t} u(x)$ with $u$ translation-invariant section of $\underline E_0$, and similarly for $\mE^*$. As a consequence, the matching condition \eqref{eq:matchingu} does not really depend on $T$, up to overall factors of $e^{\pm i \lambda_0 (T+1)}$. In particular, the dimensions of $\mK_T$ and $\mK^*_T$ are independent of $T$:
\begin{equation*}
    \dim \mK_T = \dim \mK_{0,1} + \dim \mK_{0,2} + \dim (\image \kappa_1 \cap \image \kappa_2)
\end{equation*}
and similarly for the substitute cokernel $\mK_T^*$. This implies that we can uniformly bound the $L^2$-orthogonal projections onto $\mK_T$ and $\mK^*_T$:

\begin{lem}     \label{lem:lproj}
    Let $p > 1$ and $l \in \N$. Then for $T$ large enough the norm of the $L^2$-orthogonal projection of $W^{l,p}(F_T)$ onto $\mK^*_T$ is bounded from above by a uniform constant $C_1 > 0$, depending on $l$ and $p$ but independent of $T$. Similarly, the norm of the $L^2$-orthogonal projection of $W^{l,p}(E_T)$ onto $\mK_T$ is bounded from above by a uniform constant $C^\prime_1 > 0$.
\end{lem}

\begin{proof}
    This can be proved by fixing basis for $\mK_{0,1}$, $\mK_{0,2}$ and $\image \kappa_1 \cap \image \kappa_2$ and considering the corresponding basis of $\mK_T$. By Gram--Schmidt orthonormalisation, one can deduce an explicit expression for the $L^2$-projection, from which the lemma easily follows.
\end{proof}

\paragraph{} Let us now choose an arbitrary complement $\mE^*_1$ of $\image \kappa^*_1 \cap \image \kappa^*_2$ in $\image \kappa^*_1$, and a complement $\mE^*_2$ of $\image \kappa^*_1 \cap \image \kappa^*_2$ in $\image \kappa^*_2$. Thus we have a direct sum decomposition:
\begin{equation*}
    \image \kappa^*_1 + \image \kappa^*_2 = \image \kappa^*_1 \cap \image \kappa^*_2 \oplus \mE^*_1 \oplus \mE^*_2 \subset \mE^* .
\end{equation*}
Pick a complement $\mE'$ of $\image \kappa_1 \cap \image \kappa_2$ in $\mE$, so that the pairing 
\begin{equation*}
    \mE' \times (\image \kappa^*_1 + \image \kappa^*_2) \rightarrow \C 
\end{equation*}
induced by $(\cdot, \cdot)$ is non-degenerate. For $i = 1,2$ define $\mE_i = \image \kappa_i \cap \mE'$. Then the pairings:
\begin{equation}        \label{eq:eipairing}
    \mE_1 \times \mE^*_2 \rightarrow \C ~~ \text{and} ~~ \mE_2 \times \mE_1^* \rightarrow \C
\end{equation}
induced by $(\cdot,\cdot)$ are non-degenerate. Indeed, if $u \in \mE_1$ is orthogonal to $\mE^*_2$ then it is orthogonal to $\image \kappa^*_1 + \image \kappa^*_2$ and therefore belongs to $\image \kappa_1 \cap \image \kappa_2$, which means that $u = 0$ as it is an element of $\mE'$. On the other hand if $v \in \mE_2^*$ is orthogonal to $\mE_1$, then it is orthogonal to $\image \kappa_1$ and to $\image \kappa_2$, which means that it belongs to $\image \kappa^*_1 \cap \image \kappa^*_2$ and thus $v = 0$ by definition of $\mE_2^*$. Therefore, if we define $\mE_0$ as the orthogonal space of $\mE_1^* \oplus \mE_2^*$ in $\mE'$ for the above pairing, we have a direct sum decomposition:
\begin{equation*}
    \mE' = \mE_0 \oplus \mE_1 \oplus \mE_2.
\end{equation*}
This implies that the pairing 
\begin{equation}    \label{eq:ezeropairing}
    \mE_0 \times (\image \kappa^*_1 \cap \image \kappa^*_2) \rightarrow \C
\end{equation}
is non-degenerate. These conventions will be useful to put the system \eqref{eq:linsys} in a more tractable form. For definiteness we prefer to work in a complement of $\image \kappa_1 \cap \image \kappa_2$, as this is the kernel of the system. This allows us to prove the following:

\begin{prop}        \label{prop:approxsolv}
    Let $p > 1$ and $\delta \in (0,\sigma)$, and recall that we denoted the order of the elliptic operator $P_T$ by $k$. Then there exist constants $C_2, C_3 > 0$, depending only on $p$ and $\delta$, such that for $T$ large enough the following holds. If $f \in L^p(F_T)$ is $L^2$-orthogonal to $\mK^*_T$, then there exists $u \in W^{k,p}(E_T)$ $L^2$-orthogonal to $\mK_T$ such that:
    \begin{equation*}
        \|f-P_T u\|_{L^p} \leq C_2 e^{- \delta T}\|f\|_{L^p} ~~ \text{and} ~~ \|u\|_{W^{k,p}} \leq C_3 T \|f\|_{L^p} .
    \end{equation*}
\end{prop}

\begin{proof}
    For $i=1,2$, let us fix subspaces $\mF_i \subset C^\infty(E_i)$ as in Proposition \ref{prop:lmeps}. The orthogonal space $\mK_{i,+}^*$ of $\mK_{i,0}$ in $\mK_i$ is isomorphic to $\image \kappa^*_i$, so that the decomposition $\image \kappa^*_i = \image \kappa^*_1 \cap \image \kappa^*_2 \oplus \mE_i^*$ induce a corresponding decomposition $\mK_{i,+}^* = \mK^*_{i,m} \oplus \mK^*_{i,\perp}$ (the subscript $m$ stands for matching).

    If $f \in L^p(F_T)$ is orthogonal to the substitute cokernel, we can use the above decompositions of $\mE$ amd $\mE^*$ to put the system \eqref{eq:linsys} in the form:
    \begin{equation*}
        \begin{cases}
            (v_0, \kappa^*_1[g_0]) = e^{-i\lambda_0(T+1)} \langle f,(1-\chi_{T+1}(\rho_1)) g_0 \rangle - \langle (1-\chi)f_0, \kappa^*_1[g_0] \rangle_0, & \forall g_0 \in \mK^*_{1,m} \\
            (v_1, \kappa^*_1[g_1]) = e^{-i\lambda_0 (T+1)} \langle f,(1-\chi_{T+1}(\rho_1)) g_1 \rangle - \langle (1-\chi) f_0, \kappa^*_1[g_1] \rangle_0, & \forall g_1 \in \mK^*_{1,\perp} \\
            (v_2, \kappa^*_2[g_2]) = - e^{i \lambda_0(T+1)} \langle f,(1-\chi_{T+1}(\rho_2)) g_2 \rangle + \langle (1-\chi) f_0, \kappa^*_2[g_2] \rangle_0, & \forall g_2 \in \mK^*_{2,\perp}
        \end{cases}
    \end{equation*}
    where we decompose any element $v \in \mE'$ as $v= v_0 + v_1 + v_2 \in \mE_0 \oplus \mE_1 \oplus \mE_2$ and the factors $e^{\pm i \lambda_0 (T+1)}$ come from $\kappa^*_{1,T} = e^{i \lambda_0(T+1)} \kappa^*_1$ and $\kappa_{2,T} = e^{-i \lambda_0(T+1)}\kappa^*_2$. Non-degeneracy of the pairings \eqref{eq:eipairing} and \eqref{eq:ezeropairing} implies that this is of the form:
    \begin{equation*}
        A v = b_T(f) \in \R^N
    \end{equation*}
    where $N = \dim \mE' = \dim \image \kappa_1^* + \dim \image \kappa_2^*$ and $A : \mE' \rightarrow \R^N$ is an invertible linear map which does not depend on $T$. Thus there is a unique solution $v = A^{-1} b(f)$, and if we fix norms on $\mE'$ and $\R^N$ we have a uniform bound:
    \begin{equation*}
        \| v \| \leq C \| b_T(f) \|.
    \end{equation*}
    As the elements of $\mK_1^*$ and $\mK^*_2$ are bounded in $C^0$ norm, each of the sections $(1-\chi_{T+1}) g_i$ have $L^q$-norm bounded by $C T^{\frac 1 q} \| g_i \|$, where $q$ is the conjugate exponent of $p$. Thus, we can deduce that the norm of $b_T(f)$ satisfies a bound of the form:
    \begin{equation*}
        \| b_T(f) \| \leq C' T^{\frac 1 q} \|f\|_{L^p} .
    \end{equation*}
    Hence $\| v \| \leq C'' T^{\frac 1 q} \|f\|_{L^p}$ for a constant independent of $T$.

    Following the idea outlined in the previous part, let us write:
    \begin{equation*}
        f- P_T(\zeta_0 Q_0 f_0 + \zeta_0 v) = f_1 + f_2
    \end{equation*}
    with $f_1 \in L^p(F_1)$ and $f_2 \in L^p(F_2)$, each of the sections $f_i$ being supported in the domain  $\{ \rho_i \leq T+ 2 \} \subset Z_i$. By Theorem \ref{thm:localsol}, we have estimates:
    \begin{equation*}
        \| \zeta_0 Q_0 f_0 \|_{W^{k,p}} \leq C T \|f_0 \|_{L^p} \leq C^\prime T \| f \|_{L^p}.
    \end{equation*}
    Further, as $v$ is uniformly bounded and $\zeta_0 v$ has support in a domain equivalent to a finite cylinder $[-T-1,T+1] \times X$ we have a bound:
    \begin{equation*}
        \| \zeta_0 v \|_{W^{k,p}} \leq C T^{\frac 1 p} \| v\| \leq C' T \|f\|_{L^p}.
    \end{equation*}
    As in the proof of Lemma \ref{lem:boundfi}, we can use the uniform bound on $v$ to prove that the weighted norms of $f_1$ and $f_2$ satisfy bounds:
    \begin{equation*}
        \| f_i \|_{L^p_\delta} \leq CT \|f\|_{L^p}.
    \end{equation*}
    We now consider the equations $P_1 u_1 = f_1$ on $Z_1$ and $P_2 u_2 = f_2$ on $Z_2$. By Proposition \ref{prop:lmeps}, there exist $w_i \in \mK^*_{i,0}$, $h_i \in \mF_i$ and $u_i \in W^{k,p}_\delta(E_i)$ such that:
    \begin{equation*}
        P_i(u_i + h_i) = f_i - w_i. 
    \end{equation*}
    Moreover, our choice of $v$ implies uniform bounds of the form:
    \begin{equation*}
        \| u_i \|_{W^{k,p}_\delta} \leq C \| f_i \|_{L^p_\delta} \leq C^\prime T \|f\|_{L^p}, ~~ \| h_i \| + \| w_i \| \leq C^{\prime\prime} e^{-\delta T} \|f\|_{L^p}
    \end{equation*}
    for some uniform constants $C^\prime$ and $C^{\prime\prime}$. Taking cutoffs we can write:
    \begin{equation*}
        f_i - P_T(\chi_{T+1}(\rho_i) u_i + \chi_{T+1}(\rho_i) h_i) = \chi_{T+1}(\rho_i) w_i + r_i
    \end{equation*}
    where $r_i$ is an error term of the form
    \begin{equation*}
        r_i = (P_i-P_T)(\chi_{T+1}(\rho_i) u_i + \chi_{T+1}(\rho_i) h_i) + [P_i, \chi_{T+1}(\rho_i)](u_i + h_i).
    \end{equation*}
    As the coefficients of $P_i-P_T$ and their derivatives have exponential decay with $T$, $u_i$ has exponential decay at infinity. Given the bound on $h_i$, it follows that for any $0 < \epsilon < \delta$ we can bound the errors terms by:
    \begin{equation}
        \|r_i\|_{L^p} \leq C e^{-(\delta-\epsilon) T} \| f \|_{L^p}
    \end{equation}
    for some uniform constant. Let us define:
    \begin{equation*}
        u = \zeta_0 u_0 + \chi_{T+1}(\rho_1)(u_1 + h_1) + \chi_{T+1}(\rho_2) (u_2+h_2).
    \end{equation*}
    Then $f-P_Tu = \chi_{T+1}(\rho_1) w_1 + \chi_{T+2}(\rho_2)w_2 + r_1 + r_2$ satisfies $\|f-P_Tu\|_{L^p} \leq C e^{-(\delta-\epsilon)T} \|f\|_{L^p}$, and $\|u\|_{W^{l,p}} \leq CT \| f \|_{L^p}$, for some constant $C$. By Lemma \ref{lem:lproj}, we can decompose $u = u^\prime + w$ where $w \in \mK_T$ and $u^\prime$ is orthogonal to the substitute kernel. Moreover we have bounds:
    \begin{equation*}
        \|u^\prime\|_{W^{k,p}} \leq (1+C^\prime_1) \|u\|_{W^{k,p}} \leq C^\prime T \| f \|_{L^p}, ~~ \|w\|_{W^{k,p}} \leq C_1 \|u\|_{W^{k,p}} \leq C^{\prime\prime} T \|f\|_{L^p}
    \end{equation*}
    and $u^\prime$ satisfies 
    \begin{equation}
        f-P_Tu^\prime = f-P_Tu + P_Tw
    \end{equation}
    and as $w \in \mK_T$, then $\|P_Tw\|_{L^p} \leq C e^{-\delta T} \|w\|_{W^{k,p}} \leq C e^{-(\delta-\epsilon)T}\|f\|_{L^p}$.
\end{proof}

\paragraph{} Now we have all the tools to prove Theorem \ref{thm:main}, in the case where $\lambda_0$ is a root of order $1$. Let $f \in L^p(F_T)$ be an arbitrary section. By Lemma \ref{lem:lproj}, there exists $\tilde f \in L^p(F_T)$ and $w_0 \in \mK_T^*$ such that $f = \tilde f + w_0$, $\tilde f$ is orthogonal to $\mK_T^*$ and with bounds:
\begin{equation*}
   \| \tilde f \|_{L^p} \leq (1+C_1) \|f\|_{L^p}, ~~ \|w_0\|_{L^p} \leq C_1\|f\|_{L^p}.
\end{equation*}
Moreover, by Proposition \ref{prop:approxsolv} there exists $u_0 \in W^{k,p}(E_T)$ orthogonal to $\mK_T^*$ and $f_1 \in L^p(F_T)$ such that:
\begin{equation*}
    \tilde f = P_T u_0 + f_1
\end{equation*}
with bounds:
\begin{equation*}
    \|u_0\|_{W^{k,p}} \leq C_3 T \|\tilde f\|_{L^p} \leq (1+C_1) C_3 T \|f\|_{L^p} 
\end{equation*}
and
\begin{equation*}
    \|f_1\|_{L^p} \leq C_2 e^{- \delta T} \| \tilde f \|_{L^p} \leq (1+C_1) C_2 e^{- \delta T} \|f\|_{L^p}.
\end{equation*}
Choose $T$ large enough such that $\eta = (1+C_1) C_2 e^{- \delta T} < 1$, and define $f_0 = f$. Inductively, we can construct sequences $\{f_n, ~ n \geq 0 \}$ in $L^p(F_T)$, $\{u_n, n \geq 0 \}$ in the $L^2$-orthogonal complement of $\mK_T$ in $W^{k,p}(E_T)$ and $\{w_n, ~ n \geq 0 \}$ in $\mK_T^*$ such that for all $n \geq 0$ we have:
\begin{equation}        \label{eq:ffn}
    f_n - f_{n+1} = P_T u_n + w_n
\end{equation}
with the bounds:
\begin{equation}		\label{eq:boundsfn}
    \|f_n\|_{L^p} \leq \eta^n \|f\|, ~~ \|u_n\|_{W^{k,p}} \leq \eta^n (1+C_1) C_3 T \|f\|_{L^p}, ~~ \text{and} ~~ \|w_n\|_{L^p} \leq \eta^n C_1 \|f\|_{L^p}.
\end{equation}
As $W^{k,p}(E_T)$ is complete and $\eta < 1$, the series $\sum u_n$ converges. Let $u = \sum_{n=0}^\infty u_n$. As each term of the series is orthogonal to $\mK_T$, $u$ belongs to the orthogonal space to $\mK_T$ in $W^{k,p}(E_T)$. In the same way, the series $\sum w_n$ converges to an element $w \in \mK^*_T$. It follows from the bounds \eqref{eq:boundsfn} that we have:
\begin{equation*}
    \|u\|_{W^{k,p}} \leq \frac{(1+C_1) C_2}{1-\eta}T \|f\|_{L^p}, ~~ \| w \|_{L^p} \leq \frac{C_1}{1-\eta} \|f\|_{L^p}
\end{equation*}
Further the map $W^{k,p}(E_T) \rightarrow L^p(F_T)$ is continuous, and therefore we can sum over $n$ in equality \eqref{eq:ffn} to obtain $f = P_T u + w$. 

This proves the existence part in Theorem \ref{thm:main} for $f$ in the $L^p$ range. In particular the index of $P_T$ satisfies the inequality:
\begin{equation}        \label{eq:indpin}
    \ind(P_T) \geq \dim \mK_T - \dim \mK^*_T.
\end{equation}
As the map $W^{k,p}(E_T) \rightarrow L^p(F_T)$ induced by $P_T$ is Fredholm, the uniqueness of $u \in W^{k,p}(E_T)$ orthogonal to $\mK_T$ and $w \in \mK_T^*$ satisfying $f = P_Tu + w$ is equivalent to proving that inequality \eqref{eq:indpin} is in fact an equality. But the same reasoning applied to $P^*_T$ yields:
\begin{equation*}
    \ind(P^*_T) \geq \dim \mK^*_T - \dim \mK_T.
\end{equation*}
Since $\ind(P^*_T) = - \ind(P_T)$, uniqueness in Theorem \ref{thm:main} follows.

To complete the proof of the theorem in the Sobolev range, it remains to remark that if one further assumes that $f \in W^{l,p}(F_T)$, Proposition \ref{prop:unifapriori} implies that 
\begin{equation*}
    \|u\|_{W^{k+l,p}} \leq C (\|f\|_{W^{l,p}} + \|u\|_{L^p}) \leq C \|f\|_{W^{l,p}} + C^\prime T \|f\|_{L^p}
\end{equation*}
for some constant $C^\prime > 0$.

\begin{rem}
    One of the advantages of treating the case of a root of order $1$ first is that we proved that the Sobolev constant does not grow more than linearly with $T$, whereas in the general case it is more complicated to find the optimal rate of growth of the constant. This will be useful in our applications in Section \ref{section:spectral} to derive the rate of decay of the low eigenvalues of the Laplacian.
\end{rem}

\paragraph{} Let us go back to the general case, before indicating how to modify our construction to treat the case of a single root of any order. We first prove our previous claim, that without any restrictions on the number of real roots of $P_0$ the characteristic system admits a solution if and only if $f$ is orthogonal to the substitute cokernel:

\begin{lem}     \label{lem:solchar}
    For any $T \geq 1$ and any $f \in L^p(F_T)$ orthogonal to $\mK^*_T$, the characteristic system \eqref{eq:linsys} admits a solution $v \in \mE$.
\end{lem}

\begin{proof}
    Let us use the following notations for $u_1 \in \mK_1$ and $g_1 \in \mK^*_1$:
    \begin{equation*}
        \kappa_{1,T}[u_1](t,x) = \kappa_1[u_1](t+T+1,x), ~~ \kappa^*_{1,T}[u_1](t,x) = \kappa^*_1[u_1](t+T+1,x)
    \end{equation*}
    and for $u_2 \in \mK_2$ and $g_2 \in \mK^*_2$:
    \begin{equation*}
        \kappa_{2,T}[u_2](t,x) = \kappa_2[u_2](t-T-1,x), ~~ \kappa^*_{2,T}[u_2](t,x) = \kappa^*_2[u_2](t-T-1,x).
    \end{equation*}
    As the  pairing $(\cdot, \cdot)$ is invariant by translation, it is still true that $\image \kappa_{i,T}$ is the orthogonal space to $\image \kappa^*_i$. Thus we may proceed exactly as above, choosing a complement $\mE'_T$ of $\image \kappa_{1,T} \cap \image \kappa_{2,T}$ in $\mE$, and complements $\mE^*_{i,T}$ of  $\image \kappa^*_{1,T} \cap \image \kappa^*_{2,T}$ in $\kappa^*_{i,T}$. Once these arbitrary choices are made we can decompose:
    \begin{equation*}
        \mE'_T = \mE_{0,T} \oplus \mE_{1,T} \oplus \mE_{2,T}
    \end{equation*}
    where $\mE_{i,T} = \image \kappa_{i,T} \cap \mE'_T$ and $\mE_{0,T}$ is the orthogonal space of $\mE_{1,T} \oplus \mE_{2,T}$ in $\mE'_T$. Hence the non-degenerate pairing $\mE'_T \times (\image \kappa_{1,T} + \image \kappa_{2,T}) \rightarrow \C$ induced by $(\cdot, \cdot)$ decomposes as the orthogonal sum of the non-degenerate pairings:
    \begin{equation*}
        \mE_{0,T} \times (\image \kappa_{1,T} \cap \image \kappa_{2,T}) \rightarrow \C, ~~ \mE_{1,T} \times \mE^*_{2,T} \rightarrow \C ~~ \text{and} ~~ \mE_{2,T} \times \mE^*_{1,T} \rightarrow \C.
    \end{equation*}
    As in the proof of Proposition \ref{prop:approxsolv}, for $i=1,2$ the orthogonal space $\mK_{i,+}^*$ of $\mK_{i,0}$ in $\mK_i$ is isomorphic to $\image \kappa^*_{i,T}$, so that the decomposition $\image \kappa^*_{i,T} = \image \kappa^*_{1,T} \cap \image \kappa^*_{2,T} \oplus \mE_{i,T}^*$ induces a corresponding decomposition $\mK_{i,+}^* = \mK^*_{i,T,m} \oplus \mK^*_{i,T,\perp}$. Thus if $f \in L^p(M_T)$ is orthogonal to the substitute cokernel $\mK^*_T$ the characteristic system can be written as:
    \begin{equation*}
        \begin{cases}
            (v_0, \kappa^*_{1,T}[g_0]) = \langle f,(1-\chi_{T+1}(\rho_1)) g_0 \rangle - \langle (1-\chi) f_0, \kappa^*_1[g_0] \rangle_0, & \forall g_0 \in \mK^*_{1,T,m} \\
            (v_1, \kappa^*_{1,T}[g_1]) = \langle f,(1-\chi_{T+1}(\rho_1)) g_1 \rangle - \langle (1-\chi) f_0, \kappa^*_1[g_1] \rangle_0, & \forall g_1 \in \mK^*_{1,T,\perp} \\
            (v_2, \kappa^*_{2,T}[g_2]) = - \langle f,(1-\chi_{T+1}(\rho_2)) g_2 \rangle + \langle (1-\chi) f_0, \kappa^*_2[g_2] \rangle_0, & \forall g_2 \in \mK^*_{2,T,\perp}
        \end{cases}
    \end{equation*}
    where $v = v_0 + v_1 + v_2 \in \mE_{0,T} \oplus \mE_{1,T} \oplus \mE_{2,T}$. Given the non-degeneracy of the above pairings this system is manifestly invertible. 
\end{proof}

Despite the fact that we can solve the characteristic system whenever $f$ is orthogonal to the substitute cokernel $\mK_T$, this does not imply that we can find a solution $v \in \mE$ with bounds of the form $\| v \| \leq C(T) \|f\|_{L^p}$ with a good control on $C(T)$, which was a key argument in the previous construction. This is due to the fact that the characteristic system is in general underdetermined, and only becomes determined after a choice of arbitrary complements $\mE'_T$ of $\image \kappa_{1,T} \cap \image \kappa_{2,T}$ in $\mE$ and $\mE^*_{i,T}$ of  $\image \kappa^*_{1,T} \cap \image \kappa^*_{2,T}$ in $\image \kappa^*_{i,T}$, using the notations introduced in the above proof. In the case where $P_0$ has a single root of order $1$, this was not problematic as we could simply make any arbitrary choice independently of $T$, but in general we cannot make such a consistent choice. This is especially true at values of $T$ where the rank of the characteristic system drops. 

As we discussed in \S\ref{subsection:substk}, if $P_0$ has only one real root, then in a good basis the coefficients of the characteristic system are polynomial in $T$. Therefore, the rank of the system is constant whenever $T$ is large enough and we can fix a complement of its kernel independent of $T$. On this complement, the system can be inverted with polynomial control on the norm. Thus if $f \in L^p(F_T)$ is orthogonal to the substitute cokernel we can find solutions of the characteristic system with $\| v \| \leq C T^\beta \| f\|_{L^p}$ for some exponent $\beta > 0$. In the same way, all the matching conditions can be expressed as linear equations with coefficients depending polynomially on $T$. Therefore, the norm of the $L^2$-orthogonal projections onto $\mK_T$ and $\mK^*_T$ do not grow more than polynomially. Then we can use the same argument as above to prove Theorem \ref{thm:main} in the general case.

In fact, any assumptions ensuring that we can invert the characteristic system with less than exponential growth on the norm after fixing complements of its image and kernel, and that the norm of the projections onto $\mK_T$ and $\mK^*_T$ do not grow too quickly, would yield the same result. Moreover, in cases where the rank of the system is not constant we may also obtain good bounds by staying away from the values of $T$ where it drops. However we do not need to consider more complicated cases for our applications.

\section{Spectral aspects}      \label{section:spectral}

In this section, we want to interpret our results from a spectral perspective. Indeed, for self-adjoint operators the approximate kernel can be regarded as a finite-dimensional space associated with very low eigenvalues of the operator $P_T$. For the Laplacian, we shall see in \S\ref{subsection:approxharm} that the substitute kernel is a good approximation of the space of harmonic forms. Orthogonally to the space of harmonic forms, the results of the previous section imply a bound in $O(T^2)$ on the $L^2$-norm of the inverse of $\Delta_T$. In particular if $0$ is a root of the Laplacian acting on $q$-forms on $\R \times X$, then the lowest nonzero eigenvalues of $\Delta_T$ acting on $q$-forms admit a lower bound of the form $\frac{C}{T^2}$. In \S\ref{subsection:spectrum} we study the density of eigenvalues with fastest decay rate and prove Theorem \ref{thm:density}.

    \subsection{Approximate harmonic forms}     \label{subsection:approxharm}

\paragraph{} We begin by a review of standard properties of the Laplacian on EAC manifolds (see for instance \cite[6.4]{melrose1993atiyah}). Let $(Z,g)$ be an oriented EAC Riemannian manifold of rate $\mu >0$, $Y = \R \times X$ its asymptotic cylinder and $\rho$ a cylindrical coordinate function. The space $\mH^q$ of bounded closed and co-closed $q$-forms is equal to the space of bounded harmonic $q$-forms. Moreover, there is a direct sum decomposition:
\begin{equation*}
    \mH^q = \mH^q_0 \oplus \mH^q_d \oplus \mH^q_{d^*} 
\end{equation*}
where $\mH_0^q$ is the space of decaying harmonic $q$-forms, $\mH^q_d$ is the space bounded exact harmonic $q$-forms and $\mH^q_{d^*}$ the space of bounded co-exact harmonic $q$-forms. On the other hand, the map $\kappa_q$ mapping a bounded harmonic $q$-form to its translation-invariant expansion induces two maps
\begin{equation*}
    \alpha_q : \mH^q \rightarrow H^{q}(X), ~~~ \beta_q : \mH^q \rightarrow H^{q-1}(X)
\end{equation*}
such that $\kappa_q(\eta) = \alpha_0 + dt \wedge \beta_0$ where $\alpha_0$ and $\beta_0$ are the harmonic representatives of $\alpha_q(\eta)$ and $\beta_q(\eta)$ respectively. The map $\alpha_q$ can be factorised as $\mH^q \rightarrow H^q(Z) \rightarrow H^q(X)$, where any $q$-form in $\mH^q$ is mapped to its de Rham cohomology class in $H^q(Z)$ and the map $H^q(Z) \rightarrow H^q(X)$ comes from the long exact sequence of the pair $(Z, X)$. By \cite[Proposition 4.9]{atiyah1975spectral}, $\mH^q_0$ is mapped isomorphically to the image of the map $H^q_c(Z) \rightarrow H^q(Z)$ coming from the same exact sequence. In particular this implies that $\mH^q_0 \oplus \mH^q_d \subset \ker \alpha_q$, and by considering Hodge duals it follows that $\mH^q_0 \oplus \mH^q_{d^*} \subset \ker \beta_q$. As the kernel of the map $\kappa_q$ is $\mH^q_0$ this implies that we have in fact:
\begin{equation*}
    \ker \alpha_q = \mH^q_0 \oplus \mH^q_d, ~~ \ker \beta_q = \mH^q_0 \oplus \mH^q_{d^*}.
\end{equation*}
By \cite[Proposition 6.18]{melrose1993atiyah}, the map $\mH^q_0 \oplus \mH^q_{d^*} \rightarrow H^q(Z)$ is an isomorphism, and $\alpha_q$ maps $\mH^q_{d^*}$ isomorphically onto the image of the map $H^q(Z) \rightarrow H^q(X)$ coming from the long exact sequence of $(Z, X)$.

\paragraph{}  \label{par:closedforms} Let $0 \leq q \leq \dim Z$ and denote by $\sigma_q$ the minimum of $\mu$ and of the square roots of the lowest eigenvalues of the Laplacian acting on $(q-1)$- and $q$-forms on $X$. Any bounded closed and co-closed $q$-form $\eta$ on $Z$ is asymptotic to a translation-invariant form $\eta_0 = \alpha_0 + dt \wedge \beta_0$, up to terms in $O(e^{-\delta \rho})$ for any $\delta < \sigma_q$. With the above notations $(\alpha_0, \beta_0)$ are the harmonic representatives of $(\alpha_q(\eta), \beta_q(\eta))$. It is a standard fact that there exists a $(q-1)$-from $\xi$ on $Z$ such that $\eta-\eta_0 = d\xi$ in the domain $\{\rho > 1\}$, with $|\nabla^l \xi| = O(e^{-\delta \rho})$ for any $l \geq 0$ and $\delta < \sigma_q$. A suitable $\xi$ can be constructed as follows. Identify the region $\{ \rho > 1 \}$ with the cylinder $(1,\infty) \times X$, and write $\eta-\eta_0 = \alpha(t) + dt \wedge \beta(t)$ where $\alpha$, $\beta$ and all their derivatives have the usual exponential decay. As  $\eta$ and $\eta_0$ are closed this implies:
\begin{equation*}
    d_X \alpha(t) = 0 = \del_t \alpha(t) - d_X\beta(t)
\end{equation*}
for all $t > 1$, where $d_X$ denotes the exterior differential on $X$. Hence we can define $\xi$ in the domain $\{ \rho > 1 \}$ by
\begin{equation*}
    \xi(t,x)= \int_{+\infty}^t \beta(\tau,x) d \tau, ~~~ \forall (t,x) \in (1,\infty) \times X .
\end{equation*}
This $(q-1)$-form $\xi$ allows us to build a $1$-parameter family of closed $q$-forms
\begin{equation*}
    \eta_T = \eta- d(\chi_T(\rho) \xi)
\end{equation*}
interpolating between $\eta$ when $\rho < T - \frac 1 2$ and $\eta_0$ when $\rho > T + \frac 1 2$, which all represent the cohomology class of $\eta$ in $H^q(Z)$. Moreover, the difference $\eta_T-\eta$ and all its derivatives satisfy uniform bounds in $O( e^{-\delta T})$ for any $0 < \delta < \sigma_q$.

Let $(Z_1,g_1)$ and $(Z_2,g_2)$ be two matching EAC manifolds, and consider the $1$-parameter family of compact Riemannian manifolds $(M_T,g_T)$ obtained by the gluing procedure explained in \S\ref{subsection:gluingpb}. We want to control the mapping properties of the associated operators $d+d^*_T$ and $\Delta_T$ as $T \rightarrow \infty$. Strictly speaking, these operators differ from the operators obtained by gluing $d+d^*_1$ with $d+d^*_2$ and $\Delta_1$ with $\Delta_2$ in the gluing region $\{| \rho_T | \leq \frac 3 2 \}$. Nevertheless, the results of the previous parts still apply as the coefficients of the difference and and all their derivatives have exponential decay with $T$. It is convenient to slightly modify our definition of approximate kernel. If $(\eta_1,\eta_2)$ is a matching pair of harmonic forms, define a closed form on $M_T$ by:
\begin{equation*}
    \eta_T = 
    \begin{cases}
        \eta_{1,T} & \text{if} ~ \rho_T \leq - \frac 1 2 \\
        \eta_{2,T} & \text{if} ~ \rho_T \geq \frac 1 2 \\
        \eta_0 & \text{if} ~ |\rho_T| \leq \frac 1 2
    \end{cases}
\end{equation*}
where both $\eta_1$ and $\eta_2$ are asymptotic to $\eta_0$ and $\eta_{1,T}$ and $\eta_{2,T}$ are closed forms constructed as above. It follows that $\eta_T$ is closed. We denote by $\mH^q_T$ the finite-dimensional space of $q$-forms constructed as above from a pair of matching $q$-forms. Again, this differs from our previous definition of substitute kernel only up to terms that are bounded in $O(e^{-\delta T})$ as well as all their derivatives for any $\delta < \sigma_q$. Hence our results, and in particular Theorem \ref{thm:main}, still apply. As the elements of $\mH^q_T$ are closed there is a well-defined map:
\begin{equation*}
    \mH^q_T \rightarrow H^q(M_T)
\end{equation*}
sending every element to its de Rham cohomology class. The key point is the following theorem \cite[Theorem 3.1]{norstrom2009deformations}:

\begin{thm}
    For $T$ large enough, the map $\mH^q_T \rightarrow H^q(M_T)$ is an isomorphism.
\end{thm}

Let us briefly sketch the proof of this theorem. It relies on a close examination of the Mayer-Vietoris sequence:
\begin{equation}
    \ldots \rightarrow H^{q-1}(Z_1) \oplus H^{q-1}(Z_2) \rightarrow H^{q-1}(X) \rightarrow H^q(M_T) \rightarrow H^q(Z_1) \oplus H^q(Z_2) \rightarrow \ldots
\end{equation}
As the space of approximate harmonic $q$-forms $\mH^q_T$ is isomorphic to the space
\begin{equation*}
    H^q_c(Z_1) \oplus H^q_c(Z_2) \oplus \image \alpha_{1,q} \cap \image \alpha_{2,q} \oplus \image \beta_{1,q} \cap \image \beta_{2,q}
\end{equation*}
and since $\ker \beta_{i,q} \simeq H^q(Z_i)$, it is clear that the restriction of $\mH^q_T \rightarrow H^q(M_T)$ to the space obtained by matching pairs in $\mH^q_{i,0} \oplus \mH^q_{i,d^*}$ yields an isomorphism:
\begin{equation*}
     H^q_c(Z_1) \oplus H^q_c(Z_2) \oplus \image \alpha_{1,q} \cap \image \alpha_{2,q} \simeq \image (H^q(M_T) \rightarrow H^q(Z_1) \oplus H^q(Z_2)).
\end{equation*}
Moreover, the subspace of $\mH^q_T$ obtained by gluing matching pairs of bounded exact harmonic $q$-forms, which is isomorphic to $\image \beta_{1,q} \cap \image \beta_{2,q}$, maps into the image of $H^{q-1}(X) \rightarrow H^q(M_T)$. By Lemma \ref{lem:valuepuv}, $\image \beta_{1,q} \cap \image \beta_{2,q}$ is the orthogonal space of $\image \alpha_{1,q-1} \oplus \image \alpha_{2,q-1}$  for the inner product induced by the $L^2$-product on harmonic representatives, and therefore $\image \beta_{1,q} \cap \image \beta_{2,q}$ has the same dimension as the kernel of $H^q(M_T) \rightarrow H^q(Z_1) \oplus H^q(Z_2)$. Thus it only remains to prove that the subspace of $\mH^q_T$ obtained by gluing matching pairs of exact $q$-forms maps isomorphically onto the kernel of the map $H^q(M_T) \rightarrow H^q(Z_1) \oplus H^q(Z_2)$ coming from the exact sequence. In \cite[Theorem 3.1]{nordstrom2008deformations} it is proven that this is the case for $T$ large enough.

Alternatively, one could also argue using Theorem \ref{thm:main}. The spaces $\mH^q_T$ and $H^q(M_T)$ have same dimension by the above argument. Moreover, the Laplacian $\Delta_T$ maps the orthogonal space of $\mH^q_T$ in $W^{2,2}(\Lambda^q T^*M_T)$ isomorphically onto a complement of $\mH^q_T$ in $L^2(\Lambda^qT^*M_T)$, for $T$ large enough. Hence the map $\mH^q_T \rightarrow H^q(M_T)$ must be an isomorphism for large $T$, since otherwise there would be a non-trivial exact form in $\mH^q_T$ and the image of the Laplacian would have codimension strictly less than $b^q(M_T)$ in $L^2(\Lambda^qT^*M_T)$.

\begin{rem}
    The above theorem remains true when considering the variant of our gluing problem explained in Remark \ref{rem:twist}, where the matching condition between the two building blocks is twisted by an isometry $\gamma : X \rightarrow X$ of the cross-section.
\end{rem}

\paragraph{} As a consequence, the $L^2$-projection of the space $\mH^q_T$ of approximate harmonic $q$-forms onto the space $\mH^q(M_T)$ of genuine harmonic $q$-forms is an isomorphism for $T$ large enough. It is natural to ask how close to their harmonic part the elements of $\mH^q_T$ are. If $\eta \in \mH^q_T$ is decomposed in harmonic and exact parts $\eta = \xi + d\nu$, then:
\begin{equation*}
    \|\Delta_T d \nu \|_{L^2} = \|\Delta_T(\eta-\xi)\|_{L^2} =  \|\Delta_T \eta\|_{L^2} = O \left(e^{-\delta T} \|\eta\|_{L^2} \right)
\end{equation*}
for any $\delta < \sigma_q$. By Theorem \ref{thm:main}, there exists a $(q-1)$-form $\eta^\prime$ with $\Delta_T \eta^\prime = \Delta_T d\nu$ and satisfying a bound of the form 
\begin{equation*}
    \| \eta^\prime \|_{W^{2,2}} \leq C T^\beta \|\Delta_T d\nu \|_{L^2} \leq C^\prime e^{-\delta T}\|\eta\|_{L^2}
\end{equation*}
for some constant $C^\prime$. As $\eta^\prime - d\nu$ is harmonic it follows that $\|d\nu\|_{L^2} \leq \|\eta^\prime\|_{L^2}$, which yields:
\begin{equation*}
    \|\eta-\xi\|_{L^2} = O \left( e^{-\delta T}\|\eta\|_{L^2} \right)
\end{equation*}
for any $\delta < \sigma_q$. Thus not only is the $L^2$-projection of $\mH^q_T$ onto $\mH^q(M_T)$ an isomorphism, but the norm of the projection is close to $1$, up to $O(e^{-\delta T})$ terms. Once this inequality is established in $L^2$, the a priori estimates of Proposition \ref{prop:unifapriori} imply that 
\begin{equation*}
    \|\eta-\xi\|_{W^{l,2}} = O \left( e^{-\delta T} \|\eta\|_{L^2} \right) 
\end{equation*}
for any $l \geq 0$. Then, the Sobolev embedding theorem (see Remark \ref{rem:sobemb}) yields estimates:
\begin{equation*}
    \|\eta-\xi\|_{W^{l,p}} = O \left( e^{-\delta T} \|\eta\|_{L^p} \right), ~~ \|\eta-\xi\|_{C^l} = O \left( e^{-\delta T} \|\eta\|_{C^0} \right)  
\end{equation*}
for any $l \geq 0$, $p > 1$ and $0 < \delta < \sigma_q$. 

By the same bootstrapping argument we can prove that if $\nu \in W^{l,p}(\Lambda^qT^*M_T)$ is orthogonal to $\mathscr{H}^q_T$ and $\nu^\prime$ is the unique $q$-form orthogonal to $\mathscr{H}^q(M_T)$ such that $\Delta_T \nu = \Delta_T \nu^\prime$ (or equivalently $(d+d^*_T) \nu = (d+d^*_T) \nu^\prime$) then 
\begin{equation*} 
    \| \nu - \nu^\prime \|_{W^{l,p}} = O(e^{-\delta T} \| \nu \|_{W^{l,p}}) .
\end{equation*}
These remarks, the fact that the Laplacian $\Delta_T$ is the square of the operator $d+d^*_T$ whose only root has order $1$ and Theorem \ref{thm:main} imply the following:

\begin{cor}         \label{cor:laplacian}
    Let $p > 1$ and $l \in \N$, and assume that $0$ is an indicial root of the Laplacian action on $q$-forms on $Y$, that is $b^{q-1}(X) + b^q(X) > 0$. Then there exist constants $C, C^\prime > 0$ such that, for large enough $T$ and any $\eta \in W^{l,p}(\Lambda^q T^*M_T)$ orthogonal to $\mH^q(M_T)$, the unique solution $\eta^\prime \in W^{2+l,p}(\Lambda^q T^*M_T)$ of $\Delta \eta^\prime = \eta$ orthogonal to $\mH^q(M_T)$ satisfies:
    \begin{equation*}
        \|\eta^\prime\|_{W^{l+2,p}} \leq C \|\eta\|_{W^{l,p}} + C^\prime T^2 \|\eta\|_{L^p} .
    \end{equation*}
\end{cor}

\begin{proof}
    Let us consider the operator $d+d^*_T$ acting on $\Lambda^\bullet T^*M_T$. Since the only real root of $d+d^*_T$ and this is a root of order $1$, we saw in \S\ref{subsection:construction} that in this case Theorem \ref{thm:main} holds with $\beta = 1$. Moreover, for $T$ large enough the $L^2$-projection of $\mathscr{H}_T$ on the space of harmonic forms $\mathscr{H}$ is an isomorphism, and therefore $\mathscr{H}_T$ is a linear complement of the image of $d+d^*_T$ in $W^{l,p}$. Thus by Theorem \ref{thm:main} any differential form $\eta$ of regularity $W^{l,p}$ which is orthogonal to the space of harmonic forms can be written as 
    \begin{equation*}
        \eta = (d+d^*_T) \nu
    \end{equation*}
    for a unique differential form $\nu$ of regularity $W^{l+1,p}$ orthogonal to the space of approximate harmonic forms $\mathscr{H}_T$, which satisfies bounds of the form
    \begin{equation*}
        \| \nu \|_{W^{l+1,p}} \leq C (\| \eta \|_{W^{l,p}} + T \| \eta \|_{L^p}), ~~~ \| \nu \|_{L^p} \leq \|\nu \|_{W^{1,p}} \leq C^{\prime} T \| \eta \|_{L^p} 
    \end{equation*}
    where in the second equality we use the case $l=0$ of the theorem. By the previous remarks, the unique differential form $\nu^\prime$ orthogonal to $\mathscr{H}(M_T)$ satisfying $(d+d^*_T) \nu^\prime = \eta$ satisfies bounds of the form $\|\nu - \nu^\prime \|_{W^{k,p}} \leq C_k e^{- \delta T} \| \nu \|_{W^{k,p}}$ for any $k \geq 0$ and small enough $\delta > 0$, and thus for large enough $T$ we also have
    \begin{equation*}
        \| \nu^\prime \|_{W^{l+1,p}} \leq C (\| \eta \|_{W^{l,p}} + T \| \eta \|_{L^p}), ~~~ \| \nu^\prime \|_{L^p} \leq \|\nu^\prime \|_{W^{1,p}} \leq C^{\prime} T \| \eta \|_{L^p} 
    \end{equation*}
    for some constants $C,C^\prime$, possibly different from the previous ones but independent of $T$. Iterating this argument, the unique differential form $\eta^\prime$ orthogonal to $\mathscr{H}(M_T)$ such that $(d+d^*_T) \eta^\prime = \nu^\prime$, which is a solution of $\Delta_T \eta^\prime = \eta$, satisfies the bounds
    \begin{equation*}
        \| \eta^\prime \|_{W^{l+2,p}} \leq C ( \| \nu^\prime \|_{W^{l+1,p}} + T \| \nu^\prime \|_{L^p}) \leq C^\prime \| \eta \|_{W^{l,p}} + C^{\prime\prime} T^2 \| \eta \|_{L^p} 
    \end{equation*}
    for some constants $C,C^\prime,C^{\prime\prime}$ which do not depend on $\eta$ and $T$ large enough.
\end{proof}

Consequently, if $b^{q-1}(X) + b^q(X) > 0$ the lowest eigenvalue of $\Delta_T$ acting on $q$-forms satisfies a bound of the form $\lambda_1(T) \geq \frac{C}{T^2}$ as $T \rightarrow \infty$. In the next part we study the distribution of the eigenvalues that have the fastest decay rate, that is of order $T^{-2}$.

    \subsection{Density of low eigenvalues}     \label{subsection:spectrum}

\paragraph{} We want bounds on the densities $\Lambda_{q,\inf}(s)$, $\Lambda_{q,\sup}(s)$ of low eigenvalues of the Laplacian $\Delta_T$ acting on $q$-forms defined in \S\ref{subsection:results}. When $b^{q-1}(X)+b^q(X) = 0$, the Laplacian acting on $q$-forms does not admit any real root, and thus it has no decaying eigenvalues. From now on we assume that $b^{q-1}(X)+b^q(X) > 0$. We shall prove Theorem \ref{thm:density} using a min-max principle.

The easiest part, which does not require the results of Section \ref{section:construction}, is to find a lower bound for $\Lambda_{q,\inf}(s)$. Let us denote by $0 < \lambda_1(T) \leq \ldots \leq \lambda_n(T) \leq \dots$ the non-decreasing sequence of nonzero eigenvalues of the Laplacian, counted with multiplicity. The $n$-th eigenvalue (counted with multiplicity) is determined by:
\begin{equation*}
    \lambda_n(T) = \min \left\{ \max \left\{ \frac{\|\Delta_T \eta\|_{L^2}}{\|\eta\|_{L^2}}, \eta \in V \backslash \{0\} \right\}, ~ V \subset W^{2,2}(\Lambda^qT^*M_T), ~ \dim V = n \right\}
\end{equation*}
where $V$ ranges over spaces orthogonal to harmonic forms. Using this we claim:

\begin{lem}
    Let $V \subset C^2([-1,1],\C)$ be an $n$-dimensional space of functions such that $f(-1)=f(1)=f^\prime(-1)=f^\prime(1)= 0$ for all $f \in V$. Let $\lambda > 0$ such that for all non-zero $f \in V$ we have:
    \begin{equation*}
        \int_{-1}^1 |f^{\prime\prime}(t)|^2 dt < \lambda^2 \int_{-1}^1 |f(t)|^2 dt
    \end{equation*}
    Then for $T$ large enough $\lambda_{(b^{q-1}(X)+b^q(X)) n - b^q(M_T)}(T) \leq \frac{\lambda}{T^2}$.
\end{lem}

\begin{proof}
    Any $f \in V$ can be extended as a $C^1$ function to $\R$ by setting $f(t) = 0$ for any $|t| \geq 1$. With this extension, $f \in W^{2,2}(\R)$ and $f^{\prime\prime} \in L^2(\R)$ vanishes outside of $[-1,1]$ and is equal to the usual second derivative inside this interval. Let us choose $0 < \tau < 1$ small enough so that:
    \begin{equation}        \label{eq:ineqltau}
        \int_{-1}^1 |f^{\prime\prime}(t)|^2 dt < \lambda^2(1-\tau)^4 \int_{-1}^1 |f(t)|^2 dt
    \end{equation}
    for any $f \in V$. For $T \geq 1$, let $V_T$ be the subspace of $W^{2,2}(\Lambda^q T_\C^*Y)$ spanned by sections of the form:
    \begin{equation*}
        \eta(t,x) = f\left( \frac{t}{(1-\tau)T} \right) \nu(x)
    \end{equation*}
    where $f \in V$ and $\nu$ is a translation-invariant harmonic form on $Y$. In particular, it has dimension $\dim V_T = (b^{q-1}(X)+b^q(X)) n$ over $\C$. As the elements of $V_T$ vanish outside of the finite cylinder $[-(t-\tau)T,(1-\tau)T] \times X$, it can be identified with a subspace of $W^{2,2}(\Lambda^qT_\C^*M_T)$. On the support of the elements of $V_T$, the Laplacian $\Delta_T$ and the metric $g_T$ approach $\Delta_0 = \Delta_X -\del_t^2$ and $g_0 = g_X + dt^2$ up to terms in $O(e^{-\delta \tau T})$ and similarly for all derivatives, for some $\delta > 0$ appropriately small. Therefore, there exist constants $C,C^\prime>0$ such that:
    \begin{equation*}
        \sup_{\eta \in V_T \backslash \{0\}} \frac{\|\Delta_T \eta \|_{L^2}}{\| \eta \|_{L^2}} \leq \frac{(1+ C e^{- \delta \tau T})}{(1-\tau)^2T^2} \sup_{f \in V \backslash \{0\}} \frac{\|f''\|_{L^2}}{\|f\|_{L^2}} + C^\prime e^{-\delta \tau T} \sup_{\eta \in V_T \backslash \{0\}} \frac{\|\eta\|_{W^{2,2}}}{\|\eta\|_{L^2}}.
    \end{equation*}
    As the ratio between the $W^{2,2}$-norm and the $L^2$-norm on $V_T$ does not grow more than polynomially with $T$, the second term in the right-hand-side has exponential decay. On the other hand, by \eqref{eq:ineqltau} the first term is less than $\frac{(\lambda-\epsilon)}{T^2}$ for large enough $T$ and small enough $\epsilon > 0$. This proves the lemma.
\end{proof}

We can apply the above lemma to the spaces:
\begin{equation}
    V_n = \left\{ \sum_{1 \leq |k| \leq n} a_{k} e^{ik\pi t}, ~ \sum_{1 \leq |k| \leq n} (-1)^k a_k = \sum_{1 \leq |k| \leq n} (-1)^k k a_k = 0 \right\}
\end{equation}
For $n \geq 2$, the space $V_n$ has dimension $2n-2$ and for any non-zero $f \in V_n$ we have:
\begin{equation*}
    \int_{-1}^1 |f^{\prime\prime}(t)|^2 dt < (n\pi)^2 \int_{-1}^1 |f(t)|^2 dt
\end{equation*}
The above lemma yields the inequality:
\begin{equation}        \label{infbounddens}
    \Lambda_{q,\inf}(s) \geq (2  \lfloor \sqrt s \rfloor - 2) (b^{q-1}(X)+b^q(X)) - b^q(M_T) , ~~ \forall s \geq 1. 
\end{equation}

\paragraph{} We now want an upper bound on $\Lambda_{q,\sup}(s)$. Let us denote by
\begin{equation*}
    G_T : L^2(\Lambda^qT^*M_T) \cap \mH^q(M_T)^\perp \rightarrow L^2(\Lambda^qT^*M_T) \cap \mH^q(M_T)^\perp
\end{equation*}
the composition of the inverse of the Laplacian acting on $W^{2,2}(\Lambda^qT^*M_T) \cap \mH^q(M_T)^\perp$ with the compact embedding $W^{2,2} \hookrightarrow L^2$. The eigenvalues $\lambda_{n+1}(T)$ can be characterised by:
\begin{equation*}
    \lambda^{-1}_{n+1}(T) = \min \left\{ \max \left\{ \frac{\|G_T \eta\|_{L^2}}{\|\eta\|_{L^2}}, \eta \in V \backslash \{0\} \right\},  V \subset L^2(\Lambda^qT^*M_T), \codim V = n \right\}
\end{equation*}
where moreover $V$ ranges over closed subspaces orthogonal to $\mH^q(M_T)$. Since we developed an explicit construction of solutions to the equation $\Delta_T \nu = \eta$, the idea is to show that if we impose enough orthogonality conditions to $\eta \in L^2(\Lambda^q T^*M_T)$, and not only orthogonality to the space of harmonic forms (or to the substitute kernel), we can give explicit bounds for the norm of $G_T \eta$.

Let us denote by $N$ the sum of the dimensions of the spaces of harmonic forms with at most polynomial growth on $Z_1$ and $Z_2$. Moreover, denote by $E \subset L^2([-1,1],\C)$ the closed subspace of functions $f(t) = \sum a_k e^{ik\pi t}$ which satisfy:
\begin{equation*}
    a_0 = 0, ~~~ \sum_{|k| \geq 1} (-1)^k \frac{a_k}{k} = \sum_{|k| \geq 1} (-1)^k \frac{a_k}{k^2} = 0 .
\end{equation*}
Thus $E$ has codimension $3$ in $L^2([-1,1],\C)$. For $f \in L^2(\R,\C)$ with compact essential support let us define:
\begin{equation*}
    Hf(t) = \int_{-\infty}^t (\tau-t) f(\tau) d\tau.
\end{equation*}
The first two conditions in the definition of $E$ imply that for any $f \in E$ on has:
\begin{equation}        \label{eq:vanishint}
    \int_{-1}^1 f(\tau) d\tau = \int_{-1}^1 \tau f(\tau) d\tau = 0 .
\end{equation}
On the other hand, the last condition is a matter of scaling under change of variables. Since we have
\begin{equation*}
    \int_{-T}^t (\tau-t) e^{\frac{ik\pi \tau}{T}} dt = \frac{T^2}{(k\pi)^2} e^{\frac{ik\pi t}{T}} + (-1)^k \left( \frac{T(T+t)}{ik\pi} - \frac{T^2}{(k\pi)^2} \right)
\end{equation*}
if follows that for any $f(t) = \sum a_k e^{ik\pi t} \in E$, the function $f_T(t) = f\left( \frac{t}{T} \right)$ satisfies:
\begin{equation}        \label{eq:hft}
    Hf_T(t) = \frac{T^2}{\pi^2} \sum_{|k| \geq 1} \frac{a_k}{k^2} e^{\frac{ik\pi t}{T}}
\end{equation}
for any $-T \leq t \leq T$. 

Bearing this in mind, we shall find an upper bound on $\Lambda_{q,\sup}(s)$ with the help of the following technical lemma:

\begin{lem}
    Let $V \subset E$ be a closed subspace of codimension $n$, and let $\lambda, \epsilon > 0$ such that for all $f \in V$ we have:
    \begin{equation*}
        \int_{-1}^1 |Hf(t)|^2 dt \leq \frac{1}{(\lambda+\epsilon)^2} \int_{-1}^{1} |f(t)|^2 dt .
    \end{equation*}
    Then for $T$ large enough $\lambda_{(b^{q-1}(X)+b^q(X))(n+3)+N}(T) \geq \frac{\lambda}{T^2}$. 
\end{lem}

\begin{proof}
    The idea is to follow the construction of \S\ref{subsection:construction} to build a solution of $\Delta_T \nu = \eta$, where $\eta$ is a complex $q$-form orthogonal to the space of harmonic forms, and showing that if we assume sufficiently many orthogonality conditions we can give a precise bound on the constant $C$ such that $\|\nu\|_{L^2} \leq CT^2 \|\eta\|_{L^2}$. To do this we need to introduce a parameter $\tau > 0$ and replace the cutoffs $\zeta_0$ and $\zeta_1$ (see \S\ref{subsection:characteristic}) by $\zeta_\tau$ and $\zeta_{\tau+1}$. 

    Let us define $\eta_\tau = \zeta_{\tau+1}\eta$, considered as a $q$-form on the cylinder $Y=\R \times X$ supported in the cylinder $[-T+\tau, T-\tau] \times X$. We pick a basis $\eta_1,\ldots,\eta_m$ of the space of translation-invariant harmonic $q$-forms on $Y$, orthonormal for the $L^2$-product on $X$. Then we may write:
    \begin{equation*}
        \eta_\tau(t,x) = \widetilde \eta_\tau(t,x) + \sum_{j=1}^m f_{\tau,j}(t) \eta_j(x)
    \end{equation*}
    with $\widetilde \eta_\tau$ orthogonal to any function of the form $f(t)\eta_j(x)$ with $f$ compactly supported smooth function, and $f_{\tau,j} \in L^2([-T+\tau,T-\tau],\C)$. Moreover the solution $\nu_\tau = Q\eta_\tau$ of $\Delta_0 \nu = \eta_\tau$ provided by Theorem \ref{thm:localsol} can be written as (see Example \ref{ex:fvproduct}):
    \begin{equation*}
        \nu_\tau = Q_r[\eta_\tau] + \sum_{j=1}^m Hf_{j,\tau} \cdot \eta_j .
    \end{equation*}
    with $Q_r$ defined as in \S\ref{subsection:localex}. Let us assume that each of the functions 
    \begin{equation*}
        t \in [-1,1] \mapsto f_{j,\tau}((T-\tau)t)
    \end{equation*}
    belongs to $V \subset E$. This imposes $(b^{q-1}(X)+b^q(X))(n+3)$ orthogonality conditions on $\eta$. Given \eqref{eq:vanishint} the $L^2$-functions $Hf_{j,\tau}$ vanish outside of $[-T+\tau,T-\tau]$, and therefore $\nu_\tau \in L^2(\Lambda^qT_\C^*Y)$ and from \eqref{eq:hft} it satisfies:
    \begin{equation}        \label{eq:boundetatau}
        \|\nu_\tau\|_{L^2} \leq C \|\eta_\tau\|_{L^2} + \frac{(T-\tau)^2}{\lambda+\epsilon} \|\eta_\tau\|_{L^2} \leq \frac{T^2}{\lambda+\epsilon} \|\eta_\tau\|_{L^2}
    \end{equation}
    for large enough $T$. Let us consider $\zeta_\tau$ as a section of $\Lambda^qT_C^*M_T$ supported in the neck region. As such, there exists a constant $C > 0$, which does not depend on $\tau$, such that:
    \begin{equation*}
        \|\zeta_\tau \nu_\tau \|_{L^2} \leq \frac{1+ C e^{-\delta \tau}}{\lambda + \epsilon}T^2 \|\eta\|_{L^2}
    \end{equation*}
    Following the method of \S\ref{subsection:characteristic}, we can write:
    \begin{equation*}
        \eta - \Delta_T (\zeta_\tau \nu \tau) = \eta_1 + \eta_2
    \end{equation*}
    with $\eta_i \in L^2_{\delta^\prime}(\Lambda^q T_C^*Z_i)$, for some small $\delta^\prime > 0$ that we fix. Moreover, as in the proof of Lemma \ref{lem:boundfi} to show the bounds:
    \begin{equation}        \label{eq:boundetai}
        \|\eta_i\|_{L^2_{\delta'}} \leq C e^{\delta^\prime \tau} \| \eta\|_{L^2} + C^\prime e^{-\delta \tau } \| \eta\|_{L^2} \leq C^{\prime\prime} T^2 e^{-\delta \tau} \|\eta\|_{L^2}
    \end{equation}
    for $T$ large enough, where $\delta$, $\delta^\prime$ and $\tau$ are fixed, and $C^{\prime\prime}$ does not depend on any of these choices. Up to $O(e^{-\delta T})$ terms, the vanishing of the obstructions to solving $f_i = \Delta_i \nu_i$ with $\nu_i \in W^{2,2}_{\delta^\prime}$ can be expressed as the vanishing of $N$ linear forms (this is to say that the coefficients of the characteristic system are linear in $\eta \in L^2$). Thus imposing $N$ additional orthogonality conditions on $\eta$, we can use the same argument as in Proposition \ref{prop:approxsolv} to show that there exists $\nu^\prime \in W^{2,2}$, $\eta^\prime \in L^2$ such that $\eta-\Delta_T \nu^\prime = \eta^\prime$ with $\|\eta^\prime\|_{L^2} \leq C  T^2 e^{-\delta^\prime T} \|\eta\|_{L^2}$ for some constant $C^\prime$ possibly depending on $\delta^\prime$ but not on $\tau$ or $T$. From \eqref{eq:boundetatau} and \eqref{eq:boundetai} we can moreover deduce a bound:
    \begin{equation*}
        \| \nu^\prime\|_{L^2} \leq \left( \frac{1+ C e^{-\delta \tau}}{\lambda + \epsilon} + C^\prime e^{-\delta \tau} \right) T^2 \| \eta \|_{L^2} 
    \end{equation*}
    for large enough $T$. On the other hand, as $\eta$ is by assumption orthogonal to the space of harmonic forms, so is $\eta^\prime$ and by Corollary \ref{cor:laplacian} there exists $\nu^{\prime\prime}$ such that $\Delta_T \nu^{\prime\prime} = \eta^\prime$ with a bound:
    \begin{equation*}
        \|\nu^{\prime\prime}\|_{L^2} \leq C T^2 \| \eta^\prime\|_{L^2} \leq C^{\prime\prime} T^4 e^{-\delta^\prime T} \|\eta\|_{L^2} 
    \end{equation*}
    for some constant $C^{\prime\prime}$ which does not depend on $\tau$ or on $T$ large enough. Thus if $\nu = \nu^\prime + \nu^{\prime\prime}$ we have $\Delta_T \nu = \eta$ with a universal bound:
    \begin{equation*}
        \|\nu\|_{L^2} \leq \left( \frac{1+ C e^{-\delta \tau}}{\lambda + \epsilon} + C^\prime e^{-\delta \tau} + C^{\prime\prime} T^2 e^{-\delta'T}\right) T^2
    \end{equation*}
    for some constants $C,C^\prime,C^{\prime\prime}$ that may depend on the choices of $\delta,\delta'$ but not on $\tau$ or large enough $T$. As $\|G_T \eta \|_{L^2} \leq \| \nu \|_{L^2}$ it follows that if $\tau$ and $T$ are large enough we have:
    \begin{equation*}
        \|G_T \eta \|_{L^2} \leq \frac{T^2}{\lambda} \| \eta \|_{L^2} .
    \end{equation*}
    This inequality holds true provided $\eta$ satisfies all the orthogonality conditions described above, which define a closed subspace of codimension no more than $(b^{q-1}(X)+b^q(X))(n+3) + N$ in the orthogonal space to harmonic forms in $L^2(\Lambda^qT_\C^*M_T)$. The lemma follows. 
\end{proof}

To use this lemma we consider the subspaces $V^\prime_n \subset E$ defined by:
\begin{equation}
    V^\prime_n = \left\{ f(t) = \sum a_k e^{ik\pi t} \in E, ~ a_k = 0 ~ \forall |k| \leq n \right\}.
\end{equation}
The space $V^\prime_n$ has codimension $2n$ in $E$, and for any $f \in V^\prime_n$, \eqref{eq:hft} implies:
\begin{equation*}
    \int_{-1}^1 |Hf(t)|^2 dt^2 \leq \frac{1}{(n+1)^4\pi^4} \int_{-1}^{1} |f(t)|^2 dt
\end{equation*}
Hence we have an upper bound:
\begin{equation*}
    \Lambda_{q,\sup}(s) \leq 2(\lfloor \sqrt s \rfloor + 3)(b^{q-1}(X) + b^q(X)) + N .
\end{equation*}
Together with the bound on $\Lambda_{q,\inf}(s)$ this proves the first part of Theorem \ref{thm:density}.

\paragraph{} In order to prove the second assertion in Theorem \ref{thm:density}, let us consider the subset $W_n \subset V_n$ defined by:
\begin{equation}
    W_n = \left \{ \sum_{1 \leq |k| \leq n} a_{k} e^{ik\pi t} \in V_n, ~ \sum_{1 \leq |k| \leq n} (-1)^k k^2 a_k = 0 \right \}
\end{equation}
seen as a subspace of $C^3([-1,1],\C)$. Any $f \in W_n$ can be extended as a $C^2$-function on $\R$, with $f^\prime \in V_n$. Moreover if $\beta$ is a harmonic $(q-1)$-form then:
\begin{equation*}
    d(f \beta) = f^\prime dt \wedge \beta .
\end{equation*}
Using this, we can deduce that the density of low eigenvalues of the Laplacian acting on exact $q$-forms, which we define as in \ref{par:spectral}, satisfies:
\begin{equation*}
    \Lambda^e_{q,\inf}(s) \geq 2b^{q-1}(X) \sqrt s - N_q
\end{equation*}
for some constant $N_q \geq 0$. By Hodge duality, this implies the lower bound:
\begin{equation*}
    \Lambda^*_{q,\inf}(s) \geq 2b^q(X) \sqrt s - N_{\dim M_T - q}.
\end{equation*}
As we know that $\Lambda^*_{q,\sup}(s)+\Lambda^e_{q,\sup}(s) \leq 2(b^{q-1}(X) + b^q(X)) \sqrt s + O(1)$, this means that when $b^q(X) \neq 0$ we have:
\begin{equation*}
    \Lambda^*_{q, \sup}(s) = \Lambda^*_{q, \inf}(s) + O(1) = 2 b^q(X) \sqrt{s} + O(1).
\end{equation*}
If on the other hand $b^q(X) = 0$ then $\Lambda^*_{q, \sup}(s) = O(1)$, so that only finitely many eigenvalues of the Laplacian acting on co-exact $q$-forms may decay at rate $T^{-2}$, the rest of the low eigenvalues would decay at a slower rate. 

\paragraph{}    \label{par:stronger} It is natural to ask if stronger statements about the distribution of low eigenvalues could be made. In particular one could ask if the spectrum of the Laplacian splits with a lower part represented by the spectrum of the Laplacian acting on $S^1_{2T} \times X$, where the circle factor has length $2T$. Even in the simple case of the Laplacian acting on functions this does not hold. Considering the function $\rho_T$ and applying a min-max argument as above, one can easily see that for large enough $T$ the lowest non-zero eigenvalue of the scalar Laplacian satisfies $\lambda_1(T) \leq \frac{6}{T^2}$. In general, it seems that the interactions between the two building blocks tend to shift the lower spectrum of $\Delta_T$ compared with the spectrum of the Laplacian acting on $S^1_{2T} \times X$. The precise way in which this shift happens is likely to depend on the building blocks, and it may not be tractable analytically to give sharp bounds for the sequence of low eigenvalues of the Laplacian. 

To finish this section, let us discuss possible generalisations of Theorem \ref{thm:density}. If we consider more general formally self-adjoint operators $P_T$ that fit into the set-up of Theorem \ref{thm:main}, the substitute kernel of $P_T$ can be interpreted as a finite number of very low eigenvalues, with decay rate exponential in $T$. The rest of the eigenvalues have at most polynomial decay, and in fact one could easily see that this decay rate is in $T^{-d}$ (except for a finite number that may have faster polynomial decay), where $d$ is the order of the real root of the indicial operator $P_0$. If we assume for instance that $P_T$ has non-negative eigenvalues and that $P_0$ can be written as $P_0(\lambda) = D + \lambda^d$, the proof given above carries out without problem to show that the density of eigenvalues of $P_T$ contained in the interval $\left( 0, \pi^d s T^{-d} \right]$ is equivalent to $2 (\dim \ker D)  s^{\frac 1 d}$ as $s \rightarrow \infty$.

\section{Improved estimates for twisted connected sums}     \label{section:compact}

In this section, we apply our results to the study of compact manifolds with holonomy $G_2$ constructed by twisted connected sum. We review the basics of $G_2$-geometry in \S\ref{subsection:gtwo}. In \S\ref{subsection:tcs} we prove Corollary \ref{cor:tcsestimates}, using quadratic estimates which we prove in \S\ref{subsection:quadratic}.

    \subsection{{$G_2$}-manifolds}   \label{subsection:gtwo}   

\paragraph{} Let $V$ be an oriented real vector space of dimension $7$. A $3$-form $\varphi \in \Lambda^3 V^*$ is said to be \emph{positive} if for any non-zero $v \in V$ we have:
\begin{equation}
    \iota_v \varphi  \wedge \iota_v \varphi \wedge \varphi > 0
\end{equation}
relatively to the choice of orientation, where $\iota$ denotes the interior product. Let $\Lambda^3_+ V^*$ be the set of positive forms on $V$. This is an open subset of $\Lambda^3 V^*$ and the group of orientation-preserving automorphisms $GL_+(V)$ acts transitively on it. The stabiliser of any positive form is identified with the Lie group $G_2$. 

If $\Vol$ is a volume form on $V$ and $\varphi$ a positive $3$-form, we have
\begin{equation}
    \iota_u \varphi \wedge \iota_v \varphi \wedge \varphi  = 6 \langle u, v \rangle \Vol
\end{equation}
for some inner product $\langle \cdot, \cdot \rangle$, but there is an ambiguity in the scaling of $\langle \cdot, \cdot \rangle$ and $\Vol$. This can be fixed by requiring $| \varphi |^2 = 7$, and hence any positive form $\varphi$ canonically defines a volume form $\Vol_\varphi$ and an inner product $g_\varphi$ on $V$. With this choice, $\Vol_\varphi$ is also the volume form associated with $g_\varphi$. The maps $\varphi \mapsto \Vol_\varphi$ and $\varphi \mapsto g_\varphi$ are equivariant under the action of $GL_+(V)$, and in particular the stabiliser of any positive form also stabilises the associated volume form and inner product. This shows that $G_2 \subset SO(7)$. The Hodge dual $*_\varphi \varphi$ of any positive form is denoted by $\Theta(\varphi)$.

A $G_2$-structure on an oriented manifold $M$ of dimension $7$ corresponds to the choice of a $3$-form $\varphi$ such that $\varphi_x \in \Lambda^3_+ T_x^* M$ for all $x \in M$. It naturally induces a Riemannian metric $g_\varphi$, a volume form $\Vol_\varphi$ and a $4$-form $\Theta(\varphi)$ on $M$. A $G_2$-structure $\varphi$ on $M$ is called torsion-free if $\nabla_\varphi \varphi \equiv 0$ for the Levi-Civita connection $\nabla_\varphi$ of $g_\varphi$. As was first shown in \cite{fernandez1982riemannian}, this condition is equivalent to:
\begin{equation}
    d \varphi = 0 = d\Theta(\varphi).
\end{equation}
When this condition is satisfied, the holonomy of $g_\varphi$ is contained in $G_2 \subset SO(7)$. Such metrics are automatically Ricci-flat. Combined with the Cheeger--Gromoll splitting theorem \cite{cheeger1971splitting}, this implies the following:

\begin{prop}
    A compact $G_2$-manifold $(M^7,\varphi)$ has full holonomy $G_2$ if and only if $\pi_1(M)$ is finite.
\end{prop}

\paragraph{} \label{par:grep} For later use we describe some of the representations of $G_2$. Let $V$ be an oriented $7$-dimensional vector space, $\varphi$ be a positive form and identify the stabiliser of $\varphi$ with $G_2$. The representations $V$ and $V^*$ are irreducible, as $G_2$ acts transitively on the unit sphere in $V$. The representation $\Lambda^2 V^*$ is not irreducible but decomposes as:
\begin{equation*}
    \Lambda^2 V^* = \Lambda^2_7 \oplus \Lambda^2_{14}
\end{equation*}
where $\Lambda^2_{14}$ is isomorphic to the Lie algebra of $G_2$ and $\Lambda^2_7 \simeq V^*$ is its orthogonal complement. Similarly, the representation $\Lambda^3 V^*$ admits an irreducible decomposition of the form:
\begin{equation*}
\Lambda^3 V^* = \Lambda^3_1 \oplus \Lambda^3_7 \oplus \Lambda^3_{27}.
\end{equation*}
where the subscripts correspond to the dimension of each representation. As the Hodge star operator gives isomorphisms $\Lambda^k V^* \simeq \Lambda^{7-k} V^*$, we obtain a full decomposition of the exterior algebra of $V^*$.

On a $G_2$-manifold $(M,\varphi)$, the decomposition $\Lambda^k V^* \simeq \oplus \Lambda^k_m$ induces a decomposition of the space of $k$-forms $\Omega^k(M) \simeq \oplus \Omega^k_m$. With these notations, the torsion of $\varphi$ (seen as the obstruction for the Levi-Civita connection of $g_\varphi$ to be compatible with $\varphi$) can be identified with a $1$-form valued in the bundle $\Lambda^2_7M$ \cite{fernandez1982riemannian}, which we denote by $\tau(\varphi)$. More precisely, if $\nabla'$ is any connection compatible with $\varphi$, and we write the Levi-Civita connection $\nabla = \nabla' + A$ for some $1$-form $A \in \Omega^1(\Lambda^2T^*M)$, then we can define $\tau(\varphi)$ as the orthogonal projection of $A$ onto $T^*M \otimes \Lambda^2_7 M$ (which does not depend on the choice of $\nabla'$). In particular the connection $\widetilde \nabla$ defined by:
\begin{equation*}
    \widetilde \nabla = \nabla - \tau(\varphi)
\end{equation*}
is compatible with $\varphi$. This observation will be useful later, as for some purposes it is more convenient to work with a connection compatible with $\varphi$ rather than with the Levi-Civita connection when $\varphi$ is not torsion-free.

\paragraph{} Let $X$ be a complex threefold equipped with a Calabi--Yau structure $(\omega,\Omega)$, where $\omega \in \Omega^2(X)$ is the K\"ahler form and $\Omega \in \Omega^3_\C(X)$ is the holomorphic volume form. Then $Y = \R \times X$ is naturally equipped with the $G_2$-structure:
\begin{equation}
    \varphi = dt \wedge \omega + \re(\Omega).
\end{equation}
The associated $G_2$-metric is the product metric $g_\varphi = dt^2 + g_X$, where $g_X$ is the Calabi--Yau metric on $X$. The Hodge dual of $\varphi$ takes the form:
\begin{equation}
    \Theta(\varphi) = \frac 1 2 \omega^2 - dt \wedge \im \Omega.
\end{equation}
When $X$ is compact, $(Y,\varphi)$ is called a $G_2$-cylinder. Any translation-invariant $G_2$-structure on $Y$ is obtained in this way. 

A non-compact $G_2$-manifold $(Z,\varphi)$ is called an EAC $G_2$-manifold of rate $\mu > 0$ if $g_\varphi$ is an EAC metric of rate $\mu$ on $Z$, and there exists a translation-invariant $G_2$-structure $\varphi_0$ on the asymptotic cylinder $Y = \R \times X$ such that $\varphi-\varphi_0$ and all their derivatives are $O(e^{-\mu \rho})$ at infinity, where $\rho$ is a cylindrical coordinate function on $Z$.

    \subsection{The twisted connected sum construction}     \label{subsection:tcs}

\paragraph{} All known constructions  of compact manifolds with holonomy $G_2$ are based on an argument of Joyce \cite[\S10.3]{joyce2000compact}, which we outline here. The starting point is to consider a compact manifold $M^7$ equipped with a closed $G_2$-structure $\varphi$ with appropriately small torsion, and look for a nearby torsion-free $G_2$-structure $\widetilde \varphi = \varphi + d \eta$ in the same cohomology class. As the condition defining $G_2$-structures is open, there exists a universal constant $\epsilon_0$ such that if $\xi \in \Omega^3(M)$ satisfies $\| \xi \|_{C^0} \leq \epsilon_0$, then $\varphi + \xi$ is also a $G_2$-structure. Hence $\Theta(\varphi + \xi)$ is well-defined and can be written as:
\begin{equation*}
    \Theta(\varphi +\xi) = \Theta(\varphi) + L_\varphi \xi + F(\xi)
\end{equation*}
where $L_\varphi$ is the linearisation of $\Theta$ at $\varphi$ and $F$ a smooth function defined on a ball of radius $\epsilon_0$ in $\Lambda^3T^*M$. In particular $F$ satisfies a bound of the form:
\begin{equation*}
    |F(\xi)-F(\xi^\prime)| \leq C |\xi-\xi^\prime|(|\xi|+|\xi^\prime|), ~~ |\xi|,|\xi^\prime| \leq \epsilon_0
\end{equation*}
for some universal constant $C$ \cite[Proposition 10.3.5]{joyce2000compact}. With these notations, there exists a universal constant $\epsilon_1$, which does not depend on $M$ or $\varphi$, such that the following holds \cite[Theorem 10.3.7]{joyce2000compact}:

\begin{thm}     \label{thm:torsionfree}
    Let $(M,\varphi)$ be a compact manifold equipped with a closed $G_2$-structure. Suppose $\eta$ is a $2$-form on $M$ such that $\| d\eta \|_{C^0} \leq \epsilon_1$ and $\psi$ a $4$-form on $M$ such that $d\Theta(\varphi) = *d\psi$ and $\| \psi \|_{C^0} \leq \epsilon_1$. If $(\eta,\psi)$ satisfy:
    \begin{equation*}
        \Delta \eta + d \left( \left(1+ \frac 1 3 \langle d\eta, \varphi \rangle \right) \psi \right) + *dF(d\eta) = 0
    \end{equation*}
    then $\widetilde \varphi = \varphi + d\eta$ is a torsion-free $G_2$-structure on $M$.
\end{thm}

Therefore, one can start with a compact manifold equipped with a closed $G_2$-structure with small torsion and apply a fixed point theorem in order to deform it to a nearby torsion-free $G_2$-structure in the same cohomology class. This involves a good understanding of the linearised problem, as well as some control on the non-linear map $F$ in the form of quadratic estimates. 

\paragraph{} Let us now outline the twisted connected sum construction. The building blocks are a pair of EAC $G_2$-manifolds $(Z_1,\varphi_1)$ and $(Z_2,\varphi_2)$ of rate $\mu > 0$, asymptotic to a cylinder $\R \times X$. Denote by $\varphi_{0,i}$ the asymptotic translation-invariant model for $\varphi_i$. The $G_2$-structures $\varphi_1$ and $\varphi_2$ are said to be matching if there exists an isometry $\gamma$ of the cross-section $X$ such that the map
\begin{equation*}
   \underline \gamma : \R \times X \rightarrow \R \times X, ~~ (t,x) \mapsto (-t, \gamma(x))
\end{equation*}
satisfies $\underline \gamma^*\varphi_{0,2} = \varphi_{0,1}$. If $\gamma_{0,i} = dt \wedge \omega_{0,i} + \re \Omega_{0,i}$ for Calabi-Yau structures $(\omega_{0,i},\Omega_{0,i})$ on $X$, the matching condition amounts to:
\begin{equation*}
    \gamma^* \omega_{0,2} = - \omega_{0,1}, ~~ \gamma^* \re \Omega_{0,2} = \re \Omega_{0,1}.
\end{equation*}
In all known examples \cite{kovalev2003twisted,corti2015g,nordstrom2018extra}, such matching pairs are trivial circle bundles over EAC Calabi-Yau threefolds (or some quotient of it), and the cross-section is isometric to the product of a K3 surface with a flat $2$-torus (or a corresponding quotient). Moreover, the map $\gamma$ is designed so that the family of compact manifolds $M_T$ obtained by gluing $Z_1$ and $Z_2$ along $\gamma$ (see Remark \ref{rem:twist}) has finite fundamental group, in order to construct manifolds with full holonomy $G_2$. The details of the construction of such matching pairs go beyond the scope of the present paper and do not affect our analysis, so we refer to the original references for more details.

Let us denote by $\sigma$ the minimum of $\mu$ and of the square roots of the smallest non-trivial eigenvalues of the Laplacian acting on $2$- and $3$-forms on $X$. As we have seen in \ref{par:closedforms}, the closed forms $\varphi_i$ and $\Theta(\varphi)$ admit an expansion:
\begin{equation*}
    \varphi_i = \varphi_{i,0} + d \eta_i, ~~~ \Theta(\varphi_i) = \Theta(\varphi_{i,0}) + d\xi_i
\end{equation*}
where $\eta_i \in \Omega^2(Z_i)$, $\xi_i \in \Omega^3(Z_i)$ and all their covariant derivatives have exponential decay in $O(e^{-\delta \rho_i})$, where $\rho_i$ are cylindrical coordinate functions on $Z_i$ and $0 < \delta < \sigma$. Pick a cutoff function $\chi : \R \rightarrow [0,1]$ such that $\chi \equiv 0$ in $(-\infty,-\frac 1 2]$ and $\chi \equiv 1$ in $[\frac 1 2, \infty)$. We can build $1$-parameter families of closed forms:
\begin{equation*}
    \varphi_{i,T} = \varphi_i - d(\chi(\rho_i-T-1) \eta_i), ~~~ \Theta_{i,T} = \Theta(\varphi_i)-  d(\chi(\rho_i-T-1) \xi_i)
\end{equation*}
For $T$ large enough $\varphi_{i,T}$ is a $G_2$-structure, which is closed by construction. Patching up $\varphi_{1,T}$ with $\varphi_{2,T}$ and $\Theta_{1,T}$ with $\Theta_{2,T}$ as in \ref{par:closedforms}, the $1$-parameter family of compact manifolds $M_T$ obtained by gluing $Z_1$ and $Z_2$ along $\gamma$ is endowed with a family of closed $G_2$-structures $\varphi_T$ and a family of closed $4$-forms $\Theta_T$. Let us write $\psi_T = \Theta(\varphi_T) -\Theta_T$, so that $d\psi_T = d\Theta(\varphi_T)$. By construction, we have estimates of the form:
\begin{equation}        \label{eq:estpsi}
    \|\psi_T\|_{C^k} = O \left(e^{-\delta T} \right)
\end{equation}
for any $k \geq 0$ and $0 < \delta < \sigma$, and similarly with Sobolev norms \cite[Lemma 4.25]{kovalev2003twisted}. 

For $T$ large enough we are in the correct setup to apply Theorem \ref{thm:torsionfree} an seek solutions $\eta \in \Omega^2(M_T)$ with $\| d\eta \|_{C^0} \leq \epsilon_1$ solving the equation:
\begin{equation}        \label{eq:ttorsionfree}
    \Delta_T \eta + * d \left( \left(1+ \frac 1 3 \langle d\eta, \varphi_T \rangle \right) \psi_T \right) + *dF(d\eta) = 0 .
\end{equation}
It follows from \cite[Theorem 11.6.1]{joyce2000compact} that for $T$ large enough this equation admits a solution $\eta_T$, so that $\widetilde \varphi_T = \varphi_T + d\eta_T$ is a torsion-free $G_2$-structure on $M_T$. However, the proof of \cite[Theorem 5.34]{kovalev2003twisted} that there are estimates of the form $\|\widetilde \varphi_T - \varphi_T \|_{C^0} = O(e^{-\delta T})$ seems to be incorrect as the asymptotic kernel considered for the linearised problem does not have the dimension of the space of harmonic $2$-forms on $M_T$. Rather, its dimension is the sum of the dimensions of the kernel of the Laplacian acting on decaying $2$-forms on $Z_1$ and of the kernel of the Laplacian acting on $2$-forms with less than exponential growth on $Z_2$. As the cross-section has non-zero first and second Betti numbers this asymptotic kernel is never trivial, whereas there are examples of twisted connected sums with $b_2(M_T) = 0$. Using Theorem \ref{thm:main} and Corollary \ref{cor:laplacian} we can fix this issue. Moreover, we obtain a control on the norms of all the derivatives of $\widetilde \varphi_T - \varphi_T$ in $O(e^{-\delta T})$, for any $0 < \delta < \sigma$. Note that if $\delta > 0$ is small enough, one can also apply the general result of Joyce to get $C^0$-estimates in $O(e^{- \delta T})$, but it does not give a control on the derivatives.

\paragraph{} Let us now explain how to set up a fixed-point argument to prove Corollary \ref{cor:tcsestimates}. For $T$ large enough let us consider the differential operator $P_T : C^\infty(\Lambda^2T^*M_T) \rightarrow C^\infty(\Lambda^2T^*M_T)$ defined by:
\begin{equation}
    P_T \eta = \Delta_T \eta + \frac 1 3 * d ( \langle d \eta, \varphi_T \rangle \psi_T).
\end{equation}
The image of $P_T$ is orthogonal to the space of harmonic $2$-forms. The decay of $\psi_T$ and all its derivatives implies that, for any $k \geq 0$ and $p > 1$, there is a bound of the form $\|(P_T-\Delta_T) \eta \|_{W^{k,p}} = O(e^{-\delta T} \|\eta\|_{W^{2+k,p}})$. Thus it follows from Corollary \ref{cor:laplacian} that for $T$ large enough, the map:
\begin{equation}        \label{eq:defpt}
    \mH^2(M_T)^\perp \cap W^{2+k,p}(\Lambda^2 T^*M_T) \rightarrow \mH^2(M_T)^\perp \cap W^{k,p}(\Lambda^2 T^*M_T)
\end{equation}
induced by $P_T$ admits a bounded inverse $Q_T$, with norm satisfying
\begin{equation}    \label{eq:estqt}
    \|Q_T\| \leq C T^2
\end{equation}
for some constant $C > 0$. On the other hand, we also have a control:
\begin{equation}        \label{eq:estpsiw}
    \|*d\psi_T \|_{W^{k,p}} \leq C^\prime e^{- \delta T }
\end{equation}
for any $\delta < \sigma$ from \eqref{eq:estpsi}. 

It remains to obtain a good control on the non-linear part of \eqref{eq:ttorsionfree}. The key result for our purpose are the following quadratic estimates:

\begin{prop}     \label{prop:quadratic}
    Let $p \geq 7$ and let $k \geq 1$ be an integer. There exist constants $\epsilon_{k,p} > 0$ and $C_{k,p} > 0$ such that for $T$ large enough, if $\xi, \xi' \in W^{k,p}(\Lambda^3T^*M_T)$ satisfy the condition $\|\xi\|_{W^{k,p}}, \|\xi^\prime\|_{W^{k,p}} \leq \epsilon_{k,p}$ then:
    \begin{equation*}
        \|F(\xi)-F(\xi^\prime)\|_{W^{k,p}} \leq C_{k,p} \| \xi-\xi^\prime\|_{W^{k,p}} (\|\xi\|_{W^{k,p}} + \|\xi^\prime\|_{W^{k,p}} ) .
    \end{equation*}
\end{prop}

\begin{rem}
    This proposition contains the fact that $F(\xi)$ is well-defined in a neighbourhood of $0$ in $W^{k,p}$. Essentially we need $p \geq 7$ and $k \geq 1$ in order to have a Sobolev embedding $W^{k,p}(\Lambda^3 T^*M_T) \hookrightarrow C^{k-1}(\Lambda^3T^*M_T)$.
\end{rem}

We shall prove this proposition in the next part. Once this is established, Corollary \ref{cor:tcsestimates} follows by applying a contraction-mapping argument to the map  $\eta \mapsto P_T \eta + * dF(d\eta)$, defined on an appropriately small ball in $W^{2+k,p}(\Lambda^2 T^*M_T) \cap \mH^2(M_T)^\perp$, for some choice of $p \geq 7$ and $k \geq 1$. By \eqref{eq:estqt} and \eqref{eq:estpsiw}, for small enough $\epsilon > 0$ and large enough $T$, equation \eqref{eq:ttorsionfree} admits a unique solution $\eta_T$ in a ball of radius $\frac{\epsilon}{T^2}$ centered at $0$, which moreover satisfies $\| \eta_T \|_{W^{2+k,p}} = O(e^{-\delta T})$ for any $0 < \delta < \sigma$. As the norm of the Sobolev embedding $W^{1+k,p} \hookrightarrow C^k$ is uniformly bounded (Remark \ref{rem:sobemb}) this implies that $\widetilde \varphi_T = \varphi_T + d\eta_T$ satisfies $\| \widetilde \varphi_T - \varphi_T \|_{C^k} = O(e^{-\delta T})$ as $T \rightarrow \infty$.

    \subsection{Proof of the quadratic estimates}       \label{subsection:quadratic}

\paragraph{} In order to prove Proposition \ref{prop:quadratic} it is useful to work at some level of generality. Essentially, what we need is to improve \cite[Proposition 10.3.5]{joyce2000compact} in order to control higher order derivatives. The key property which allows us to obtain good estimates on the derivatives of $F$ is the equivariance of the map $\Theta$ under the action of $GL_+(7,\R)$. Thus we consider the following setting. Let $M$ be an oriented $n$-dimensional manifold, which need not be compact or complete as we will mostly work locally. Let us denote by $\mathscr{P}$ the oriented frame bundle of $M$, considered as a $GL_+(n,\R)$-bundle. Let $G \subset SO(n)$ be a compact subgroup, and $\mathscr Q \subset \mathscr P$ a $G$-structure inducing a Riemannian metric $g$ on $M$. We implicitly endow all representations of $G$ with an invariant inner product, and consider the corresponding metrics on associated bundles.

Let $(\rho_0,V_0)$, $(\rho_1,V_1)$ be representations of $GL_+(n,\R)$ and let $E_i = \mathscr P \times_{\rho_i} V_i$ be the associated bundles. Let $O_0$ be an open subset of $V_0$ stable under $\rho_0$, and let $\Upsilon : O_0 \rightarrow V_1$ be a smooth map which is equivariant under the action of $GL_+(n,\R)$, that is:
\begin{equation}        \label{eq:equivariance}
    \Upsilon(\rho_0(g) u) = \rho_1(g)\Upsilon(u), ~~ \forall g \in GL_+(n,\R).
\end{equation}
Consider $U_0 = \mathscr P \times_{\rho_0} O_0$ as an open subbundle of $E_0$. Then $\Upsilon$ induces a bundle map $U_0 \rightarrow V_1$, which we still denote by $\Upsilon$. Let us remark that the differential map:
\begin{equation*}
    D \Upsilon : O_0 \rightarrow \End(V_0,V_1)
\end{equation*}
is also $G$-equivariant. Indeed, if we differentiate \eqref{eq:equivariance} with respect to $u$ we obtain:
\begin{equation*}
    D_{\rho_0(g)u}\Upsilon(\rho_0(g) \dot u) = \rho_1(g) D_u \Upsilon(\dot u), ~~ \forall g \in GL_+(n,\R) ~ \text{and} ~ \forall \dot u \in V_0. 
\end{equation*}
Thus we have a family of $G$-equivariant maps
\begin{equation*}
    D^k\Upsilon : O_0 \rightarrow V_0^* \otimes \cdots \otimes V_0^* \otimes V_1
\end{equation*}
and induced maps on the corresponding bundles. 

Let $\nabla$ be a connection on $M$ which is compatible with $\mathscr Q$. In particular, we can find trivialisations of $\mathscr P$ with transition maps valued in $G$ and local connection forms valued in its Lie algebra $\mathfrak g$. If $u$ is a local section of $U_0$, we can compute in local trivialisations:
\begin{align*}
    \nabla \Upsilon(u) & = d\Upsilon(u) + d\rho_1(A) \Upsilon(u) \\
    & = D_u \Upsilon(d u) + D_u\Upsilon(d\rho_0 (A) u) \\
    & = D_u\Upsilon(\nabla u)
\end{align*}
where $A$ is the local connection form. To pass from the first to the second line we differentiated \eqref{eq:equivariance} with respect to $g$ to obtain the identity:
\begin{equation*}
    d\rho_1(a) \Upsilon(u) = D_u\Upsilon(d\rho_0(a) u), ~~ \forall a \in \mathfrak {gl}_n ~ \text{and} ~ u \in O_0 .
\end{equation*}
Given the equivariance of the maps $D^l\Upsilon$, we can deduce by iteration that:
\begin{equation*}
    \nabla^k \Upsilon(u) = \sum_{l=1}^k  ~ \sum_{j_1+\cdots+j_l = k, ~ j_i \geq 1} C_{j_1 \dots j_l} D^l_u \Upsilon(\nabla^{j_1}u,...,\nabla^{j_l}u)
\end{equation*}
where $C_{j_1 \ldots j_l}$ are universal combinatorial coefficients. 

Suppose now that $u_0 \in O_0$ is invariant under the action of $G$. Then, it induces a section of $U_0$ which is parallel for $\nabla$ and we have:
\begin{equation*}
    \nabla u_0 = \nabla \Upsilon(u_0) = \nabla D_{u_0}\Upsilon  = 0.
\end{equation*}
As $O_0$ is open in $V_0$, there exists a universal constant $\epsilon_0 > 0$ such that if $u$ is a section of $E_0$ with $\|u\|_{C^0} \leq \epsilon_0$, then $u_0 + u$ is a section of $U_0$. Here the $C^0$-norm is measured with respect to the bundle metric induced by $\mathscr Q$.  Thus if $\|u\|_{C^0} \leq \epsilon_0$ we may write:
\begin{multline*}
    \nabla^k(\Upsilon(u_0+u) - \Upsilon(u_0) - D_{u_0} \Upsilon(u)) = (D_{u_0+u}\Upsilon- D_{u_0} \Upsilon)(\nabla^k u) \\ + \sum_{l=2}^k \sum_{j_1, \ldots ,j_l} C_{j_1 \ldots j_l} D^l_{u_0+u} \Upsilon(\nabla^{j_1}u, \ldots ,\nabla^{j_l}u).
\end{multline*}
In local trivialisations, we may consider $D^l_{u_0+u}\Upsilon$ as fixed maps defined on a small ball of radius $\epsilon_0$ in $V_0$. Let $k \in \N$ and pick some $0 < \epsilon_k < \epsilon_0$. If $u,v$ are contained in the closed ball of radius $\epsilon_k$ in $V_0$, there are estimates of the form:
\begin{equation}        \label{ineq:dlupsi}
    |D^l_{u_0+u} \Upsilon | \leq C_l , ~~   |D^l_{u_0+u}\Upsilon - D^l_{u_0+v} \Upsilon | \leq C^\prime_l |u-v|, ~~~ \forall 0 \leq l \leq k
\end{equation}
for some constants $C_l, C^\prime_l$ depending only on the choice of $\epsilon_k$.

\paragraph{} Let us now write $R(u) = \Upsilon(u_0+u) - \Upsilon(u_0) - D_{u_0} \Upsilon(u)$, defined on a ball of radius $\epsilon_0$ in $V_0$. If $u,v$ are sections of $E_0$ with $C^0$-norm less than $\epsilon_0$, we want to bound the quantity $|\nabla^k (R(u) - R(v))|$. The triangular inequality yields:
\begin{multline*}
    |(D_{u_0+u}\Upsilon- D_{u_0} \Upsilon)(\nabla^k u)- (D_{u_0+v}\Upsilon- D_{u_0} \Upsilon)(\nabla^k v)|  \leq |(D_{u_0+u} \Upsilon - D_{u_0+v} \Upsilon)(\nabla^k u)| \\ + |(D_{u_0+v} \Upsilon-D_{u_0}\Upsilon)(\nabla^k u-\nabla^k v)| .
\end{multline*}
Assuming that $u,v$ moreover satisfy $\|u\|_{C^0}, \| v \|_{C^0} \leq \epsilon_k$ and using \eqref{ineq:dlupsi} we obtain:
\begin{multline}        \label{eq:nablak}
    |(D_{u_0+u}\Upsilon- D_{u_0} \Upsilon)(\nabla^k u)- (D_{u_0+v}\Upsilon- D_{u_0} \Upsilon)(\nabla^k v)| \leq C (|u-v| \cdot |\nabla^k u|) \\ + C^\prime(|v| \cdot |\nabla^k u-\nabla^k v|)
\end{multline}
for some universal constants which depend only on the choice of $\epsilon_k$. Similarly, for $2 \leq l \leq k$ we have universal estimates of the form:
\begin{multline*}
    | D^l_{u_0+u} \Upsilon(\nabla^{j_1}u, \ldots ,\nabla^{j_l}u) - D^l_{u_0+v} \Upsilon(\nabla^{j_1}v, \ldots ,\nabla^{j_l}v)| \leq C (|u-v| \cdot |\nabla^{j_1} u | \cdots | \nabla^{j_l} u |) \\ + C_1 |\nabla^{j_1}u-\nabla^{j_1} v | \cdot |\nabla^{j_2}u | \cdots |\nabla^{j_l}u| + \cdots + C_l |\nabla^{j_1}v | \cdots |\nabla^{j_{l-1}} v| \cdot|\nabla^{j_l}u-\nabla^{j_l} v | .
\end{multline*}
Each term in the above sum only contains factors of $|\nabla^{j} u|$ or $|\nabla^j v|$ with $1 \leq j \leq k-1$, so that the only terms containing factors of $|\nabla^k u|$, $|\nabla^kv|$ or $|\nabla^ku-\nabla^k v|$ come from \eqref{eq:nablak}. Thus, if we impose not only $\|u\|_{C^0}, \|v\|_{C^0} \leq \epsilon_k$ but also $\sum_{0 \leq l \leq k-1} \|\nabla^l u \|_{C^0} \leq \epsilon_k$ and similarly for $v$, we obtain an estimate:
\begin{multline*}
    |\nabla^k (R(u) - R(v))| \leq C |u-v| \cdot (|\nabla^k u| + |\nabla^k v|) + C^\prime |\nabla^ku-\nabla^k v| \cdot (|u|+|v|) \\ + C^{\prime\prime} \sum_{0 \leq l,l^\prime \leq k- 1} |\nabla^l u - \nabla^l v| \cdot (|\nabla^{l^\prime} u| + |\nabla^{l^\prime} v |)
\end{multline*}
for some universal constants $C,C^\prime,C^{\prime\prime}$ depending only on our choice of $\epsilon_k$. 

Globally, this implies that if $u, v$ are sections of $E_0$ satisfying $\|u\|_{C^{k-1}}, \| v \|_{C^{k-1}} \leq \epsilon_k$, and $u,v$ are in $W^{k,p}$ for some $p > 1$, we have:
\begin{multline*}
    \|\nabla^k(R(u)-R(v))\|_{L^p} \leq C \|u-v\|_{C^{k-1}} \cdot (\|u\|_{W^{k,p}} + \|v\|_{W^{k,p}}) \\ + C^\prime (\|u\|_{C^{k-1}}+\|v\|_{C^{k-1}}) \|u-v\|_{W^{k,p}}
\end{multline*}
for some uniform constants. In particular if we are on a compact manifold and $p \geq n$, the $W^{k,p}$-norms controls the $C^{k-1}$ norm. Thus, there exists an $\epsilon_{k,p} > 0$ such that if $\|u\|_{W^{k,p}} \leq \epsilon_{k,p}$ then $\|u\|_{C^{k-1}} \leq \epsilon_k$, so that $F(u)$ is well-defined and in $W^{k,p}$. If moreover $v$ is another section satisfying $\|v\|_{W^{k,p}} \leq \epsilon_{k,p}$, then we have an estimate:
\begin{equation*}
    \|R(u)-R(v)\|_{W^{k,p}} \leq C_{k,p} \| u-v \|_{W^{k,p}} (\|u\|_{W^{k,p}} + \|v\|_{W^{k,p}}).
\end{equation*}
for some constant $C_{k,p}$. Note however that $\epsilon_{k,p}$ and $C_{k,p}$ are not universal, as they depend on the manifold $(M,g)$ through the constant in the Sobolev embedding $W^{k,p} \hookrightarrow C^{k-1}$. Moreover, it is important to note that to obtain such  estimates we implicitly redefine the $W^{k,p}$ and $C^k$-norms using the compatible connection $\nabla$, and not the Levi-Civita connection of the metric induced by $\mathscr Q$ on $M$. This yields equivalent definitions, but to use the above estimates we also need to take into account the extra constants coming from the fact that we are dealing with a different definition of the usual norms.

\paragraph{} We want to apply this reasoning to the family of compact manifolds $M_T$ equipped with the $G_2$-structures $\varphi_T$, where we consider the equivariant map $\Theta$. In this case, the norm of the Sobolev embedding $W^{k,p} \hookrightarrow C^{k-1}$ can be bounded independently of $T \geq 1$ (see Remark \ref{rem:sobemb}). However, we cannot directly apply the above argument to the Levi-Civita connection $\nabla_T$ of $g_T$ as it is not compatible with $\varphi_T$. Nevertheless, we have seen in \ref{par:grep} that the torsion of $\varphi_T$, which is represented here by $d\Theta(\varphi_T)$, can be identified with a $1$-form $\tau(\varphi_T)$ valued in the bundle $\Lambda^2_7M_T$. Moreover the connection $\widetilde \nabla_T = \nabla_T - \tau(\varphi_T)$ is compatible with $\varphi_T$. The torsion $\tau(\varphi_T)$ satisfies
\begin{equation*}
    \left\|\nabla_T^l \tau(\varphi_T) \right\|_{C^0} = O \left(e^{-\delta T} \right)
\end{equation*}
for $l \geq 0$ and any small enough $\delta > 0$. Hence it follows that for $T$ large enough, the $W^{k,p}$ and $C^{k-1}$-norms defined with respect to $\nabla_T$ or $\widetilde \nabla_T$ are equivalent up to a constant of the form $1+O(e^{-\delta T})$, so that for all the norms that we need to consider we can work with $\widetilde \nabla_T$ instead of $\nabla_T$. Thus Proposition \ref{prop:quadratic} follows by applying the above argument.

\newpage

\addcontentsline{toc}{section}{References}

\small

\bibliographystyle{abbrv}
\bibliography{ref_neck-stretching}

\end{document}